\documentclass[11pt]{article}
\usepackage{amsmath, amssymb, theorem, latexsym, epsfig}
\numberwithin{equation}{section}

\theoremstyle{plain}
\theorembodyfont{\itshape}
\newtheorem{theorem}{Theorem}[section]
\newtheorem{proposition}[theorem]{Proposition}
\newtheorem{lemma}[theorem]{Lemma}
\newtheorem{corollary}[theorem]{Corollary}

\theorembodyfont{\rmfamily}
\newtheorem{definition}[theorem]{Definition}

\newtheorem{example}[theorem]{Example}
\newtheorem{remark}[theorem]{Remark}
\newtheorem{convention}[theorem]{Convention}

\newenvironment{proof}{{\noindent \textbf{Proof}\,\,}}{\hspace*{\fill}$\Box$\medskip}

\newcommand{\Sing}{\operatorname{Sing}}

\def\cc{\mathbb C}
\def\oc{\overline{\cc}}

\def\cp{\mathbb{CP}}
\def\wt#1{\widetilde#1}
\def\rr{\mathbb R}
\def\var{\varepsilon}

\def\nn{\mathbb N}

 \def\mcb{\mathcal B}
 \def\zz{\mathbb Z}
 \def\rp{\mathbb{RP}}
 \def\mcb{\mathcal B}
 \def\Sing{\operatorname{Sing}}
\def\La{\Lambda}
\def\ii{\mathbb I}
\def\mca{\mathcal A}
  \def\diag{\operatorname{diag}}
 \def\la{\lambda}
\def\mcc{\mathcal C}

\def\kla{\kappa_{\la}}
\def\mcm{\mathcal M}
\def\mcq{\mathcal Q}
\def\mcg{\mathcal G}

\def\Chi{\mathcal X}
\def\mcd{\mathcal D}
\def\Xi{\mathcal Z}
\def\mcl{\mathcal L}
\def\deg{\operatorname{deg}}

\def\ker{\operatorname{Ker}}

\title{On polynomially integrable Birkhoff billiards on surfaces of constant curvature}
\author{Alexey Glutsyuk\thanks{ CNRS, France (UMR 5669 (UMPA, ENS de Lyon) and UMI 2615 (Interdisciplinary Scientific Center J.-V.Poncelet)), 
Lyon, France. 
E-mail:
aglutsyu@ens-lyon.fr}
\thanks{National Research University Higher School of Economics, Russian Federation}
 \thanks{Supported by part by RFBR grants 13-01-00969-a, 16-01-00748, 16-01-00766
 and ANR grant ANR-13-JS01-0010.}}
\begin{document}
\maketitle
\begin{abstract} The polynomial version of Birkhoff Conjecture on integrable billiards on complete simply connected surfaces of constant curvature (plane, sphere, hyperbolic plane) was first stated, studied and solved in a particular case by Sergei Bolotin 
in 1990-1992. Here we present a complete solution of the polynomial version of Birkhoff Conjecture. 
Namely we show that {\it every polynomially integrable real bounded planar billiard with $C^2$-smooth connected 
boundary is an ellipse.} 
We  extend this result to  billiards with piecewise-smooth and not necessarily convex 
boundary on 
arbitrary two-dimensional simply connected complete surface of constant curvature: plane, sphere, 
Lobachevsky--Poincar\'e (hyperbolic) plane; each of them 
being modeled as a plane or a (pseudo-) sphere in $\rr^3$ equipped with appropriate quadratic form. Namely,  we show that 
{\it a billiard is polynomially integrable, if and only if its boundary is a union of confocal conical arcs and appropriate 
geodesic segments.}  We also present a complexification of these results. These are joint results of 
Mikhail Bialy, Andrey Mironov and the author. The proof  is split into two parts. The first part is given in two  papers by  Bialy and Mironov 
(in Euclidean and non-Euclidean cases respectively). Their geometric construction reduced the Polynomial Birkhoff Conjecture to a purely algebro-geometric problem  to show that an irreducible algebraic curve in $\cp^2$ with 
certain properties is a conic. They have shown that its singular and inflection points  lie in the complex light conic 
of the above-mentioned quadratic form. In the present paper we solve the above algebro-geometric problem completely. 
\end{abstract}
\tableofcontents

 \section{Introduction}
 
 \subsection{Main results}
 The famous Birkhoff Conjecture deals with strictly convex bounded planar billiards with smooth boundary. 
 Recall that a {\it caustic} of a planar billiard $\Omega\subset\rr^2$ 
  is a curve $C$ 
 such that each tangent line to $C$ reflects from the boundary of the billiard to a line tangent to $C$. 
 A billiard $\Omega$  is called {\it Birkhoff caustic-integrable,} if a neighborhood of its boundary in $\Omega$ is foliated by closed caustics, and the boundary $\partial\Omega$ is a leaf of this foliation. 
 It is well-known that each elliptic billiard is integrable, see \cite[section 4]{tab}. 
 The {\bf Birkhoff Conjecture}  states the converse: {\it the only Birkhoff caustic-integrable 
 convex bounded planar 
 billiard with smooth boundary is an ellipse.}\footnote{This  conjecture,  classically  attributed  to 
 G.Birkhoff,  was published   in  print  only  in the paper \cite{poritsky} by  H. Poritsky, who worked with Birkhoff as a post-doctoral fellow 
 in late 1920-ths.} 
 
  Let now $\Sigma$ be a two-dimensional surface with a Riemannian metric, $\Omega\subset\Sigma$ be a connected domain\footnote{Everywhere in the paper a billiard is a {\bf connected domain} $\Omega\subset\Sigma$.} with 
  piecewise smooth boundary. The {\it billiard flow} $B_t$ acts on the  tangent bundle $T\Sigma|_{\Omega}$ as follows. A point $(Q,P)\in T\Sigma|_{\Omega}$, 
  $Q\in\Omega$, $P\in  T_Q\Sigma$ moves along a trajectory of the geodesic flow of the surface $\Sigma$ until 
  $Q$ hits the boundary $\partial\Omega$. While hitting the boundary, the point $Q$ remains unchanged, and the 
  velocity vector $P$ is reflected from the boundary to the vector $P^*$  according to the usual reflection law: the 
 angle of incidence equals to the angle of reflection; $|P|=|P^*|$. Afterwards the new point $(Q,P^*)$ again 
 moves along a trajectory of the geodesic flow etc. The billiard flow thus defined, which 
 can be viewed as a geodesic flow with impacts on $T\Sigma|_{\Omega}$, has an obvious first integral: 
 the absolute value $|P|$ of the velocity. A strictly convex billiard $\Omega$ with smooth boundary 
 is called {\it integrable in the Liouville sense,} if its flow has an additional first integral independent with $|P|$ on the 
 intersection with $T\Sigma|_{\overline\Omega}$ of a neighborhood of the 
 unit tangent bundle to the boundary. 
  
 The notions of a caustic and Birkhoff caustic-integrability extend to the case of a strictly 
 convex domain $\Omega$ on an  arbitrary surface $\Sigma$ 
 equipped with a Riemannian metric, with lines replaced by geodesics. 
  The {\it Liouville and Birkhoff caustic integrabilities   are equivalent:} it is a well-known folklore fact.

 There is an analogue of the Birkhoff Conjecture 
 for billiards  on a simply connected complete Riemannian surface  of non-zero constant curvature: sphere or hyperbolic 
 (Lobachevsky--Poincar\'e) plane. This is also an open problem.

 The particular case of the Birkhoff Conjecture, when 
 the additional first integral is supposed to be  polynomial in the velocity components, motivated  the next definition and conjecture.
 
 \begin{definition} \label{din} Let $\Sigma$ be a two-dimensional surface with Riemannian metric, and let 
 $\Omega\subset\Sigma$ be a domain with piecewise smooth boundary. We say that the billiard in $\Omega$ is 
 {\it polynomially integrable,} if its flow has a first integral on 
 $T\Sigma|_{\Omega}$ that is a polynomial  in the velocity $P$ and whose restriction to the  hypersurface $\{|P|=1\}$ is non-constant.
 \end{definition}

\begin{definition} \label{defanaly} Let $\Sigma$ be as above, and  let $\Omega\subset\Sigma$ be a 
domain with piecewise smooth boundary. We say that 
 $\Omega$ is {\it analytically integrable,} if there exists a first integral analytic in $P$ on a neighborhood in 
 $T\Sigma|_{\Omega}$ of the zero  section of the tangent bundle $T\Sigma|_{\Omega}$  
 that is not a function of just the modulus $|P|$. In addition, it is required that  there exists a $r>0$ such that 
 the integral  is defined for all $(Q,P)$ with $Q\in\Omega$ and $|P|\leq r$ and 
its Taylor series in $P$  converges uniformly in the above $(Q,P)$. 
 \end{definition}

Note that all 
the integrals under question, which are defined over an open domain $\Omega$,  
should be invariant under  the geodesic flow in $\Omega$ and under the reflections from 
its boundary.

 \begin{remark} The following facts are well-known:
 
 - Analytic integrability implies polynomial integrability, since each homogeneous part in $P$ of an analytic integral 
 is a first integral itself, see \cite[p. 107]{kd} (the converse is obvious);
 
 - In the case, when $\Sigma$ is a simply connected {\it  complete  surface  of constant curvature} and the boundary 
 $\partial\Omega$ is smooth and connected, polynomial integrability is equivalent to the existence of a polynomial integral as above in a neighborhood of 
 the unit tangent bundle to $\partial\Omega$ in $T\Sigma|_{\Omega}$, by S.V.Bolotin's results \cite{bolotin0, bolotin, bolotin2}, see Theorem \ref{algnew} below. In this case 
  each first integral that is just polynomial in $P$ is  {\it globally analytic} on 
 $T\Sigma$, see \cite[proof of proposition 2]{bolotin2} and Theorem \ref{algnew}. 
  \end{remark}
 
 The {\bf Polynomial Birkhoff Conjecture} states that {\it if a  convex planar billiard with smooth boundary  is polynomially 
integrable, then its boundary  is a conic.}  The  Polynomial Birkhoff Conjecture together with its 
generalization to billiards with piecewise smooth (may be  non-convex) 
boundaries on simply connected complete surfaces 
of constant curvature was first stated and studied by S.V.Bolotin  \cite{bolotin, bolotin2} and later studied  in joint papers of M.Bialy and A.E.Mironov \cite{bm, bm2, bm3}. In the present paper we give a complete solution of the Polynomial Birkhoff Conjecture in full generality (Theorems \ref{thal0} and \ref{thal}). 
 
 \begin{remark} \label{algan}  The Polynomial Birkhoff Conjecture and its generalization are important and  interesting themselves, {\it independently on a potential solution of the classical Birkhoff Conjecture.} 
 They lie on the crossing of  different domains of mathematics, first of all, 
 dynamical systems,  algebraic geometry and singularity theory. {\it They 
are not implied by the classical Birkhoff Conjecture.} For the general case of piecewise-smooth boundaries this is obvious. 
Even in the case of smooth convex boundary, while  the polynomiality condition is a very strong restriction, the condition of 
 just non-constance of a polynomial integral on the unit velocity level hypersurface is topologically 
 weaker than the independence condition in the 
  Liouville integrability, which requires  independence of the additional integral and the energy on a whole neighborhood in 
  $T\rr^2|_{\overline\Omega}$ of the 
 unit tangent bundle to the boundary.
 \end{remark}

 Without loss of generality we consider simply connected complete  surfaces $\Sigma$ 
 of constant curvature  equal to 0 or $\pm1$: 
 one can make non-zero constant curvature equal to $\pm1$ 
  by multiplication of  metric by constant factor; this changes neither geodesics, nor 
 polynomial integrability. Thus, $\Sigma$ is either the Euclidean plane, or the unit sphere, or the Lobachevsky--Poincar\'e 
 hyperbolic plane. 
 It is modeled  as one of the three following  surfaces in the space $\rr^3$ with coordinates  $x=(x_1,x_2,x_3)$ equipped with the quadratic form 
 $$<Ax,x>, \ A\in\{\diag(1,1,0),\diag(1,1,\pm1)\}, \ < x,x >=x_1^2+x_2^2+x_3^2.$$

- Euclidean plane: $\Sigma=\{x_3=1\}$, $A=\diag(1,1,0)$. 

- The unit sphere: $\Sigma=\{ x_1^2+x_2^2+x_3^2=1\}$, $A=Id$.

- The hyperbolic plane: $\Sigma=\{ x_1^2+x_2^2-x_3^2=-1\}\cap\{ x_3>0\}$, $A=\diag(1,1,-1)$. 

The metric of constant curvature on the surface $\Sigma$ under question is induced by the  quadratic form $<Ax,x>$. 
The {\it geodesics} on $\Sigma$ are its intersections with two-dimensional vector subspaces in $\rr^3$. The {\it conics} 
on $\Sigma$ are its intersections with quadrics $\{<Cx,x>=0\}\subset\rr^3$, where $C$ is a real symmetric $3\times 3$-matrix. 

\begin{example}  The billiard in a disk in $\rr^2_{(x_1,x_2)}$ centered at 0 
has  first integral $x_1P_2-P_1x_2$ linear in $P$. 
The billiard in any conic in any of the above surfaces $\Sigma$ has an  
 integral quadratic in $P$,  see \cite[proposition 1]{bolotin2}. 
 \end{example}

 \begin{theorem} \label{thal0} Let a   billiard in $\Sigma$ with a $C^2$-smooth  connected boundary  be polynomially integrable. Let its boundary be not contained in a geodesic.   
 Then the billiard boundary is a conic (or a connected component of a conic). 
 \end{theorem}

 \begin{corollary} Every bounded polynomially integrable planar billiard with a $C^2$-smooth connected 
 boundary is an ellipse.
 \end{corollary}

 Below we extend the above theorem to billiards with countably piecewise smooth boundaries, see the following 
 definition. 
 
 \begin{definition} \label{counts} A domain $\Omega\subset\Sigma$ has {\it countably piecewise ($C^r$-) 
 smooth boundary,} if $\partial\Omega$ consists of the two following parts:
 
 - the {\it regular part:} an open and dense subset $\partial\Omega_{reg}\subset\partial\Omega$, where each point $X\in\partial\Omega_{reg}$ has a neighborhood 
 $U=U(X)\subset\Sigma$ such that the intersection $U\cap\partial\Omega$ is a ($C^r$-) smooth one-dimensional submanifold 
 in $U$; 
 
 - the {\it singular part:} the 
 closed subset $\partial\Omega_{sing}=\partial\Omega\setminus\partial\Omega_{reg}\subset\partial\Omega$.
 \end{definition} 
 
 \begin{remark} In the above definition the regular part of the boundary is always a dense and 
 at most countable disjoint union of ($C^2$-) smooth arcs 
 (taken without endpoints). The particular case of domains with piecewise smooth boundaries corresponds to the case, when 
 the above union is finite, the arcs are smooth up to their endpoints and the singular part of the boundary is a finite set (which consists of endpoints and may be 
 empty). For a general  billiard with countably piecewise smooth boundary the  billiard flow is well-defined on a residual set 
 for all time values. 
In the case, when the singular part of the boundary has zero one-dimensional Hausdorff measure, 
the billiard flow is well-defined as a flow of measurable  transformations.
\end{remark}
\begin{remark}
The notions of polynomially (analytically)  integrable billiards obviously extend to billiards with countably piecewise smooth boundaries, and these two 
notions are equivalent, as in the piecewise smooth case.  
\end{remark}

  \begin{definition} \label{count-con} A billiard in $\rr^2$ 
  with countably piecewise smooth boundary is called {\it countably confocal,} if the regular 
  part of its boundary  consists of arcs of confocal conics  and may be some straight-line segments such that 
 
 - at least one conical arc is present; 
 
 - in the case, when the common foci of  the conics are distinct and finite (i.e., the conics are ellipses and (or) hyperbolas),  
 the ambient line of each straight-line segment of the boundary is either the line through the foci, or the middle orthogonal line 
 to the segment connecting the foci, see Fig. 1a);
 
 - in the case, when the conics are concentric circles, the above ambient lines may be any lines through their common center, see Fig. 1b);
 
 - in the case, when the conics are confocal parabolas,  the ambient line of each straight-line segment of the boundary is either 
 the common axis of the parabolas, or the line through the focus that is orthogonal to the axis, see Fig. 1 c), d). 
  \end{definition}
 
 Let us extend the above definition to the non-Euclidean case. To do this, let us recall the  following definition.

 \begin{definition} \label{pencil} \cite[p.84]{veselov2}.
 Let  $\Sigma\subset\rr^3$ be one of the standard surfaces of constant curvature defined by a quadratic form   $<Ax,x>$. Let $B$ be a real symmetric 
$3\times 3$-matrix that is not proportional to $A$. In the Euclidean case, when $A=\diag(1,1,0)$, we require in addition that the $x_3$-axis 
 does not lie in  $\ker B$. 
 The {\it pencil of confocal conics} in $\Sigma$ defined by $B$ 
 consists of the conics 
 \begin{equation}\Gamma_{\lambda}=\Sigma\cap\{<B_{\la}x,x>=0\}, \ B_{\la}=(B-\la A)^{-1}.\label{confb}\end{equation}
 For those $\la$, for which $\det(B-\la A)=0$ and the kernel $K_{\la}=\ker(B-\la A)$ is one-dimensional, we set $\Gamma_{\la}$ to be the geodesic\footnote{Everywhere below, whenever the contrary is not specified, the orthogonal complement sign $\perp$  and the vector product are understood with respect to the standard Euclidean scalar product on 
 $\rr^3$.} 
 \begin{equation}\Gamma_{\la}=\Sigma\cap K_{\la}^{\perp},\label{geodconf}\end{equation}
 provided that the latter intersection is non-empty. 
 In the case, when  $dim K_{\la}=2$,  for every two-dimensional vector subspace $H\subset\rr^3$ orthogonal to $K_{\la}$ the intersection 
 $\Sigma\cap H$ will be  also denoted $\Gamma_{\la}=\Gamma_{\la}(H)$. 
 \end{definition}
 
 \begin{remark} \label{detneq} In the conditions of Definition \ref{pencil} the confocal conic pencil is 
 well-defined: $\det(B-\la A)\not\equiv0$ as a function of $\lambda$.  In the 
 non-Euclidean cases this is obvious, since the matrix $A$ is non-degenerate. In the 
 Euclidean case one has $A=\diag(1,1,0)$ and the $x_3$-axis does not lie in $\ker B$: 
 that is, some of the matrix elements $B_{13}$, $B_{23}$, $B_{33}$ is non-zero. One has 
 $$\det(B-\la A)=-\la^3\det(A-\la^{-1}B)$$
 \begin{equation}=\la^2B_{33}+\la(B_{13}^2+B_{23}^2-B_{33}(B_{11}
 +B_{22}))+\det B\not\equiv0:\label{detl}\end{equation}
 in the above right-hand side the identical vanishing of the coefficients at 
 $\la^2$ and at $\la$ 
  would imply that $B_{33}=B_{13}=B_{23}=0$, which is forbidden by our assumptions. 
  Hence, $\det(B-\la A)\not\equiv0$. Conversely, if in the Euclidean case 
  the $x_3$-axis were contained in the kernel of the matrix $B$, 
  then obviously $\det(B-\la A)\equiv0$, and the confocal pencil would not be well-defined. 
  \end{remark}

 \begin{remark} \label{rpencil} The matrix $B$ 
 is uniquely defined modulo transformation $B\mapsto \mu B-\la A$, $\mu\neq0$ 
 (i.e., modulo $\rr A$ and up to constant factor) by the corresponding confocal pencil. In the Euclidean case, when $\Sigma=\{ x_3=1\}$,  $A=(1,1,0)$,  the above notion of confocal conics 
  coincides with the classical one. In the Euclidean case the 
 kernel $K_{\la}$ is two-dimensional for some $\la$, if and only if the confocal conics under question are concentric circles; then the 
 corresponding geodesics $\Gamma_{\la}(H)$ are the lines through their common center. 
  \end{remark}
 
 \def\Ga{\Gamma}
 \begin{definition} \label{adm} Consider a confocal conic pencil (\ref{confb}) defined by a matrix $B$. 
 The corresponding {\it admissible geodesics} are the following: 
  
 1) Each geodesic  $\Gamma_{\la}$ in (\ref{geodconf}) and (or) $\Gamma_{\la}(H)$ (if any) is admissible. 
 
 2) Consider the special case, when  $B=Aa\otimes b+b\otimes Aa$ (modulo $\rr A$, see Remark 
 \ref{rpencil}) where $a,b\in\rr^3\setminus\{0\}$, $<a,b>=0$.
 
 2a) In the subcase, when $\Sigma$ is non-Euclidean: those  of the geodesics 
 \begin{equation}\{ r\in\Sigma \ | \ <r,a>=0\}, \ \{ r\in\Sigma \ | \ <r,Ab>=0\}\label{newgeod}\end{equation}
 that are well-defined (i.e., non-empty) are also admissible.
 
 2b) In the subcase, when $\Sigma=\{ x_3=1\}$ is  Euclidean  and $b$ is not parallel to $\Sigma$: 
 only $\Gamma_{\lambda}$ and the first geodesic in (\ref{newgeod}) are admissible.
  \end{definition}
 
 \begin{remark}  Note that the subcase in 2) when $\Sigma=\{ x_3=1\}$ and $b$ is parallel to $\Sigma$ 
 is impossible, since in this subcase the $x_3$-axis would lie in the kernel $\ker B$, 
 which is forbidded by our assumptions. This implies that in subcase 2b) the first geodesic in (\ref{newgeod}) is non-empty: 
 the vector $a$ is not vertical, since its orthogonal $b$ is not horizontal. 
 In the above subcases 2a), 2b) the corresponding admissible geodesics from (\ref{newgeod}) 
 do not coincide   with geodesics $\Gamma_\la$ ($\Gamma_{\la}(H)$). Indeed, suppose the contrary, say, the first 
 geodesic $a^{\perp}\cap\Sigma$ in (\ref{newgeod}) is non-empty and 
 coincides with some $\Gamma_\la$ or $\Gamma_\la(H)$. Then $a\in \ker(B-\la A)$, that is, 
$$<Aa,a>b+<b,a>Aa=<Aa,a>b=\la Aa.$$
 Thus, either $<Aa,a>=0$ and $\la Aa=0$, or the vector $b$, which is orthogonal to $a$, is proportional to $Aa$, 
 thus $<Aa,a>=0$ again. But then $a^{\perp}\cap\Sigma=\emptyset$, see \cite[p.122]{bolotin2}, -- 
 a contradiction. The case, when the second geodesic in (\ref{newgeod}) coincides with $\Gamma_\la$, is treated 
 analogously. The above non-coincidence statement can be also  deduced from the next proposition. 
 \end{remark}
 
 \begin{remark} \label{permgeod}
In the subcase 2a) set $\wt a=Ab$, $\wt b=Aa$. Then 
 $B=A\wt a\otimes \wt b+\wt b\otimes A\wt a$, 
 and $<\wt a,\wt b>=0$, since $A^2=Id$. 
The geodesics in (\ref{newgeod}) are written in terms of the new vectors $\wt a$ and $\wt b$  in the opposite order. Thus, each 
geodesic of type (\ref{newgeod}) can be represented by  the first equation in (\ref{newgeod}) for appropriate presentation $B=Aa\otimes b+b\otimes Aa$. 
\end{remark}

 \begin{definition} \label{confocality-2} A billiard $\Omega\subset\Sigma$ with a countably piecewise smooth boundary is {\it countably confocal,} 
if its boundary consists of arcs of confocal conics (at least one conical arc is present) and may be 
 some  segments of geodesics admissible with respect to  the confocal conic pencil given by the conical arcs in 
 $\partial\Omega$, see Definition \ref{adm}.
 \end{definition}
  
  Confocal billiards with piecewise smooth boundaries were introduced by S.V.Bolotin \cite{bolotin2}, who proved their polynomial integrability with integrals of first, second or fourth degree. See the following proposition, whose proof 
 presented in loc. cit.  remains valid in the countably piecewise smooth case. 
 
 \begin{proposition} \label{propconf} \cite[proposition 1 in section 2; the theorem in section 4]{bolotin2}
 Each countably confocal  billiard is polynomially integrable: it has a non-trivial 
 first integral that is either  linear, or  quadratic, or a degree 4 polynomial 
 in the velocity components 
  that is non-constant on the unit velocity hypersurface.  If all the geodesic pieces of its boundary lie in $\Gamma_\la$ ($\Gamma_\la(H)$), then the  integral can be chosen of degree at most 2. 
  The case of a degree 4 integral that cannot be 
 reduced to an integral of degree at most 2 is exactly the case, when the conics forming the billiard boundary are contained in 
 a confocal pencil of types 2a) or 2b) from Definition \ref{adm} and the billiard boundary contains at least one segment of 
   some of the admissible geodesics from 
 (\ref{newgeod}) mentioned in 2a) and 2b) respectively. 
 \end{proposition}
 
 \begin{figure}[ht]
  \begin{center}
   \epsfig{file=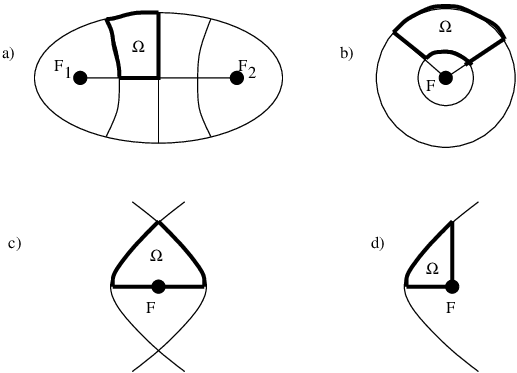}
    \caption{Examples of  confocal  planar billiards; $F_1$, $F_2$, $F$ are the foci; the conics in c) and d) are parabolas.  All of these billiards 
    have quadratic integrals, except for the billiard at Fig. 1d), which has a degree 4 integral.}
    \label{fig:0}
  \end{center}
\end{figure}

 \begin{example} \label{rconf} For Euclidean billiards the two countably confocality notions given by 
 Definitions \ref{count-con} and \ref{confocality-2} are equivalent. 
A Euclidean billiard whose boundary contains an arc of parabola and a segment of the line 
through the focus that is orthogonal to the axis of the parabola, as at Fig. 1d), is exactly a  
billiard of  type 2b), see 
the end of  paper \cite{bolotin2}; the above line is the first geodesic in  (\ref{newgeod}). 
This example of a billiard having a 
degree 4 integral was first discovered in \cite{gram}. Analogous 
billiards on surfaces of non-zero constant curvature were constructed in \cite{abdr2}. 
  \end{example}

The main result of the paper is the following theorem. 

 \begin{theorem} \label{thal}\footnote{Theorem \ref{thal} with a brief proof was announced in the author's note \cite{gl}.}  Let a  billiard in $\Sigma$ with countably piecewise $C^2$-smooth boundary be polynomially integrable (or equivalently, analytically  
 integrable, see Definition \ref{defanaly}), and let the regular part of its boundary  contain at least one 
non-geodesic arc. Then the billiard is countably confocal.
 \end{theorem}

 Theorem \ref{thal} is a  joint result of M.Bialy, A.E.Mironov and the author. Its proof sketched below 
 consists of the two following  parts: 
 
1) The papers \cite{bm, bm2} of Bialy and Mironov, whose geometric construction reduced the proof of Theorem \ref{thal}  
 to a purely algebro-geometric problem that was partially investigated by them.
 
 2) The complete solution of the above-mentioned algebro-geometric problem obtained in the 
 present paper (Theorem \ref{conj1}).

\subsection{Sketch of proof of Theorem \ref{thal} and plan of the paper} 
In what follows a point $r\in\Sigma$ will be identified with its radius-vector in $\rr^3$. 

 \begin{theorem} \label{algnew} (S.V.Bolotin, see  \cite{bolotin}, \cite[p.118;  proposition 2 and its proof on p.119]{bolotin2}, \cite[chapter 5, section 3, proposition 5]{kozlov}.) For every polynomially 
 integrable billiard $\Omega\subset\Sigma$ with countably piecewise 
 $C^2$-smooth boundary 
 a polynomial  integral non-constant on the unit velocity  hypersurface $\{|P|=1\}$ can be 
 chosen to be a homogeneous polynomial $\Psi(M)$ 
 of even degree in the components of the moment vector
  \begin{equation}M=[r,P]=(x_2P_3-x_3P_2, -x_1P_3+x_3P_1, x_1P_2-x_2P_1),\label{moment}\end{equation}
 $$r=(x_1,x_2,x_3)\in\Sigma, \ \ P=(P_1,P_2,P_3) \text{ is the velocity vector.}$$ 
 (This statement is  local and holds for  reflection from an arbitrary smooth curve in $\Sigma$.)  
 Each $C^2$-smooth arc of the boundary $\partial\Omega$ with non-zero geodesic curvature 
 lies in an algebraic curve.
 \end{theorem}

 \begin{theorem} \label{bol2}  (see   \cite[section 4]{bolotin2}). 
 Let a billiard on $\Sigma$ with a countably piecewise $C^2$-smooth boundary be polynomially integrable. 
 Let its boundary contain a non-geodesic  {\bf conical} arc. Then the billiard is countably confocal. 
 \end{theorem}
 
 \begin{remark} S.V.Bolotin's theorems implying Theorems \ref{algnew} and \ref{bol2} were stated and proved in loc. cit. for piecewise smooth boundaries, but their proofs remain valid in the countably piecewise smooth case. 
   To make the paper self-contained and to extend the main results to complex domain, 
 we give a proof of Theorem \ref{bol2} in Subsection 2.2. It  follows the arguments 
 from \cite[section 4]{bolotin2}, but  here it is done in the dual terms using results of Bialy and Mironov from \cite{bm, bm2}. 
 \end{remark}

The boundary $\partial\Omega$ is countably piecewise $C^2$-smooth. Therefore, it  
contains an open and dense subset contained in $\partial\Omega_{reg}$ 
that is a disjoint union of geodesic segments and $C^2$-smooth arcs with non-zero 
geodesic curvature. 

Let $\alpha\subset\partial\Omega$ be a  $C^2$-smooth arc with non-zero geodesic curvature: it existence 
follows from assumptions. 
By Bolotin's Theorem \ref{bol2}, for the proof of Theorem  \ref{thal} it suffices to show that $\alpha$ contains a 
conical sub-arc. To do this, we use Bialy--Mironov construction of the dual billiard 
and their results presented in Subsection 2.1. Let us describe them briefly.

In what follows $\pi:\rr^3\setminus\{0\}\to\rp^2$ denotes the tautological projection. Its complexification 
and restriction to  $\Sigma$ will be also denoted by  $\pi$. 

Recall that the standard Euclidean scalar product $<x,x>$ on $\rr^3$ defines the {\it orthogonal polarity:} 
the correspondence sending each two-dimensional vector subspace in $\rr^3$ to its 
Euclidean-orthogonal one-dimensional subspace. This together with the projection $\pi$ induces a projective duality 
$\rp^{2*}_{(x_1:x_2:x_3)}\to\rp^2_{(M_1:M_2:M_3)}$ sending lines  to points. 
Namely, each projective line, which is the projection of a two-dimensional vector subspace $H$ (punctured at the origin), 
is dual to the point $\pi(H^\perp\setminus\{0\})$. (It is well-known that in the affine chart $(x_1:x_2:1)$ the projective duality defined 
by the orthogonal polarity is the composition of the polar duality with respect to the  unit circle and the central symmetry with respect to the origin.)

For simplicity, the curve dual to the projection 
$\pi(\alpha)\subset\rp^2$ with respect to the above projective duality will be denoted by $\alpha^*$ 
and called the {\it curve $\Sigma$-dual to $\alpha$.} By definition, the dual curve $\alpha^*$ is the family of points 
in $\rp^2$ that are dual to the projective tangent lines to the curve $\pi(\alpha)\subset\rp^2$. The curve $\alpha^*$ is 
$C^1$-smooth, since the curve $\pi(\alpha)$ is $C^2$-smooth and has no inflection points: the geodesic curvature 
of the curve $\alpha$ is non-zero. 

Bialy and Mironov proved the following results in \cite{bm, bm2}:

- Let $\Psi(M)$ be the homogeneous first integral of even degree $2n$ from Bolotin's Theorem \ref{algnew}. For every point $B\in\alpha^*$ the restriction to the projective tangent line $T_B\alpha^*$ of the rational function 
\begin{equation}G(M)=\frac{\Psi(M)}{<AM,M>^n}\label{ratg}\end{equation} 
is invariant under a special projective involution $T_B\alpha^*\to T_B\alpha^*$ 
fixing $B$: the so-called angular symmetry centered at $B$. More 
precisely,  invariance of the function $G$ is  equivalent to the statement saying 
that for every $r\in\alpha$ the function  
$\Psi(M)=\Psi([r,v])$ in $v\in T_r\Sigma$  is invariant under reflection of the vector $v$ 
 from the line $T_r\alpha$. 

-  Consider the so-called {\it absolute:} the complex conic
\begin{equation}\ii=\{<AM,M>=0\}\subset\cp^2_{(M_1:M_2:M_3)}.\label{absol}\end{equation} 
The above angular symmetry coincides with the restriction to $T_B\alpha^*$ of the 
 unique projective involution $\cp^2\to\cp^2$ fixing $B$ that fixes each line 
through $B$ and permutes its intersection points with the conic $\ii$: the so-called 
 {\it $\ii$-angular symmetry} centered at $B$.
 
- Concider the complex projective Zariski closure of the curve $\alpha^*$, which is an algebraic curve, by  
Theorem \ref{algnew}. Each its non-linear irreducible  component $\gamma$  {\it generates a rationally integrable $\ii$-angular billiard}, see Definition \ref{defang}: 
for every point $B\in\gamma\setminus\ii$ the 
restriction of a rational function $G$ to the projective tangent line $T_B\gamma$ is invariant under the $\ii$-angular 
symmetry centered at $B$; the  function $G$ is non-constant on $\cp^2$ and has poles in $\ii$. 

- For every  curve $\gamma$ generating a rationally integrable $\ii$-angular billiard all its singular and inflection 
points (if any) lie in $\ii$. 

The main algebro-geometric result of the present paper, which implies the main results, is the following theorem. 

\begin{theorem} \label{conj1}  Let $\ii\subset\cp^2$ be a conic: either regular, or a union of two distinct 
 lines. Every irreducible algebraic curve $\gamma\subset\cp^2$ different from a line and from $\ii$ and
  generating a rationally integrable 
 $\ii$-angular billiard is a conic. 
 \end{theorem}

For the proof of Theorem \ref{conj1} 
 we study  {\it local branches} of the curve $\gamma$ at points  $C\in\gamma\cap\ii$: the  irreducible 
components  of the germ $(\gamma,C)$. Each local branch is 
 holomorphically bijectively parametrized in so-called adapted affine 
coordinates by small complex parameter $t$ as follows: 
$$t\mapsto(t^{q},ct^{p}(1+o(1))), \ \text { as } t\to0; \  \ q,p\in\nn,  \ \ 1\leq q<p, \ \ c\neq0.$$
In Section 4 we prove Theorem \ref{iii} giving a list of statements  on $p$ and $q$ 
satisfied by  local branches of appropriate type (see Cases 1) and 2) below). 
Afterwards in Section 5 
we prove the following general algebro-geometric theorem. It states that Bialy--Mironov  inclusions 
$Sing(\gamma), Infl(\gamma)\subset\ii$  and the  
statements of Theorem \ref{iii} on local branches together imply that $\gamma$ is a conic. 

 \begin{theorem} \label{thalg2} Let $\ii\subset\cp^2$ be a conic: either regular, or a union of two distinct 
  lines.  Let $\gamma\subset\cp^2$ be an irreducible complex algebraic curve different from a line and from 
  $\ii$.  Let all the singularities and inflection points (if any) of the curve 
 $\gamma$ lie in  $\ii$. Let for every $C\in\gamma\cap\ii$ the 
 local branches $b$ of the curve $\gamma$ at $C$ satisfy the following statements:
 
Case 1): $C$ is a regular point of the conic $\ii$.  If $b$ is tangent to $\ii$, then it is quadratic: $p=2q$.  
If $b$ is transversal to $\ii$, then it  is regular and quadratic: $q=1$, $p=2$. 
 
Case 2): $\ii$ is a union of two distinct lines intersecting at $C$. If $b$ is  transversal to both 
lines, then $b$ is subquadratic: $p\leq2q$. 

Then $\gamma$ is a conic. 
 \end{theorem}
 
 The proof of Theorem \ref{thalg2}, which will be  given in Section 5, 
 is based on the ideas and  arguments due to E.Shustin 
 on plane curve invariants from the proof of its analogue for the outer billiards case \cite[subsections 4.1, 4.2]{gs}. 
 
 The most technical part of the paper is the proof of statement (ii-b) of Theorem \ref{iii}, which asserts that 
 each local branch of the curve $\gamma$ that is transversal to $\ii$ and is based at a regular point of the conic 
 $\ii$ is regular and quadratic. Its proof is based on a remarkable formula of Bialy and Mironov for the Hessian of the 
 function defining $\gamma$, see \cite[theorem 6.1]{bm} 
and \cite[formulas (16) and (32)]{bm2}. This formula is recalled in Section 3 as formula (\ref{hesf}). We use  
 asymptotic formulas for both sides of Bialy--Mironov formula along the transversal local branches that are stated and proved 
 in Subsection 3.4. In their proofs we use asymptotic formulas for the defining functions 
 and their Hessians stated and proved in Subsections 3.2 and 3.3 
 respectively.

In Section 6 we prove the main results: Theorems \ref{conj1}, \ref{thal} and \ref{thal0} and the complexification of Theorem \ref{thal} stated in the next subsection. 

 \def\mca{\mathcal A} 
 
 \subsection{Complexification} 
 \def\la{\lambda}
 
 Here we state a complexification of Theorem \ref{thal}, which deals with the space 
 $\cc^3_{(x_1,x_2,x_3)}$ 
 equipped with a quadratic form $<Ax,x>$, $x=(x_1,x_2,x_3)$, 
 $A\in\{\diag(1,1,0), \ \diag(1,1,\pm1)\}$, and a complex surface $\Sigma\subset\cc^3$. 
 
 -  Euclidean case: $\Sigma=\{ x_3=1\}$,  $A=\diag(1,1,0)$.
 
 - Non-Euclidean case: $\Sigma=\Sigma_{\pm}=\{<Ax,x>=\pm1\}$, 
 $A=\diag(1,1,\pm1)$.
 
 We equip the surface $\Sigma$ under question  with the complex bilinear quadratic form induced by the form $<Ax,x>$ on its tangent planes. 
 
 Note that the surfaces $\Sigma_{\pm}$  are regular, connected and obtained one from the other by the transformation 
 $(x_1,x_2,x_3)\mapsto(ix_1,ix_2,x_3)$, but the latter transformation changes the sign of the  
 quadratic form $<Ax,x>$ on $T\Sigma_{\pm}$. 
 
 Recall that a one-dimensional subspace  $\Lambda$ in a complex linear space equipped 
 with a $\cc$-bilinear scalar product is {\it isotropic,} if each vector in $\Lambda$ has 
 zero scalar square. A holomorphic curve $\Lambda$ in a complex manifold $\Sigma$ 
 equipped with a $\cc$-bilinear quadratic form on $T\Sigma$ 
 is {\it isotropic,} if for every $x\in\Lambda$ the  tangent subspace 
 $T_x\Lambda\subset T_x\Sigma$ is isotropic. 
 
    A {\it complex  geodesic} is 
 
 - a non-isotropic line in $\Sigma=\cc^2$ in the Euclidean case;
 
 -  the  intersection of the surface 
 $\Sigma$ with 
 a two-dimensional subspace in $\cc^3$ that is not tangent to the light cone $\widehat\ii=\{<Ax,x>=0\}$ in the non-Euclidean case.  
 
 The reason to cross out the planes tangent to $\widehat\ii$ is the following.
 
 \begin{proposition} \label{propiso}  Consider the non-Euclidean case: $A=\diag(1,1,\pm1)$. For every two-dimensional vector subspace $H\subset\cc^3$ 
 tangent to the light cone $\widehat\ii$ the intersection $H\cap\Sigma$ is a union of two parallel isotropic straight lines.  
 Each isotropic holomorphic curve in $\Sigma$ is a line contained in a two-dimensional vector subspace in $\cc^3$ 
 tangent to $\widehat \ii$. For every $r\in\Sigma$ the one-dimensional isotropic vector subspaces in the plane $T_r\Sigma$ are 
 exactly its intersections with two-dimensional vector subspaces in $\cc^3$  containing $r$ and tangent to $\widehat \ii$: there are exactly two of them. 
 \end{proposition}
 \begin{proof} For every $r\in\Sigma$ the quadratic form on $T_r\Sigma$ induced by 
  $<Ax,x>$ is non-degenerate, 
 since $T_r\Sigma$ is $<Ax,x>$-orthogonal  to the radius-vector of the point $r$ and transversal to it: $<Ar,r>=\pm1\neq0$. 
For every two-dimensional  subspace $H$ tangent to $\widehat\ii$ the restriction of the form $<Ax,x>$ to 
 $H$ is non-zero and has a non-zero kernel $K$: the tangency line of the plane $H$ with $\widehat\ii$. Hence, in 
 appropriate affine coordinates $(z_1,z_2)$ on $H$ centered at $0$ one has $<Ax,x>|_H=z_1^2$, $K=\{ z_1=0\}$, 
 $H\cap\Sigma=\{ z_1^2=\pm1\}$. 
 Therefore, the intersection $H\cap\Sigma$ is a union of two lines parallel to $K$, which are thus isotropic. The first 
 statement of the proposition is proved. 
 
 Let us now prove the third and the second statements of the proposition. For every point $r\in\Sigma$ the tangent plane 
 $T_r\Sigma$ equipped with the quadratic form induced by $<Ax,x>$ contains two distinct one-dimensional isotropic vector subspaces, by non-degeneracy. 
 Each ot them is the line of intersection of the plane $T_r\Sigma$ with a two-dimensional 
 subspace $H$ through $r$ that is tangent to $\widehat\ii$.  This follows from the first statement of the proposition 
 and the fact that there are two distinct 2-dimensional subspaces through $r$ that are tangent to $\widehat\ii$. 
 This implies the third statement of the proposition. This also implies that 
 each isotropic  curve in $\Sigma$ is locally a phase curve of 
 a (double-valued) holomorphic line field defined by the above intersections.  The only phase curves of the latter field are the isotropic lines in $H\cap\Sigma$, $H$ being tangent to $\widehat\ii$, 
 by the first statement of the proposition and uniqueness theorem in ordinary differential equations. 
 This proves the proposition.
 \end{proof}
  
  \begin{definition} Consider the surface $\Sigma$ in the non-Euclidean case. Let $\gamma\subset\Sigma$ 
  be a complex geodesic. Let $\mcg_{\gamma}$ denote the stabilizer of the geodesic $\gamma$ in the  
  group of automorphisms $\cc^3\to\cc^3$ preserving the form $<Ax,x>$. Its identity component, which will be 
  denoted by $\mcg_{\gamma}^0$, will be called the {\it group 
  of translations along the geodesic $\gamma$}. A translation of the complex Euclidean plane along 
  a complex line $L$ is the translation by a vector parallel to $L$. 
  \end{definition}
  \begin{remark} In the above definition in the non-Euclidean case 
  let $H\subset\cc^3$ denote the corresponding two-dimensional vector 
  subspace: $\gamma=H\cap\Sigma$.  The geodesic $\gamma$ is thus a regular conic in the plane $H$ that is 
  biholomorphically parametrized by $\cc^*$. Its projective closure $\hat\gamma$ in 
  $\cp^2\supset H$ intersects the infinity line $\cp^2\setminus H$ at two distinct points. 
  The restrictions to $\gamma$ of the translations along the geodesic $\gamma$ 
  are exactly those conformal automorphisms $\hat\gamma\to\hat\gamma$ that fix 
  the latter intersection points. One has $\hat\gamma\simeq\oc$, $\gamma\simeq\cc^*$. 
  This yields to a natural isomorphism $\mcg_{\gamma}^0\simeq\cc^*$.
  \end{remark}
 
  \begin{definition}  A {\it complex billiard} on $\Sigma$ 
  is a collection (finite or infinite, countable or uncountable) of holomorphic curves 
  $\Gamma_{t}\subset\Sigma$ distinct from isotropic lines 
  (see  \cite[definition 1.3]{alg} for finite collections in the Euclidean case). 
  A complex billiard is said to be {\it polynomially integrable,} if there exists a function $\Phi(r,P)$ on $T\Sigma$ (called a 
  {\it polynomial integral}) that is polynomial in $P\in T_r\Sigma$ with the following 
 properties:
 
 - $\Phi|_{\{<AP,P>=1\}}\not\equiv const$;
 
 - the restriction of the function $\Phi$ to the tangent bundle of every complex geodesic
 is invariant under the translations along the geodesic; 
 
  - for every point $r\in\Gamma_{t}$ such that the line $T_r\Gamma_{t}$ is non-isotropic for the quadratic form on $T_r\Sigma$ induced by $<Ax,x>$ 
  the restriction  $\Phi|_{T_r\Sigma}$ 
  is invariant under the {\it symmetry} with respect to  the complex line $T_r\Gamma_{t}$ 
  (see \cite[definition 2.1]{alg}): the unique non-trivial $\cc$-linear involution $T_r\Sigma\to T_r\Sigma$ that preserves the   form  $<Ax,x>$ on $T_r\Sigma$ and fixes the points of the 
  line $T_r\Gamma_t$.  
 \end{definition}
 
 \begin{example} Consider a polynomially integrable  billiard with countably piecewise smooth boundary in a real surface of 
 constant curvature.  
 Then the smooth part of the boundary is contained in a union of arcs of conics and segments of admissible geodesics 
 (Theorem \ref{thal}). 
 Their complexifications form a complex billiard having a polynomial integral that is the complexification of the real polynomial integral 
 of the real billiard: it can be chosen of degree no greater than four, see Proposition \ref{propconf}. 
 \end{example}
 
 \begin{remark}
 The confocality notion  from Definition \ref{pencil} for real conics extends to the case of complex conics in $\Sigma$ 
 without changes  in both non-Euclidean and Euclidean cases 
 with $B$ being a complex symmetric matrix 
and $\la\in\cc$. In the Euclidean case this complex confocality notion is equivalent to 
the one  
 given in  \cite[definition 2.24]{alg}, which follows from definition and 
Remark \ref{rpencil}. 
\end{remark}

\begin{remark} A pencil of confocal conics given by a matrix $B$ is well-defined, if and only if inequality (\ref{detl}) holds: 
$\det(B-\la A)\not\equiv0$ as a function of $\la$. In the real case inequality (\ref{detl}) is equivalent 
to the condition that the $x_3$-axis is not contained in the kernel of the matrix $B$, i.e., 
$(B_{13}, B_{23},B_{33})\neq(0,0,0)$, see Remark \ref{detneq}. In the complex case inequality (\ref{detl}) is equivalent 
to the following weaker condition: for every choice of sign $\pm$ the equalities
$$B_{33}=0, \ B_{23}=\pm iB_{13}, \ B_{13}^2(B_{11}-B_{22}\pm2iB_{12})=0$$
do not hold simultaneously. 
\end{remark}

 \begin{definition} A complex billiard $\Gamma_{t}$ is said to be {\it confocal,} if the set of its curves  different from complex geodesics 
 is non-empty,  all of them are confocal 
 complex conics, and the complex geodesics from the family $\Gamma_{t}$ are admissible with respect to the 
 corresponding confocal conic pencil in the sense of  Definition \ref{adm}, where now everything is complex: $B$, $\la$, $a$, $b$,... 
  \end{definition}
  \begin{remark} A priori, some complex curves $\Gamma_{\lambda}$ in (\ref{geodconf}), 
  $\Gamma_{\lambda}(H)$ and some subsets in (\ref{newgeod}) 
  may be isotropic lines; then they are not complex geodesics, and we do not call them 
  admissible. For example, in the non-Euclidean case let $\la\in\cc$ be such that the kernel $K_{\la}=\ker(B-\la A)$ 
  is one-dimensional. The corresponding intersection $\Gamma_{\lambda}=K_\la^{\perp}\cap\Sigma$ 
  is isotropic, if and only if $K_{\lambda}\subset\widehat\ii$. This 
  follows from Proposition  \ref{propiso} and since  the Euclidean orthogonal $K_{\lambda}^{\perp}$ is tangent to 
 the light cone $\widehat\ii$ if and only if $K_\la\subset\widehat\ii$: see the last statement of Corollary \ref{self} in Subsection 2.2.
  \end{remark}
 \begin{theorem} \label{talc} Every polynomially integrable complex billiard $\Gamma_{t}$ on $\Sigma$ 
 containing at least one non-geodesic curve is 
 confocal and has an integral $\Phi(r,P)=\Psi(M)$ (where $M=[r,P]$ is the complexified Euclidean vector product) 
 that is a homogeneous polynomial in $M$ of degree at most four. The integral can be chosen 
  quadratic in  $M$, except for the cases 2a), 2b) in Definition \ref{adm}, when $\Gamma_{t}$ contains a corresponding 
  admissible complex geodesic of type (\ref{newgeod}): 
   in this case there is an integral of degree four. 
  \end{theorem}
  
   Theorem \ref{talc} will be proved in Subsection 6.4.

 \subsection{Historical remarks}
 Existence  of caustics in any strictly convex planar billiard with sufficiently 
 smooth boundary was proved by V.F.Lazutkin \cite{laz}. Non-existence of caustics 
 in higher-dimensional billiards with boundaries different from quadrics was proved 
 by M.Berger \cite{berger}. 
  
 The Birkhoff Conjecture was studied by many mathematicians. In 1950 H.Poritsky \cite{poritsky} 
 proved it under the additional assumption that 
 the billiard in each closed caustic near the boundary has the same closed caustics, as the initial billiard. Later in 1988 another proof 
 of the same result was obtained by E.Amiran \cite{amiran}.  Recall that the reflection from the boundary of a 
 convex planar billiard $\Omega$ acts on the space of oriented lines intersecting $\Omega$, and their space is called the 
 phase cylinder: each line is reflected from the boundary $\partial\Omega$ at its last 
 point of intersection with $\partial\Omega$ (with respect to its orientation), and its reflected image is directed inside 
 the domain $\Omega$ at this point. In 1993 M.Bialy \cite{bialy} proved that if the phase cylinder of the billiard is 
 foliated  by non-contractible continuous closed curves which are invariant under the billiard map,  then the boundary 
 $\partial\Omega$ is a circle. (Another proof of the same result was later obtained in \cite{wojt}.) 
  In particular, Bialy's result implies 
 Birkhoff Conjecture under the assumption that the foliation by caustics extends to the whole billiard domain punctured at 
 one point:  then the  boundary is a circle. In 2012 he proved a similar result  for billiards on the constant curvature surfaces \cite{bialy1} and also for magnetic billiards 
 \cite{bialy2}. In 1995 A.Delshams and R.Ramirez-Ros suggested an 
 approach to prove splitting of separatrices for generic perturbation of ellipse \cite{dr}.
 In 2013 D.V.Treschev \cite {treshchev} made a numerical experience indicating that 
 there should exist analytic {\it locally integrable} billiards, with the billiard reflection map 
 having a  two-periodic point where the  germ of its second iterate is analytically conjugated to a disk rotation. Recently Treschev studied 
 the billiards from \cite{treshchev} in more detail  in \cite{tres2} and their multi-dimensional versions in \cite{tres3}. A similar effect for a ball rolling on a vertical 
 cylinder under the gravitation force was discovered in \cite{tad}: the authors have shown 
 that the ratio between its vertical and horizontal oscillation periods is a universal irrational constant, the number $\sqrt{7/2}$. 
  Recently V.Kaloshin and A.Sorrentino have proved a {\it local version} of the Birkhoff Conjecture \cite{kalsor}: 
 {\it an integrable deformation of 
 an ellipse is an ellipse}. (The case of ellipses with small extentricities was treated in the previous paper by A.Avila,  J. De Simoi 
 and V.Kaloshin 
 \cite{kavila}.)  A dynamical entropic version of Birkhoff Conjecture was stated and 
 partially studied by J.-P.Marco in \cite{marco}. 
 
 In 1988 A.P.Veselov proved that every billiard bounded by confocal quadrics in any dimension has a complete system of first 
 integrals in involution that are quadratic in $P$ \cite[proposition 4]{veselov}. In 1990 he studied a 
 billiard in a non-Euclidean 
 ellipsoid:  in the sphere and in the Lobachevsky (i.e., hyperbolic) space of any dimension  $n$.  He proved its complete integrability and 
  provided an explicit complete list of  first integrals \cite[the corollary on p. 95]{veselov2}. 
  In the same paper he proved that all the sides 
  of a billiard  trajectory are tangent to the same $n-1$ quadrics confocal to the boundary of the ellipsoid and the billiard dynamics 
  corresponds to a shift of the Jacobi variety corresponding to an appropriate hyperelliptic curve \cite[theorems 3, 2 on  p. 99]{veselov2}. 
 The Polynomial Birkhoff Conjecture together with its generalization to surfaces of constant curvature was  
 stated and studied by S.V.Bolotin, who proved in 1990 
 that in its conditions the billiard boundary lies in an algebraic 
 curve \cite{bolotin}. In the same paper and in \cite[section 4]{bolotin2} he proved the conjecture under the assumption that at least one irreducible component of the corresponding complex projective planar algebraic curve 
 is non-linear and nonsingular (in the non-Euclidean case it is also required that in addition, at least one intersection point of 
 the  latter component with the absolute be transversal). 
In \cite{bolotin2} Bolotin proved integrability of countably confocal billiards with piecewise smooth 
boundaries  with integrals of degrees two or four and a similar statement in higher-dimensional spaces of constant curvature. 
 M.Bialy and A.E.Mironov proved  the planar Polynomial Birkhoff Conjecture in the case of integrals 
of degree four \cite{bm3}. A version of the planar Polynomial Birkhoff Conjecture for {\it  families} of billiards sharing  the same polynomial 
integral (with boundaries depending continuously on one parameter) was solved in \cite{abdr}: 
in loc. cit. it is sufficient to require that the union of the boundaries do not lie in an algebraic curve in $\rr^2$, 
see \cite[end of p.30]{abdr}. Dynamics in countably confocal billiards with piecewise smooth boundaries in two and higher 
dimensions was studied in \cite{drag, dr2, dr3, dr4, dr5}. Dynamics 
in the so-called pseudo-integrable billiards (more precisely, confocal billiards with non-convex angles) was studied in \cite{dr2, dr3, dr4, dr5}. 
For further results on the Polynomial Birkhoff  Conjecture and its version for magnetic billiards see the above-mentioned papers \cite{bm, bm2, bm3} 
by M.Bialy and A.E.Mironov, \cite{bm4, bm5} and references therein. 

The analogue of the  Birkhoff Conjecture for outer billiards was stated  by S.L.Tabachnikov 
\cite{tab08} in 2008. Its polynomial version  was stated by Tabachnikov and proved by himself under genericity assumptions in the same paper, and recently  solved completely in the joint work of the author of the present paper with E.I.Shustin \cite{gs}.

\section{Preliminaries: from polynomially integrable to $\ii$-angular billiards}

\subsection{Reflection  and $\ii$-angular symmetry}

Here we  present results of M.Bialy and A.E.Mironov mentioned in Subsection 1.2 and 
 give self-contained proofs of some of them.

\begin{proposition} \label{vectis} (S.V.Bolotin, see \cite[formula (15), p.23]{bolotin2}, \cite[formula (3.12), p.140]{kozlov}). 
 For every $r\in\Sigma$ the linear operator $\mcm_r:T_r\Sigma\to V_r=r^{\perp}$, $v\mapsto[r,v]$ is an isomorphism 
 preserving the quadratic form $<Ax,x>$. Here the orthogonal complement and the vector product are taken with 
 respect to  the standard Euclidean scalar product, see Footnote 3.
 \end{proposition}

\begin{definition} Let the space $\rr^n$ be equipped with  a quadratic form $<Ax,x>$, 
 $A$ being a symmetric $n\times n$-matrix, and let $\ell\subset\rr^n$ be a one-dimensional vector subspace 
 such that $\ell\not\subset\{<Ax,x>=0\}$. The {\it pseudo-symmetry} of the space $\rr^n$ with 
 respect to the line $\ell$  is the linear involution $\rr^n\to\rr^n$ preserving 
 the quadratic form, fixing the points of the line $\ell$ and acting as central symmetry in its orthogonal complement  with respect to the form. The definition of complex pseudo-symmetry of the space $\cc^n$ equipped with a $\cc$-bilinear quadratic form 
  is analogous. 
 \end{definition}
  \begin{corollary} \label{cve} For every $r\in\Sigma$ and one-dimensional subspace $\ell\subset T_r\Sigma$ 
 the   mapping $\mcm_r:T_r\Sigma\to V_r$, $v\mapsto M$ conjugates  the pseudo-symmetry 
 $T_r\Sigma\to T_r\Sigma$ with respect to the 
 line $\ell$ and the pseudo-symmetry  $V_r\to V_r$ with respect to the one-dimensional subspace 
  orthogonal to both $r$ and $\ell$.
 \end{corollary}
  
 \begin{definition} Let $\ii\subset\cp^2$ be a conic: either a smooth conic, or a union of two distinct lines. Let  
$B\in\cp^2\setminus\ii$. For every complex line $L$ through $B$ consider its complex projective involution fixing $B$ and permuting 
its intersection points with $\ii$. (If $L$ is tangent to $\ii$, the involution under question is the unique 
non-trivial projective involution $L\to L$ fixing $B$ and the tangency point.) The transformation thus constructed for each 
$L$ is a projective involution $\cp^2\to\cp^2$ fixing $B$, which will be called the {\it $\ii$-angular symmetry} with center $B$. 
See Fig. 2 in the Euclidean case.
\end{definition} 

 \begin{proposition} \label{projps} Consider the space $\cc^3_{(M_1,M_2,M_3)}$ equipped with  a quadratic form $<AM,M>$, 
 $dim(\ker A)\leq1$. The absolute  
 $\ii=\{<AM,M>=0\}\subset\cp^2_{(M_1:M_2:M_3)}$, see (\ref{absol}), is either a regular conic, or a union of two 
 distinct lines. 
  The projectivization of a  pseudo-symmetry  $\cc^3\to\cc^3$ with respect to a one-dimensional 
 subspace $\ell$ is the $\ii$-angular symmetry with center $\pi(\ell)$. 
 \end{proposition}
 The proposition follows from definition.
 \begin{figure}[ht]
  \begin{center}
   \epsfig{file=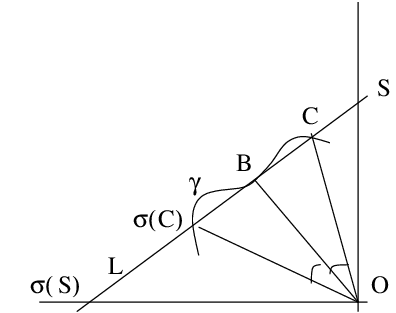}
    \caption{The $\ii$-angular symmetry $\sigma:\cp^2\to\cp^2$ with center $B$ in the Euclidean case, when 
    $\ii=\{ x_1^2+x_2^2=0\}$: the action in the affine chart 
    $\cc^2_{(x_1,x_2)}$; $O=(0,0)$. The lines $OC$ and $O\sigma(C)$ are symmetric with respect to the line $OB$. The projective lines $OS$ and $O\sigma(S)$ are isotropic for the complex Euclidean metric $dx_1^2+dx_2^2$ 
    on $\cc^2$, that is,  $\ii=OS\cup O\sigma(S)$.}
    \label{fig:2}
  \end{center}
\end{figure}

 \begin{theorem} \label{psimm}  (see \cite[theorem 1.3, p.151]{bm2} in the non-Euclidean case). 
 Let $\Omega\subset\Sigma$ be a polynomially integrable 
billiard with countably piecewise smooth boundary and a homogeneous polynomial integral 
$\Psi(M)$ of even degree. Let $r$ be a point in a smooth arc in $\partial\Omega$. Set $V_r=r^{\perp}\subset\rr^3$. Let   $L\subset V_r$ be the one-dimensional subspace Euclidean-orthogonal to  both $r$ and the tangent line $T_r\partial\Omega$. The restriction 
$\Psi|_{V_r}$ is invariant under the pseudo-symmetry of the plane $V_r$ equipped with the form 
$<Ax,x>$ with respect to the line $L$.
\end{theorem}

\begin{proof} The polynomial integral $\Psi([r,v])$ is invariant under the action on $v$ of 
the  pseudo-symmetry 
$T_r\Sigma\to T_r\Sigma$ with respect to the line $\ell=T_r\partial\Omega$ (invariance under reflection). 
This together with Corollary \ref{cve} implies the statement of the theorem.
\end{proof}

\begin{convention} \label{conven} Recall that for every $C^2$-smooth curve $\alpha\subset\Sigma$ with non-zero 
geodesic curvature its {\it $\Sigma$-dual}  is the 
curve $\alpha^*\subset\rp^2$ orthogonal-polar-dual to the projection 
$\pi(\alpha)\subset\rp^2$, see Subsection 1.2. 
For every $r\in\Sigma$ each one-dimensional vector subspace $\ell\subset T_r\Sigma$ is the intersection of the tangent 
plane $T_r\Sigma$ with a two-dimensional subspace $H\subset\rr^3$ containing $r$. The intersection 
$\widehat\ell=H\cap\Sigma$ 
is the geodesic tangent to $\ell$. The point  $\pi(H^\perp\setminus\{0\})\in\rp^2$ will be called the {\it point $\Sigma$-dual} to 
the subspace $\ell$ and to the geodesic $\widehat\ell$. It will be denoted by $\widehat\ell^*$. 
\end{convention}
 \begin{theorem} \label{ext}   Let $\Omega\subset\Sigma$ be a polynomially integrable billiard with a countably piecewise 
 $C^2$-smooth  boundary. Let $\Psi(M)$ be its    homogeneous polynomial integral of even degree $2n$. The function $G=\frac{\Psi(M)}{<AM,M>^n}$  from (\ref{ratg}) treated as a rational function on 
 $\cp^2_{(M_1:M_2:M_3)}$ satisfies the following 
 statements.
 
 1)  For every $C^2$-smooth arc $\alpha\subset\partial\Omega$ with non-zero 
geodesic curvature, let $\alpha^*\subset\rp^2$ be its 
$\Sigma$-dual curve, for every point $C\in\alpha^*$  the restriction of the function $G$ to the  projective line $T_C\alpha^*$ is invariant under the 
$\ii$-angular symmetry with center $C$. One has $G|_{\alpha^*}\equiv const$. 
 
 2) For every geodesic $\widehat \ell\subset\Sigma$ that contains a segment of the boundary $\partial\Omega$
   the function $G$ is invariant under the 
 $\ii$-angular symmetry of the whole projective plane $\cp^2$ with center $\widehat\ell^*$: the point $\Sigma$-dual to $\widehat\ell$. 
 \end{theorem}

 \begin{remark}\label{rkin} A version of statement 1) of Theorem \ref{ext} in the Euclidean case was proved in  \cite[theorem 3]{bm} for convex domains with smooth boundary. 
 But its proof remains valid in the general Euclidean case. 
  \end{remark}

\begin{proof} {\bf of Theorem \ref{ext}.} Each point $C\in\alpha^*$ is dual to the projective line 
tangent to the curve $\pi(\alpha)$ at some point $\pi(r)$, $r\in\alpha$, by definition. Consider 
the projective line $T_C\alpha^*$ and set 
$V=\pi^{-1}(T_C\alpha^*)\cup\{0\}\subset\rr^3$. It is 
 the two-dimensional subspace orthogonal to the 
line $Or$, by definition. Set $L=\pi^{-1}(C)\cup\{0\}\subset V$: it is the one-dimensional 
subspace orthogonal to both lines $T_{r}\alpha$ and $Or$, by definition. 
The restrictions to $V$ of 
both functions $\Psi(M)$ and $<AM,M>$ are invariant 
under the pseudo-symmetry of the plane $V$ with respect to the line  $L$, by Theorem \ref{psimm} and isometry. 
Hence, the restriction to $V$ of the ratio $G(M)=\frac{\Psi(M)}{<AM,M>^n}$ is also invariant. Therefore, 
the  restriction to $\pi(V\setminus\{0\})=T_C\alpha^*$ of the  function $G$ treated as a rational function on $\cp^2$ 
is invariant under the  projectivized pseudo-symmetry, which coincides with the $\ii$-angular symmetry centered 
at $C$, by Proposition \ref{projps}. 
The equality $G|_{\alpha^*}\equiv const$ holds since the derivative of the function $G$ at $C$ along  
a vector tangent to $T_C\alpha^*$ vanishes. Indeed,  the 
function $G|_{T_C\alpha^*}$, which is invariant under a projective involution fixing  $C$, has zero derivative 
at  $C$, similarly to vanishing of derivative of an even function at 0. Statement 1) is proved. 
The proof of statement 2) is analogous. In more detail, let $\Lambda\subset\Sigma$ 
be a geodesic whose segment $I\subset \Lambda$ is contained in $\partial\Omega$. For every point $Q\in I$ 
the projective line $Q^*$ dual to $\pi(Q)$ passes through the point $\Lambda^*$ $\Sigma$-dual to $\Lambda$. The  
restriction $G|_{Q^*}$  is invariant under the $\ii$-angular symmetry with center $\Lambda^*$, 
as in the above argument. Therefore, this holds for 
the restriction of the function $G$ to every complex line through $\Lambda^*$, and hence, on 
all of $\cp^2$, by uniqueness of analytic extension. Statement 2) is proved. 
\end{proof}

  \begin{definition} \label{defang} Let $\ii\subset\cp^2$ be a conic: 
  either a regular conic, or a pair of distinct lines.   Let $\gamma\subset\cp^2$ be 
 an irreducible algebraic curve different from a line and from $\ii$. 
 We say that $\gamma$ {\it generates a rationally integrable $\ii$-angular billiard,} if there exists 
 a non-constant rational function $G$ on $\cp^2$ with poles contained in $\ii$ (called the {\it integral of the $\ii$-angular billiard}) 
 such that for every $C\in\gamma\setminus\ii$ the restriction of the function $G$ to the projective tangent line 
 $T_C\gamma$ is invariant under the $\ii$-angular symmetry with center $C$. 
 \end{definition}

\begin{corollary} \label{iang} Let $\ii\subset\cp^2_{(M_1:M_2:M_3)}$ be the absolute, see (\ref{absol}). Let 
$\Omega\subset\Sigma$ 
be a polynomially integrable billiard with a non-trivial homogeneous integral $\Psi(M)$ of even degree $2n$. 
Let $\alpha\subset\partial\Omega$ be a  $C^2$-smooth arc with non-zero 
geodesic curvature,  and let 
$\alpha^*\subset\rp^2\subset\cp^2$ be its $\Sigma$-dual curve. The 
complex projective Zariski closure of the curve $\alpha^*$ is an algebraic curve. 
Each its non-linear irreducible 
component   generates a  rationally integrable 
$\ii$-angular billiard with integral $G(M)=\frac{\Psi(M)}{<AM,M>^n}$, see (\ref{ratg}). 
\end{corollary}

\begin{proof} The function $G$ is non-constant on $\cp^2$, since $\Psi|_{\{<AM,M>=1\}}\not\equiv const$: 
the latter statement follows from  non-constance of the function $\Psi([r,v])$ on the hypersurface $\{<Av,v>=1\}$ 
(non-triviality of the integral) and Proposition \ref{vectis}.  The first statement of the corollary, which  follows from Bolotin's 
Theorem \ref{algnew}, also follows from constance of the function $G$ on $\alpha^*$, 
see Statement 1) of Theorem \ref{ext}.  Its second statement follows from the invariance 
of the function $G$ in Statement 1) of Theorem \ref{ext}  by straightforward 
analytic extension argument.
\end{proof}

 \begin{proposition} \label{propc} Let an irreducible algebraic curve $\gamma\subset\cp^2$ generate a rationally integrable $\ii$-angular 
 billiard with the integral $G$. Then $G|_{\gamma}\equiv const$. 
 \end{proposition}
 
 The proof of the proposition repeats literally the above proof of the analogous statement from 
 Theorem  \ref{ext}, part 1).

 \subsection{Duality and $\ii$-angular billiards. Proof of Theorem \ref{bol2}}
 
For the proof of Theorem \ref{bol2}  
we use the well-known classical properties of the  orthogonal polarity given by the following proposition and its corollary. 
We present the proof of the proposition for completeness of presentation. 

\begin{proposition} \label{classb} Let $B$ be a non-degenerate complex symmetric $3\times 3$-matrix. Consider the complex 
space $\cc^3$ with coordinates $x=(x_1,x_2,x_3)$ 
equipped with the complex-bilinear Euclidean quadratic form $dx_1^2+dx_2^2+dx_3^2$.  
The complex orthogonal-polar-dual to 
the conic in $\cp^{2}_{(x_1:x_2:x_3)}$ given by the equation $<Bx,x>=0$ is the conic given by the equation $<B^{-1}x,x>=0$. 
\end{proposition}

\begin{proof} Consider the cone $K=\{ x\in\cc^3\setminus\{0\} \ | \ <Bx,x>=0\}$ and its tautological projection  $\Gamma=\pi(K)\subset\cp^2$, 
which is the  conic under consideration. Let $x\in K$. The projective tangent  line $L=T_{\pi(x)}\Gamma$  is defined by the tangent plane  
$T_xK$ considered as a vector subspace in $\cc^3$. It follows from definition that $T_xK$ consists 
of those vectors $v$ for which $<Bx,v>=0$. Thus,  $(T_xK)^{\perp}=\cc(Bx)$, 
and the  dual $L^*$ is $\pi(Bx)$. Therefore, the dual  $\Gamma^*$ is the projection 
 $\pi(B(K))$, which  is obviously defined by the equation $<B(B^{-1}y),B^{-1}y>=<B^{-1}y,y>=0$. 
This proves the proposition.
\end{proof}

\begin{definition} \label{deuclid} \cite[p.84]{veselov2}. Let $A, B$ be two real 
non-proportional  symmetric $3\times 3$-matrices. They define a 
 {\it pseudo-Euclidean pencil} of conics in $\rp^2$: the conics given by the equation
 $$\{<(B-\la A)M,M>=0\}\subset\rp^2_{(M_1:M_2:M_3)},  \ \la\in\rr.$$
 The same pencil of complex conics in $\cp^2$  depending on $\la\in\cc$ will be also called pseudo-Euclidean. 
\end{definition}

\begin{corollary} \label{self} The $\Sigma$-duality 
transforms every confocal pencil of conics to the corresponding pseudo-Euclidean pencil. Namely, for every real 
symmetric $3\times 3$-matrix $B$ satisfying the conditions of Definition \ref{pencil} 
for any two conics in 
$\Sigma$ lying in the  confocal pencil (\ref{confb}) defined by $B$  
their $\Sigma$-dual curves lie in conics belonging to the pseudo-Euclidean pencil defined by  the same matrix $B$. In the non-Euclidean case, when 
the absolute $\ii$  is a regular conic, $\ii$ is self-dual with respect to complex orthogonal polarity.  
\end{corollary}

The first statement of the corollary is obvious. The self-duality 
follows from Proposition \ref{classb} and  involutivity: 
$A^2=Id$ in the non-Euclidean case. 
 
 \begin{proof} {\bf of Theorem \ref{bol2}.} Let $\Omega\subset\Sigma$ be a polynomially integrable billiard. 
 Let $\Psi(M_1,M_2,M_3)$ be a non-trivial homogeneous 
 polynomial integral of the billiard $\Omega$ of even degree $2n$. Consider the affine chart $M_3\neq0$ on 
 $\cp^2_{(M_1:M_2:M_3)}$ with  coordinates $(x,y)$:  $x=\frac{M_1}{M_3}$, $y=\frac{M_2}{M_3}$. Set 
 $$\mcq(x,y)=<AM,M>, \text{ where } M=(x,y,1):$$
 $$\mcq(x,y)=x^2+y^2 \text{ in the Euclidean case; otherwise } \mcq(x,y)=x^2+y^2\pm1.$$
 In this  affine chart  the function $G$ on $\cp^2$  from (\ref{ratg}) takes the form
 $$G(x,y)=\frac{F(x,y)}{(\mcq(x,y))^n}, \  F(x,y)=\Psi(x,y,1), \ \deg F\leq2n.$$
 In what follows for every conic $\alpha\subset\Sigma$ the corresponding complex conic containing its $\Sigma$-dual 
 $\alpha^*$ will be denoted by $\wt\alpha^*$. 
 
Let the boundary $\partial\Omega$ contain an arc of a conic $\alpha$. 
 Let $\mcc$ denote the confocal conic pencil containing $\alpha$, and let $\mcc^*$ denote the corresponding 
 ($\Sigma$-dual) pseudo-Euclidean pencil of conics containing $\wt\alpha^*$: 
 $$\kla=\{ <B_{\la}X,X>=0\}\subset\rr^3_{(X_1,X_2,X_3)}, \ B_{\la}=(B-\la A)^{-1},\  \mcc_{\la}=\kla\cap\Sigma;$$ 
 $$ \kla^*=\{<(B-\la A)M,M>=0\}\subset\cc^3_{(M_1,M_2,M_3)}, \ 
\mcc^*_{\la}=\pi(\kla^*\setminus\{0\})\subset\cp^2,$$
$$\kappa^*_{\infty}=\widehat\ii=\{<AM,M>=0\}\subset\cc^3, \ \mcc^*_{\infty}=\pi(\kappa^*_{\infty}\setminus\{0\})=\ii.$$

  \medskip

 {\bf Claim 1.} {\it Each  $C^2$-smooth  arc  of the boundary $\partial\Omega$ with non-zero geodesic curvature
 lies in a conic confocal to $\alpha$.}

 \begin{proof} The conic $\wt\alpha^*$ generates a rationally integrable $\ii$-angular billiard with integral $G$, 
 by Corollary \ref{iang}. On the other hand, it is known that the billiard on a conic $\alpha$ 
 admits a non-trivial quadratic homogeneous first integral $\wt\Phi=\wt\Phi(M)$, see 
 \cite[proposition 1]{bolotin2}. Set 
 $$\wt F(x,y)=\wt\Phi(x,y,1), \ \wt G(x,y)=\frac{\wt F(x,y)}{\mcq(x,y)}.$$
   
{\bf Claim 2.} {\it The level curves of the function $\wt G$ are conics from the pencil $\mcc^*$, 
and the function $G$ is constant on each of them.}

\begin{proof} For every conic $\beta$ 
 confocal to $\alpha$ the quadratic  integral $\wt\Phi$ is also an integral for the billiard on the conic $\beta$. This is 
 well-known, see \cite{bolotin2}, and 
 follows from the explicit formula \cite[formula (12)]{bolotin2} for the quadratic integral. Therefore, both corresponding 
 complexified dual  conics $\wt\alpha^*$ 
 and $\wt\beta^*$ generate rationally integrable $\ii$-angular billiards with a common quadratic rational integral $\wt G$ having first order 
 pole at $\ii$, by Corollary \ref{iang}. Hence, 
 $\wt G$ is constant on $\wt\alpha^*$ and $\wt\beta^*$, by Proposition \ref{propc}. Thus, the integral $\wt G$ is 
 constant on every conic from the complex pseudo-Euclidean pensil $\mcc^*$, since the above conics $\wt\beta^*$ with 
 $\beta$ being confocal to $\alpha$ 
 form a real one-dimensional subfamily in  $\mcc^*$.  Let us normalize the integral $\wt G$ by additive constant (or equivalently, the integral $\wt\Phi$ by addition of 
 $c<AM,M>$, $c=const$) 
so that  $\wt G|_{\wt\alpha^*}\equiv0$. After this normalization one has $\wt F|_{\wt\alpha^*}\equiv0$: that is, 
 $\wt F$ is the quadratic polynomial defining the conic $\wt\alpha^*$. On the other 
 hand, $\wt\alpha^*$ generates a rationally integrable $\ii$-angular billiard with integral 
 $G$ (Corollary \ref{iang}). Hence,  $G|_{\wt\alpha^*}\equiv c_1=const$, 
  by Proposition \ref{propc}. Therefore, 
 $$G(x,y)=c_1+G_1(x,y)\wt G(x,y),$$
 $$G_1(x,y)=\frac{f_1(x,y)}{(\mcq(x,y))^{n-1}}, \  \deg f_1\leq 2n-2.$$
 Hence, the fraction $G_1$ is also a rational integral  of the $\ii$-angular billiard generated by $\wt\alpha^*$, 
 as are $G$ and $\wt G$. Thus, $G_1|_{\wt\alpha^*}\equiv c_2= const$, by 
 Proposition \ref{propc}. Similarly we get that 
 $$G_1(x,y)=c_2+G_2(x,y)\wt G(x,y), \ G_2(x,y)=\frac{f_2(x,y)}{(\mcq(x,y))^{n-2}},$$
 $\deg f_2\leq 2n-4$, and $G_2$ is an integral of the $\ii$-angular billiard generated by $\wt\alpha^*$, as are $G_1$ and 
 $\wt G$. 
 Continuing this prodecure we get that $G$ is a polynomial in $\wt G$. Hence, $G\equiv const$ on the level curves of the 
 function $\wt G$, that is, on the conics from the pencil $\mcc^*$. 
 Claim 2 is proved.
 \end{proof}
 
 Let $\phi$ be a $C^2$-smooth arc in $\partial\Omega$ with non-zero geodesic curvature, 
 and let $\phi^*\subset\rp^2\subset\cp^2$ 
 denote its $\Sigma$-dual curve. 
 The curve $\phi^*$ lies in a level curve of the function $G$, by Theorem \ref{ext}, statement 1). Hence, it lies in 
 a finite union of  conics from the pencil $\mcc^*$, since each level curve of the function $G$ is a finite union of conics in  $\mcc^*$  (follows from Claim 2). Therefore, $\phi$ lies in just one conic confocal to $\alpha$, by 
 smoothness, since any two intersecting confocal conics are orthogonal. 
 This proves Claim 1.
 \end{proof}

Now it remains to show that if $\partial\Omega$ contains geodesic segments, then their ambient geodesics are admissible 
with respect to the pencil $\mcc$, see Definition \ref{adm}. As it is shown below,  this is implied by the following proposition.

\begin{proposition} \label{pis} Let $B$ be a real symmetric $3\times3$-matrix as in Definition \ref{pencil}. 
Let $\mathcal C$ denote the 
corresponding pencil (\ref{confb}) of confocal conics in $\Sigma$. The corresponding admissible 
geodesics in $\Sigma$ from Definition \ref{adm} are exactly those geodesics $\widehat l$, for which the  symmetry of the surface 
$\Sigma$ with respect to $\widehat l$  leaves the pencil $\mathcal C$ invariant: the symmetry permutes confocal conics. 
Or equivalently,  the geodesics $\widehat l$ for which the 
 $\ii$-angular symmetry with center $\widehat l^*$ $\Sigma$-dual to $\widehat l$ leaves the $\Sigma$-dual 
 pseudo-Euclidean pencil $\mcc^*$ invariant. 
\end{proposition}

\begin{remark} We will be using only the second  statement of Proposition \ref{pis} characterizing 
admissible geodesics $\widehat l$ in terms of $\ii$-angular symmetry with center $\widehat l^*$ of the pencil $\mcc^*$. 
Their characterization in terms of symmetry of the pencil $\mcc$ will be proved just for completeness of presentation.
\end{remark} 

\begin{proof} {\bf of Proposition \ref{pis}.} Let us first prove that for every given geodesic $\widehat l\subset\Sigma$ the  two statements 
of the proposition are indeed equivalent. As it is shown below, this is implied by the following proposition.

\begin{proposition} \label{symequ} Consider the action of the symmetry with respect to a given geodesic $\widehat l\subset\Sigma$ 
on the space of all the geodesics in $\Sigma$. The $\Sigma$-duality conjugates this action to the 
$\ii$-angular symmetry with center $\widehat l^*$. 
\end{proposition}

\begin{proof} It suffices to prove the above conjugacy on the space of those geodesics that intersect $\widehat l$, 
by analyticity and since they form an open subset in the connected manifold of geodesics. Each geodesic 
through a point $r\in\widehat l$ is uniquely determined by its tangent line: a one-dimensional subspace 
$\Lambda\subset T_r\Sigma$. Thus, it suffices to show that the $\Sigma$-duality conjugates the 
symmetry action on the projectized tangent  plane $\mathbb{P}(T_r\Sigma)$ with the $\ii$-angular symmetry 
centered at $\widehat l^*$. Indeed, the $\Sigma$-duality sends each one-dimensional subspace 
$\Lambda\subset T_r\Sigma$ to the point $\widehat\Lambda^*\in\rp^2$ represented by the one-dimensional vector subspace 
 $\Lambda^r\subset\rr^3$ orthogonal to both $r$ and $\Lambda$ (see Convention \ref{conven}). 
The linear isomorphism $\mcm_r:T_r\Sigma\to V_r=r^{\perp}$, $v\mapsto[r,v]$ sends each subspace 
$\Lambda$  to $\Lambda^r$ and conjugates the pseudo-symmetries with respect to the lines 
$T_r\widehat l\subset T_r\Sigma$ and $(T_r\widehat l)^r\subset V_r$, by definition and Corollary \ref{cve}. 
Therefore, its projectivization realizes the $\Sigma$-duality $\mathbb{P}(T_r\Sigma)\to\mathbb P(V_r)$ 
  and conjugates the action of the symmetry with 
respect to the line $T_r\widehat l$ on the source with the projectivized pseudo-symmetry of the 
image:  the $\ii$-angular symmetry with center $\widehat l^*=\pi((T_r\widehat l)^r)$ (Proposition \ref{projps}). Proposition \ref{symequ} is proved. 
\end{proof}

Note that for every curve $\gamma\subset\Sigma$ the $\Sigma$-duality sends the 
family of geodesics tangent to $\gamma$ to the $\Sigma$-dual curve $\gamma^*$  (see Convention \ref{conven}). 
This together with the 
above proposition implies equivalence of the two statements of Proposition \ref{pis}. 
Thus, it suffices to prove its second statement: those geodesics $\widehat l$,  for which 
the pseudo-Euclidean pencil $\mcc^*$ is invariant under the $\ii$-angular symmetry with center $\widehat l^*$,  
are exactly the admissible geodesics from Definition \ref{adm}.   

Fix a geodesic $\widehat l$. Let $H\subset\rr^3$ denote the  two-dimensional vector subspace
 containing $\widehat l$. Fix a vector $a\in H^{\perp}\subset\rr^3$, $a\neq0$. It  represents the $\Sigma$-dual 
$\widehat l^*= \pi(a)$. 
The vector $a$ lies in a unique cone $\kla^*$ with 
$\la\neq\infty$, since $<Aa,a>\neq0$: otherwise, if $<Aa,a>=0$, then the intersection $\widehat l=H\cap\Sigma$ would be empty. 
Indeed, in the Euclidean case the equality $<Aa,a>=0$ on a real vector $a$ 
holds exactly when $a$ lies in the $x_3$-axis; then 
$H$ is parallel to the plane $\Sigma$, $H\cap\Sigma=\emptyset$. 
In the non-Euclidean case the equality $<Aa,a>=0$ implies that $A=\diag(1,1,-1)$ and 
the projective line  $a^*=\pi(H\setminus\{0\})$ is tangent to the real absolute 
$\{<Ax,x>=0\}\subset\rp^2_{(x_1:x_2:x_3)}$, by self-duality (Corollary \ref{self}). Then 
 $H$ is tangent to the cone $\{<Ax,x>=0\}\subset\rr^3$ and hence, it is disjoint from the inner component containing $\Sigma$ of 
 the  complement of the latter cone.  Thus, $H\cap\Sigma=\emptyset$, -- a contradiction.

Without loss of generality we  will consider that $a\in\kappa_0^*$, after replacing $B$ by $B-\la A$ 
for appropriate $\la$, by the inequality $<Aa,a>\neq0$. Let $S:\cc^3\to\cc^3$ denote the pseudo-symmetry with respect to the line 
$\cc a$. 

\medskip
{\bf Claim 3.} {\it The pseudo-Euclidean pencil $\mcc^*$ is invariant under the $\ii$-angular symmetry with center 
$\widehat l^*$, if and only if $S(\kappa_0^*)=\kappa_0^*$.}

\begin{proof} The above $\ii$-angular symmetry is the projectivization of the pseudo-symmetry $S$. 
Therefore,   invariance of the pencil $\mcc^*$ under the $\ii$-angular symmetry is equivalent to the 
$S$-invariance of the family of cones $\kla^*$, 
that is, to the existence of an involution $h:\la\to h(\la)$ such that $S(\kla^*)=\kappa_{h(\la)}^*$. 
In the latter case one has $S(\kappa_0^*)=\kappa_0^*$, since $S(a)=a$, $a\in\kappa_0^*$ 
and $a\notin\kla^*$ for every $\la\neq0$. Conversely, let $S(\kappa_0^*)=\kappa_0^*$. 
This means that the involution $S$ sends the quadratic form $<Bx,x>$ to itself up to sign. Hence, $S(\kla^*)=\kappa_{\pm\la}^*$ 
for every $\la$, since $S$ preserves the quadratic form $<AX,X>$. 
This together with the previous equivalence statement proves the claim.
\end{proof}

{\bf Claim 4.} {\it One has $S(\kappa_0^*)=\kappa_0^*$, if and only if $\kappa_0^*$ is a union of a pair of 
2-planes through the origin 
  in $\cc^3$ that has one of the following types: 
  
 $\alpha$) both planes contain the line $\cc a$ (they may coincide);
 
 $\beta$) one plane in $\kappa_0^*$ contains the line $\cc a$, and the other plane coincides with the two-dimensional 
 subspace $H_A\subset\cc^3$ that is orthogonal to the vector $a$ with respect to the scalar product $<Ax,x>$.}

 \begin{proof} Every hyperplane $W\subset\cc^3$ parallel to the plane $H_A$ is $S$-invariant, and $S$ acts there as 
 the central symmetry with respect to the point $C_W$ of intersection $W\cap(\cc a)$. The $S$-invariance of the cone $\kappa_0^*$ 
 is equivalent to the invariance of each intersection $I_W=W\cap\kappa_0^*$ under the latter symmetry  for every $W$ as above. The  intersection  $I_W$  is either all of $W$, or 
 a line through $C_W$, or a conic in $W$ containing the center of its symmetry $C_W$, since $\cc a\subset\kappa_0^*$. 
 In the latter case $I_W$ is a union of 
 two  lines through $C_W$, since a planar conic central-symmetric with respect to some its point $C$ is 
 a union of two lines through $C$ (the lines under question may coincide). Note that 
 all the intersections $I_W$ with  $W\neq H_A$ are naturally isomorphic between 
 themselves via homotheties centered at the origin, since $\kappa_0^*$ is a cone. 
 Therefore, the following two cases are possible.
 
 $\alpha$) $I_W$ is a union of two  (may be coinciding) lines through $C_W$ for every $W$; then $\kappa_0^*$ 
 is a union of two planes containing the line $\cc a$.
 
 $\beta$) $I_W$ is a line for all $W\neq H_A$,  and $I_W=W$ for $W=H_A$; then $\kappa_0^*$ 
 is a union of the plane $H_A$ and another plane containing $\cc a$.
 
This proves the claim. 
\end{proof}

Now let us return to the proof of Proposition \ref{pis}. Let the pencil $\mcc^*$ be invariant under the $\ii$-angular symmetry 
centered at $\widehat l^*$; or equivalently, the cone $\kappa_0^*=\{<Bx,x>=0\}$ be a union of two planes, as in Claim 4. 

Case $\alpha$). The above planes both contain $a$, thus 
 $a\in \ker B$; $dim (\ker B)=1$, if the planes are distinct; $dim (\ker B)=2$, if they coincide. 
  Hence, the hyperplane $H$ orthogonal to $a$ with respect to the standard Euclidean scalar product 
  is orthogonal to $\ker B$. Therefore, the geodesic $\widehat l=H\cap\Sigma$ is admissible of  type 1) in Definition \ref{adm}. 
  Vice versa, each admissible geodesic of type 1) 
can be represented as above after replacing $B$ by $B-\la A$.

Case $\beta$). Then the cone $\kappa_0^*$ is the union of the plane $H_A$ and a plane $\Pi$ containing the line $\cc a$. The 
plane $\Pi$ is the complexification of a real plane, which will be here also denoted by $\Pi$, since $\kappa_0^*$ is defined by a 
quadratic equation over real numbers and $H_A$ is the complexification of a real plane. 
Let $b\in\rr^3\setminus\{0\}$ denote a  vector Euclidean-orthogonal to $\Pi$. Thus,  $<a,b>=0$. Note that the vector $Aa$ is non-zero, 
since $<Aa,a>\neq0$, as was shown above, and it is Euclidean-orthogonal to  $H_A$, by definition. Therefore, 
$<BM,M>=c<Aa,M><b,M>$, $c\in\rr\setminus\{0\}$. Let us normalize the vectors $a$ and $b$ by constant factors so that 
$c=2$. Then the quadratic form $<BM,M>$ can be represented in the tensor form as $Aa\otimes b+b\otimes Aa$. 
The plane $H$ defining the geodesic $\widehat l$ is the plane orthogonal to the vector $a$, by definition. Hence, $\widehat l$ is an admissible 
geodesic of type 2) in Definition \ref{adm}: the first geodesic in (\ref{newgeod}). Vice versa, each geodesic of type 2) 
can be represented as above, see Remark \ref{permgeod}. 
Proposition \ref{pis} is proved.
\end{proof}

Let now $\widehat l\subset\Sigma$ be a geodesic whose some segment is contained in the boundary of the polynomially integrable billiard 
under question. 
The $\ii$-angular symmetry with center $\widehat l^*$  leaves invariant the rational integral $G$, 
by Theorem \ref{ext}. Hence, it permutes the level curves of the 
quadratic rational function $\wt G$, and the pencil $\mcc^*$ is invariant, by  Claim 2. 
Thus, the geodesic $\widehat l$ is admissible, by Proposition \ref{pis}. Theorem \ref{bol2} is proved.
\end{proof}
  
\section{Bialy--Mironov Hessian Formula and  asymptotics of Hessians}
The material of the present section will be used in Section 4 in the proof of Theorem \ref{iii}, statement (ii-b). It includes: 

-   Bialy--Mironov Hessian Formula (\ref{hesf}) recalled in Subsection 3.1;

- the asymptotics of its left- and right-hand sides along those local branches of the curve $\gamma$ that are 
transversal to $\ii$  (Subsection 3.4).

In the proof of the above asymptotics 
we use general asymptotic formulas 

- for the defining function of an irreducible  germ $a$ of analytic curve along another irreducible germ $b$  (Subsection 3.2); 

- for the Hessian $H(f)$ of defining function of a given germ $b$ along $b$ (Subsection 3.3). 

\subsection{Bialy--Mironov formula}
Let $\gamma\subset\cp^2$ be an irreducible algebraic curve generating a rationally integrable 
$\ii$-angular billiard with integral $G$. The function $G$ has poles contained in $\ii$ and is constant 
on $\gamma$, by Proposition \ref{propc}. 
In what follows we normalize it so that $G|_{\gamma}\equiv0$, and set 
$$\Gamma=\{ G=0\}\supset\gamma.$$

Fix an affine chart  $\cc^2\subset\cp^2$ with  coordinates $(x,y)$ such that the infinity line is not contained in $\ii$.
 In this chart the function $G$ takes the form
 $$G(x,y)=\frac{F_1(x,y)}{(\mcq(x,y))^n}, \text{ where } F_1 \text{ is a  polynomial of degree at most } 2n,$$
 $$\mcq(x,y) \text{ is a fixed quadratic polynomial defining } \ii: \ \ii=\{\mcq=0\}.$$
Let  $f(x,y)$ be the polynomial defining the curve $\gamma$, which is  irreducible, as is $\gamma$: 
 $\gamma=\{ f=0\}$, the differential $df$ being non-zero on a Zariski open subset in $\gamma$. Recall that the polynomial $F_1$ vanishes on $\gamma$. Therefore,  
 \begin{equation}F_1=f^kg_1, \ k\in\nn, \ g_1 \text{ is a polynomial coprime with } f.\label{fkg}
 \end{equation}
 Set
 \begin{equation} g=g_1^{\frac1k}, \ F=F_1^{\frac1k}=fg, \ m=\frac nk.\label{gff}\end{equation}
 We consider the Hessian quadratic form of the function $f(x,y)$ evaluated on appropriately normalized tangent vector to $\gamma=\{ f=0\}$ at a  
 point $(x,y)$, namely, the skew gradient $(f_y,-f_x)$ with respect to the standard complex symplectic form $dx\wedge dy$:  
 \begin{equation} H(f)=f_{xx}f_y^2-2f_{xy}f_xf_y+f_{yy}f_x^2.\label{hess}
 \end{equation}

 \begin{theorem}  \label{bmf} (see, \cite[theorem 6.1]{bm}, \cite[formulas (16) and (32)]{bm2}) The following formula 
 holds for all $(x,y)\in\gamma$:
\begin{equation}g^3(x,y)H(f)(x,y)=H(gf)=c(\mcq(x,y))^{3m-3}, \ c\equiv const\neq0.\label{hesf}\end{equation}
\end{theorem}

\begin{remark} In 2008 S.Tabachnikov obtained a version of formula (\ref{hesf}) with $k=1$ and constant right-hand side
for polynomially integrable outer billiards satisfying some non-degeneracy assumptions \cite[p. 102]{tab08}. 
Theorem 6.1 in \cite{bm} deals with a polynomially integrable planar billiard $\Omega\subset\rr^2$, a curve 
$\Gamma_1\subset\rr^2$ that is polar-dual to a $C^2$-smooth arc in $\partial\Omega$ 
with non-zero geodesic curvature, and the absolute $\ii=\{ x^2+y^2=0\}$.  It states that formula (\ref{hesf}) holds along the curve $\Gamma_1$. Then it holds automatically on every irreducible component 
$\gamma$ of its complex Zariski closure. Its proof given in \cite{bm} remains valid for every irreducible algebraic curve 
$\gamma$ generating a rationally integrable $\ii$-angular billiard. The same remark concerns formulas (16) and (32) from the paper 
\cite{bm2}, which deal with the non-Euclidean case and the corresponding absolute $\ii=\{ x^2+y^2\pm1=0\}$. 
These results from \cite{bm, bm2} together cover Theorem \ref{bmf} in the general case, since every conic different from 
a double line is projectively equivalent to some of the above absolutes. 
\end{remark}

 Without loss of generality we  will 
consider that $G$ is an irreducible fraction, that is, its nominator $F_1(x,y)$ does not vanish identically on $\ii$ in the case, when 
$\ii$ is regular, 
and in case, when $\ii$ is a union of two lines $\La_1$ and $\La_2$, one 
has $F_1\not\equiv0$ on each $\Lambda_j$. In the former case 
we can do this, by irreducibility of the conic $\ii$: if $F_1$ vanishes on $\ii$ with a certain multiplicity $s$, then 
we can divide both nominator and denominator in $G$ by $(\mcq(x,y))^s$ and achieve the desired property. 
In the latter case we can do this, by the fact that both lines $\La_1$ and 
$\La_2$  forming $\ii$ enter the divisor of the function $G$ (the zero-pole divisor) with the same multiplicity. 
Indeed, for every $u\in\gamma\setminus\ii$ 
the tangent line $T_u\gamma$ intersects both lines $\La_1$ and $\La_2$, and their intersection points 
with the line $T_u\gamma$ are permuted by its $\ii$-angular symmetry with center $u$, by definition. 
Both intersection points enter the divisor of the function 
$G|_{T_u\gamma}$ with the same multiplicity, by its invariance under the $\ii$-angular symmetry. This implies the above 
statement on coincidence of multiplicities of the lines $\La_1$ and $\La_2$. 

The above discussion implies that $G$ has pole along each irreducible component 
of the conic  
$\ii$. Therefore, no component in $\ii$ is contained in $\Gamma$. 
We choose the above affine chart $\cc^2_{(x,y)}$ 
so that the finite intersection $\Gamma\cap\ii$ lies in $\cc^2$, in particular, $G\not\equiv0$ on the infinity line, hence $\deg F_1=2n$. 
Let  $\Delta$ denote the zero divisor of the function $G$. 
Finally, in our assumptions made without 
loss of generality one has $F_1\not\equiv0$ on every irreducible component of the conic $\ii$, 
\begin{equation*} \Gamma=\{ F_1=0\}, \ \deg F_1=2n,\end{equation*}
\begin{equation}\Delta \text{ is the zero divisor of the polynomial } F_1,\label{deldeg}
\end{equation}
{\it the intersection $\Gamma\cap\ii$, and hence, $\gamma\cap\ii$ 
 lie in the affine chart $\cc^2_{(x,y)}$.} 

\subsection{Asymptotics of defining function} 

\begin{definition} Let $b$ be a non-linear irreducible germ of analytic curve at a point $C\in\cp^2$. An {\it adapted system 
of coordinates} to $b$ is a system of affine coordinates $(z,w)$ centered at $C$ such that the $z$-axis is 
tangent to $b$. In adapted coordinates the germ 
$b$ can be locally holomorphically and bijectively parametrized by 
small complex parameter $t$: 
 \begin{equation}
t\mapsto(t^{q},ct^{p}(1+o(1))), \ \text { as } t\to0; \  \ q,p\in\nn,  \ \ 1\leq q<p, \ \ c\neq0,\label{curve} \end{equation}
$$q=q_b, \ p=p_b, \  \ c=c_b,$$
$$q=1, \ \text{ if and only if } b \text{ is a regular germ.}$$
The {\it projective
Puiseux exponent} \cite[p. 250, definition 2.9]{alg} of the germ $b$ is the ratio
$$r=r_b=\frac{p_b}{q_b}.$$
The germ $b$ is called {\it quadratic}, if $r_b=2$, and is called
{\it subquadratic}, if $r_b\le2$, see \cite[definition 3.5]{gs}.
In the case, when $b$ is a germ of line, it is parametrized by $t\mapsto(t,0)$; then  
we set $q_b=1$, 
$p_b=\infty$, and put the Puiseux exponent $r_b$ to be equal to infinity, as in loc. cit.  
\end{definition}

\begin{proposition} \label{asfb} Let $a$, $b$ be  irreducible germs of holomorphic curves at a point $C\in\cc^2$, 
let $b$ be non-linear. 
Let $f_a$, $f_b$ be the irreducible 
germs of holomorphic functions defining them: $g=\{ f_g=0\}$ for $g=a,b$.  
Set 
\begin{equation}\rho_a=\begin{cases} 1,\text{ if } a \text{ is transversal to } b \\
r_a, \text{  if } a \text{ is tangent to } b\end{cases}.\label{roa}\end{equation}
Let $(z,w)$ be affine coordinates centered at 
$C$ that are adapted to $b$. One has 
\begin{equation}f_a(u)=O((z(u))^{q_{a}\min\{\rho_a,r_b\}}), \text{ as } u\in b \text{ tends to } C. \label{fopa}\end{equation}
\end{proposition}

The proof of Proposition \ref{asfb} is based on the following property of Newton diagram of irreducible germs of analytic curves. 

 \begin{proposition} \label{newdiag} Let $b\subset\cp^2$ be a non-linear irreducible germ of analytic curve at a point $C$, 
 and let $(z,w)$ be local affine coordinates adapted to it. 
 Let $t\mapsto(t^q,ct^p(1+o(1)))$ be its local parametrization: $1\leq q<p$, $c\neq0$,  see (\ref{curve}).  
Let  $f$ be an irreducible germ of analytic function at $C$ defining $b$: $b=\{ f=0\}$.  The Newton diagram of the  function $f$ 
 consists of one edge: the segment connecting the points $(p,0)$ and $(0,q)$. More precisely, the Taylor series of the 
 function $f(z,w)$ contains only monomials $z^{\alpha}w^{\beta}$ such that 
 \begin{equation}\nu_{\alpha\beta}=q\alpha+p\beta\geq qp.\label{nuab}\end{equation}
  \end{proposition}
 
 \begin{proof} 
 Without loss of generality we will consider that $f$ is a Weierstrass polynomial: 
 \begin{equation}f(z,w)=\phi_{z}(w)=w^d+h_1(z)w^{d-1}+\dots+h_d(z), \ \ h_j(0)=0,\label{wei}\end{equation}
 since each germ of holomorphic function at 0 that vanishes at $0$ and 
 does not vanish identically on the $w$-axis is the product of a unique polynomial 
 as above (called Weierstrass polynomial) and a non-zero holomorphic function, by Weierstrass Preparatory Theorem 
 \cite[chapter 0, section 1]{grh}. For every $z$ small enough the polynomial $\phi_{z}(w)=f(z,w)$ has $q$ roots $\zeta_l(z)$,  
 $l=1,\dots,q$:  $\zeta_l(z)=ct_l^p(1+o(1))$, $t_l^q=z$, as $z\to0$; thus, 
 $\zeta_l(z)\simeq cz^{\frac pq}$. 
  This implies that the  Weierstrass polynomial (\ref{wei})  is  
 the product of $q$ factors $w-\zeta_l(z)$ with 
 $\zeta_l(z)\simeq cz^{\frac pq}$, as $z\to0$. 
Hence, in formula (\ref{wei}) one has $d=q$, 
$$h_q(z)=(-1)^q\prod_{l=1}^q\zeta_l(z)=(-1)^{q+p(q+1)}c^qz^p(1+o(1)).$$ 
The latter equality follows from the equality $\prod_{l=1}^qt_l=(-1)^{q+1}z$: the product of $q$-th roots of unity equals to 
$(-1)^{q+1}$. One has 
 \begin{equation} h_s(z)=O(z^{\frac pqs}) \text{ for } 1\leq s<q, \text{ as } z\to0,
 \label{sst}\end{equation}
  since $h_s(z)=(-1)^s\sigma_s$, where $\sigma_s$ is the 
  $s$-th elementary symmetric polynomial 
 in the roots $\zeta_l(z)\simeq cz^{\frac pq}$. 
Formula (\ref{sst}) implies that  the Taylor series of the Weierstrass polynomial (\ref{wei}) contains only the monomials $w^q$, $z^p$  and those monomials 
$z^{\alpha}w^{\beta}$ for which  $\beta< q$ (set $s=q-\beta$) and  
$\alpha\geq\frac pqs=\frac pq(q-\beta)$, i.e., $q\alpha+p\beta\geq pq$. This proves the proposition.
 \end{proof}

 \begin{proof} {\bf of Proposition \ref{asfb}.} 
  Case 1): the curve $a$ is transversal to $b$. Then $\rho_a=1<r=r_b=\frac{p_b}{q_b}$, and we have to show that 
 $f_a|_b=O(z^{q_{a}})$. To do this, let us take the coordinates $(z_a,w_a)$ adapted to $a$ so that 
 the $w_a$-axis coincides with the $z$-axis $T_Cb$, $w_a=z$ on $T_Cb$ and $z_a=w$: one can do this, by transversality. 
  One has  
  \begin{equation}w_a\simeq z, \ z_a=w\simeq c_bz^r \  \text{ along the curve } b.\label{rpq1}\end{equation}  
 Hence, each Taylor monomial $z_a^{\alpha}w_a^{\beta}$ of the function $f_a$ 
 has asymptotics $O(z^{\alpha r+\beta})$ along the curve $b$. Now it suffices to show that $\alpha r+\beta\geq q_{a}$. 
 Recall that $\alpha q_{a}+\beta p_{a}\geq p_{a}q_{a}$, by (\ref{nuab}). Dividing the latter inequality by $p_{a}$ 
 yields to  $\nu=\alpha r_{a}^{-1}+\beta\geq q_{a}$. Hence, $\alpha r+\beta \geq\nu\geq q_{a}$, since $r_{a}, r>1$. 
 This proves the proposition.
 
 Case 2): the curve $a$ is tangent to $b$, thus $\rho_a=r_{a}$. 
 Then the coordinates $(z,w)$ are adapted for both curves $b$ and $a$. 
 Each Taylor monomial $z^{\alpha}w^{\beta}$ of the function $f_a(z,w)$ is asymptotic to $cz^{\nu}$, $\nu=\alpha +\beta r$, $c=const$, 
 along the curve $b$,  since $w\simeq c_bz^r$. It suffices to show that $\alpha+\beta r\geq s=q_a\min\{ r_a, r\}$. 
 
 Subcase 2a): $r_{a}\leq r$. Thus,  $s=q_ar_a=p_a$. One has 
 $\alpha +\beta r\geq \alpha +\beta r_{a}\geq p_{a}=s$, by inequality (\ref{nuab}) 
 divided by $q$.
 
 Subcase 2b): $r_{a}>r$. Thus, $\min\{\rho_a, r\}=r$, $s=q_ar$, 
 $$\frac{r_{a}}r(\alpha +\beta r)=\alpha\frac{r_{a}}r+\beta r_{a}\geq \alpha+\beta r_{a}\geq p_{a}=q_{a}r_{a},$$
 by (\ref{nuab}). Multiplying the latter inequality by $\frac r{r_{a}}$ yields to $\alpha+\beta r\geq q_{a}r=s$. Proposition  
 \ref{asfb} is proved.
 \end{proof}

\subsection{Asymptotics of Hessian of local defining function}

\begin{proposition} \label{propfhp} Let $b\subset\cp^2$ be a nonlinear irreducible germ of analytic curve at a point $C$. 
Let $f$ be the irreducible germ of its defining function, $b=\{ f=0\}$, and let $H(f)$ be its Hessian defined in (\ref{hess}) 
in some affine chart $\cc^2_{(x,y)}$ containing $C$. 
Let $(z,w)$ be an affine chart on $\cp^2$ centered at $C$ 
that is adapted to $b$: the projective line $T_Cb$ is the $z$-axis. Then 
\begin{equation} H(f)(u)=O((z(u))^{3q_br-2(r+1)}), \ r=r_b, \text{ as } u\in b \text{ tends to } C.\label{fhp}\end{equation}
\end{proposition} 

\begin{proof} 
Everywhere below by $\nabla_{skew}f=(\frac{\partial f}{\partial w}, -\frac{\partial f}{\partial z})$ 
 we denote the skew gradient with respect to the standard symplectic form $dz\wedge dw$ 
 in the coordinates $(z,w)$. It is obtained from the previous skew gradient taken with respect to the symplectic form 
 $dx\wedge dy$ by multiplication by the ratio of the above symplectic forms: the 
 Jacobian of the coordinate change $(z,w)\mapsto(x,y)$. 
  For every $u\in b$ let $L_u\subset\cc^2$ denote the affine line tangent to $b$ 
at $u$, and let $v$ denote the extension 
of the vector $\nabla_{skew}f(u)\in T_ub=T_uL_u$ to a constant vector field on $L_u$. It 
 suffices to prove formula (\ref{fhp}) for its left-hand side replaced by the  derivative $\frac{d^2f}{dv^2}(u)$: for $u\in b$ 
 the ratio of the absolute values of the latter second derivative and the expression 
 $H(f)(u)$ 
 equals to the squared modulus of the above Jacobian, which is a non-zero holomorphic function on a neighborhood of the 
 base point $C$.

 \def\wab{z^{\alpha}w^{\beta}}
 We  evaluate the Hessian 
 quadratic form of each  Taylor 
 monomial of the function $f$ on  $\nabla_{skew}f(u)$. We show that the expression thus obtained 
 has asymptotics given by the right-hand side in (\ref{fhp}). This will prove the proposition. 
 
 Let $z^{\alpha}w^{\beta}$ be the Taylor monomials of the function $f$. 
 The skew gradient $(\nabla_{skew}f)|_b$ is a linear combination of the vector monomials 
 $$h_{\alpha,\beta}=\wt h_{\alpha,\beta}\frac{\partial}{\partial w}, \ \wt h_{\alpha,\beta}=z^{\alpha-1}w^{\beta}
 \simeq 
 cz^{\alpha +\beta r-1},$$ 
 $$v_{\alpha,\beta}=\wt v_{\alpha,\beta}\frac{\partial}{\partial z}, \ \wt v_{\alpha,\beta}=
 z^{\alpha}w^{\beta-1}\simeq c'z^{\alpha +\beta r-r},  \ c,c'\neq0;$$
both above 
asymptotics are written along the curve $b$. The restrictions to the curve $b$ of the second derivatives of a monomial 
$z^{\alpha}w^{\beta}$ are asymptotic  to 
 $$\frac{\partial^2(\wab)}{\partial w^2}=\beta(\beta-1)z^{\alpha}w^{\beta-2}=O(z^{\alpha+\beta r-2r});$$
 $$\frac{\partial^2(\wab)}{\partial z^2}=\alpha(\alpha-1)z^{\alpha-2}w^{\beta}= O( z^{\alpha +\beta r-2});$$
 $$ \frac{\partial^2(\wab)}{\partial z\partial w}=\alpha\beta z^{\alpha-1}w^{\beta-1}=O(z^{\alpha +\beta r-r-1}).$$  
 Therefore, applying the Hessian quadratic form of each monomial $\wab$ to a linear combination of the vectors  $h_{\alpha',\beta'}$ 
 and $v_{\alpha',\beta'}$ 
 yields to a linear combination  of expressions of the three following types: 
$$\frac{\partial^2(\wab)}{\partial w^2}\wt h_{\alpha',\beta'}\wt h_{\alpha'',\beta''}=O(z^{\nu}), \ \nu=(\alpha'+\beta' r-1)+ 
(\alpha''+\beta'' r-1)$$
 \begin{equation}+\alpha +\beta r-2r=(\alpha'+\beta'r)+(\alpha''+\beta''r)+(\alpha +\beta r)-2(r+1);\label{fornu}\end{equation}
 $$\frac{\partial^2(\wab)}{\partial z^2}\wt v_{\alpha',\beta'}\wt v_{\alpha'',\beta''}=O(z^{\nu_2}), \ \nu_2=(\alpha'+\beta' r)
 +(\alpha''+\beta'' r)-2r+\alpha +\beta r-2=\nu;$$
 $$\frac{\partial^2(\wab)}{\partial z\partial w}\wt h_{\alpha',\beta'}\wt v_{\alpha'',\beta''}=O(z^{\nu_3}), \ \nu_3=(\alpha'+\beta' r)+
 (\alpha''+\beta'' r)
 +\alpha +\beta r-2r-2=\nu.$$
 Let us now estimate $\nu$ from below. Recall that for every Taylor monomial $\wab$ of the function $f$ one has 
 $$\alpha +\beta r=\frac1{q_b}(\alpha q_b+\beta p_b)\geq p_b=q_br,$$ 
 by (\ref{nuab}), and hence, the same inequality holds for $(\alpha',\beta')$ and $(\alpha'',\beta'')$. 
 This together with formula (\ref{fornu}) for the 
 number $\nu$ implies that $\nu\geq 3q_br-2(r+1)$. This together with the above discussion proves formula (\ref{fhp}). 
 \end{proof}
 
\subsection{ Asymptotics of Bialy--Mironov Formula} 

Everywhere below in this subsection 
$C\in\gamma\cap\ii$ is a regular point of the conic $\ii$, and $b$ is a local branch of the curve 
$\gamma$ at $C$ that is transversal to $\ii$; $(z,w)$ are affine coordinates centered at 
$C$ and adapted to $b$. Recall that $\Delta$ is the zero divisor of 
the function $G$, it coincides with the zero divisor of the polynomial $F_1$, and 
$\deg F_1=\deg(\Delta)=2n$, see (\ref{deldeg}). 

\begin{proposition} The right-hand side in (\ref{hesf}) has the following asymptotics, as $u=(x,y)\in b$ tends to $C$:  
\begin{equation} (\mcq(u))^{3m-3}\simeq c(z(u))^{3m-3}, \ c\neq0, \ m=\frac nk=\frac1{2k} \deg(\Delta).\label{righ}\end{equation}
\end{proposition}
\begin{proof}  The degree equality in (\ref{righ}) follows from definition. The restriction to $T_Cb$ of the differential 
$d\mcq(C)$ does not vanish, since $C$ is a regular point of the 
conic $\ii=\{\mcq=0\}$ and $b$ is transversal to $\ii$. Recall that the tangent line $T_Cb$ is the  $z$-axis. 
Therefore,  $\mcq(u)|_b\simeq cz(u)$, $c\neq0$, as $u\to C$. 
This implies the asymptotic formula in (\ref{righ}).
\end{proof}

Let $\sum_{j=1}^ls_jb_j$ denote the germ at $C$ of the 
divisor $\Delta$. Here $s_j\in\nn$, 
and $b_j$ are distinct irreducible germs of analytic curves in $\Delta$ at $C$ 
numerated so that $b_1=b$; thus, $s_1=k$. 
For $j=1,\dots,l$ let $f_j$ denote 
the germ at $C$ of defining function of the curve $b_j$. Set 
$$k_j=\frac{s_j}k, \ \wt g=\prod_{j=2}^lf_j^{k_j}; \ \ \ k_1=\frac{s_1}k=1; \ \ k_j=1 \text{ whenever } b_j\subset\gamma,$$
by definition. Let $F$ be the same, as in (\ref{gff}). 
\begin{proposition} Set $r=r_b$. As $u\in b$ tends to $C$, one has 
$$H(F)(u)\simeq c_1\wt g^3H(f_1)(u)=O((z(u))^{\eta}), \ c_1\neq0,$$ 
\begin{equation} \eta=\eta(b)=3\sum_{j=1}^lk_jq_{b_j}\min\{\rho_{b_j},r\}-2(r+1). 
\label{nuprec}\end{equation}
Here $\rho_{b_j}$ are the same, as in (\ref{roa}); $\rho_{b_1}=\rho_b=r$. 
\end{proposition}
\begin{proof} We use \cite[formula (17)]{bm2} valid for every two functions $f_1$ and 
$\beta$: 
\begin{equation} H(f_1(x,y)\beta(x,y))|_{\{ f_1=0\}}=\beta^3(x,y)H(f_1(x,y)).
\label{gradm}\end{equation}
One has 
 \begin{equation}  F(x,y)=h(x,y) f_1(x,y)\wt g(x,y),\label{start}\end{equation}
where $h$ is a germ of holomorphic function at $C$, $h(C)\neq0$. 
 Formula (\ref{start}) follows from definition, see (\ref{gff}). This together with (\ref{gradm}) implies that 
\begin{equation*}H(F)(u)\simeq c_1(\wt g^3H(f_1))(u)=c_1(H(f_1)\prod_{j=2}^lf_j^{3k_j})(u), \ c_1=(h(C))^3\neq0.
\end{equation*}
Substituting formula (\ref{fopa}) with $a=b_j$ and (\ref{fhp}) to the above right-hand side 
 yields to (\ref{nuprec}), taking into account that $k_1=1$ and $\rho_{b_1}=\rho_b=r$.
\end{proof}

\begin{corollary} For every local branch $b$ as at the beginning of the subsection the corresponding exponent 
$\eta=\eta(b)$ satisfies the inequality 
\begin{equation}\eta=3\sum_{j=1}^lk_jq_{b_j}\min\{\rho_{b_j},r\}-2(r+1)\leq3m-3=3\frac{\deg(\Delta)}{2k}-3.
\label{ineta}\end{equation}
\end{corollary}
\begin{proof} If the contrary inequality were true, then the left-hand side in (\ref{hesf}) would 
be asymptotically dominated by the right-hand side along the branch $b$. This follows from formulas (\ref{righ}) and (\ref{nuprec}). 
Thus obtained contradiction to formula (\ref{hesf}) proves the corollary.
\end{proof}

 \section{Local branches and relative $\ii$-angular symmetry property}
 In this section we prove the following theorem.
 \begin{theorem} \label{iii} 
 Let $\ii\subset\cp^2$ be a conic (either regular, or a pair of distinct lines). Let 
 $\gamma\subset\cp^2$ be an irreducible algebraic curve different from a line and from $\ii$ that 
generates a rationally integrable $\ii$-angular billiard. Then every intersection point $C\in\gamma\cap\ii$ satisfies the 
following statements: 

(i) Case, when $\ii$ is a union of two distinct lines through $C$. Let $b$ be a local branch of the curve $\gamma$ 
at $C$ that is transversal to both lines forming $\ii$. Then $b$ is quadratic.

(ii) Case, when $C$ is a regular point of the conic $\ii$. Then 

(ii-a) each local branch of the curve $\gamma$ at $C$ that is tangent  to $\ii$ is quadratic;

(ii-b) each its branch at $C$ that is transversal to $\ii$ is regular and quadratic.
\end{theorem}

In our assumptions for every $u\in\gamma$ the restriction to $T_u\gamma$ of the rational function $G$ 
is invariant under the $\ii$-angular symmetry with center $u$, and  $\gamma\subset\Gamma=\{ G=0\}$. This implies 
that the following {\it relative projective symmetry property} takes place: {\it for every $u\in\gamma$ 
 the intersection of the projective tangent line $T_u\gamma$ with a bigger algebraic curve $\Gamma\supset\gamma$ 
 (or a divisor)  is invariant under a projective involution $T_u\gamma\to T_u\gamma$ 
 fixing $u$:} the $\ii$-angular symmetry in our case.  
 
 In Subsection 4.4 we state  and prove 
 Theorem \ref{tquad}, which unifies and generalizes statements (i) and (ii-a) of Theorem \ref{iii}, and deduce 
 statements (i) and (ii-a). 
 Theorem \ref{tquad} is stated  for a nonlinear germ of analytic curve $b$ at $C\in\cp^2$ (that 
 needs not be algebraic) that has    
 local relative projective symmetry property with respect to a bigger finite collection $\Gamma$ of irreducible 
 germs of analytic curves at points 
 in $T_Cb$ (called a local multigerm) and projective involutions $T_ub\to T_ub$ fixing $u$ 
  with appropriate asymptotics, as $u\to C$. The formal definitions of a local multigerm 
  and the latter local symmetry property  will be given 
  in Subsections 4.1 and  4.3 respectively. 
  
  For  the proof of Theorem \ref{iii} we first describe   those points 
  of intersection $T_ub\cap\Gamma$,  whose $z$- ($w$-) coordinates in the chart $(z,w)$ adapted to $b$ 
  have asymptotics linear, sublinear and superlinear in $z(u)$ (respectively, $w(u)$), as $u\in b$ tends to $C$. 
  Their description, which mostly follows from results of \cite{alg, gs}, is presented in  Subsection 4.1.  
Then in Subsection 4.3 we show that for every local branch $b$ as in 
Theorem \ref{iii} the $\ii$-angular symmetries of 
the tangent lines $T_ub$ written in appropriate affine coordinate form families of degenerating conformal involutions 
of two possible asymptotic types A or B. The latter families of involutions 
are introduced  in Subsection 4.2, where we  
prove general Propositions \ref{invdeg} and \ref{invdeg2} on their asymptotics. 
In Subsection 4.4 we show that the collection (divisor) of asymptotic factors of points of the intersection $T_ub\cap\Gamma$ with linear asymptotics in $z(u)$ 
($w(u)$) is symmetric with respect to appropriate conformal involution $\oc\to\oc$,  
and then deduce Theorem \ref{tquad}. 

The proof of statement (ii-b) takes the rest of the section: Subsections 4.5--4.8. First in Subsection 4.5 we prove 
subquadraticity of the branch $b$ under question. In Subsection 4.6 we prove that every local branch  of the 
curve $\Gamma$ that is tangent to $b$ (if any) has Puiseux exponent no greater than $r_b$. In Subsection 4.7 
we deal with the zero divisor $\wt\Delta=\frac 1k\Delta$ of the function $F_1^{\frac1k}$, whose germ at $C$ contains $b$ 
with multiplicity 1.  We prove that its local intersection index 
 with the tangent line to $b$ at its base point $C$ is  
no less than its half-degree plus 1, and this inequality is strict, unless the germ $b$ is regular and quadratic. 
The above-mentioned Puiseux exponent and intersection  index inequalities 
will be proved in a general situation, for a germ $b$ having  local projective symmetry property, with 
the projective symmetries forming a family of involutions of type A in the adapted coordinate $z$. 

Afterwards in Subsection 4.8 we prove statement (ii-b). Namely, we show that the above-mentioned 
Puiseux exponent and  intersection index inequalities  
together would bring a contradiction to upper bound (\ref{ineta}) of the exponent 
$\eta$ in the asymptotics of Bialy--Mironov formula, unless the germ $b$ is regular and quadratic. This will finish the proof of Theorem \ref{iii}. 

\subsection{Local multigerms and asymptotics of intersections with tangent line} 

Let $a$, $b$ be irreducible germs of planar complex analytic curves at the origin in $\cc^2$. Let $p_g$, $q_g$, $c_g$, 
$g=a,b$ be respectively the corresponding exponents and constants from their parametrizations (\ref{curve}) in 
their adapted coordinates. 
Let $t$ be the corresponding local parameter  of the germ $b$.  
 We identify points of the curve $b$ with the corresponding local parameter values $t$. 
We use the following statements on 
 the asymptotics of the points of intersection $T_tb\cap a$.

\begin{proposition} \label{pasym}\cite[proposition 3.8]{gs} Let $a$, $b$ be transversal irreducible germs of holomorphic curves at the origin in $\cc^2$, and let $b$ be nonlinear. Let $(z,w)$ be affine coordinates centered at 0 and adapted to $b$: the germ 
$b$ is tangent to the $z$-axis. 
Then for every $t$ small enough the intersection $T_tb\cap a$  consists 
of $q_a$ points $\xi_1,\dots,\xi_{q_a}$ whose coordinates have the following asymptotics, as $t\to0$:
$$z(\xi_j)=O(t^{p_b})=O(w(t))=o(z(t))=o(t^{q_b}),$$
\begin{equation} w(\xi_j)=(1-r_b)w(t)(1+o(1))=(1-r_b)c_bt^{p_b}(1+o(1)).\label{zwb}\end{equation}
(Recall that $q_a=1$, if $a$ is a germ of line.)
\end{proposition}

\begin{proposition}  \label{asym} (\cite[p. 268, proposition 2.50]{alg}, \cite[proposition 3.10]{gs}) Let $a$, $b$ be  
irreducible tangent germs of holomorphic curves at the origin $O$ in the plane $\cc^2$, and let $b$ be nonlinear. 
Consider their parametrizations (\ref{curve})
in common adapted coordinates $(z,w)$. Let $c_a$ and $c_b$ be the corresponding constants from (\ref{curve}). 
Then for every $t$ small enough the intersection $T_tb\cap a$  consists 
of $p_a$ points $\xi_1,\dots,\xi_{p_a}$ (or just one point $\xi_1$, if $a$ is the germ of the line $T_Ob$) 
whose coordinates have the following asymptotics, as $t\to0$.

Case 1): $r_a>r_b$ (including the case, when $a$ is linear, i.e., $r_a=\infty$).  One has two types of intersection points $\xi_j$:
\begin{equation}\text{for } j\leq q_a: \ \ \ z(\xi_j)=\frac{r_b-1}{r_b}z(t)(1+o(1))=\frac{r_b-1}{r_b}t^{q_b}(1+o(1)),\ \ \label{a>b1}\end{equation} 
$$w(\xi_j)=O(t^{q_br_a})= o(t^{p_b})=o(w(t)); $$
\begin{equation}\text{ for } j>q_a: \ \ \ \ \ \ \ \ \ \ \ \  \ z(t)=O((z(\xi_j))^{\frac{r_a-1}{r_b-1}})=o(z(\xi_j)),\ \ \ \ \ \ \ \ \ \ \ \  \ \ \label{a>b2}\end{equation} 
$$w(t)=O(z^{r_b}(t))=O((z(\xi_j))^{\frac{r_b(r_a-1)}{r_b-1}})=o(z^{r_a}(\xi_j))=o(w(\xi_j)).$$
(Points satisfying (\ref{a>b2}) exist  if and only if $a$ is non-linear.) 

Case 2): $r_a=r_b=r$. One has 
\begin{equation} z(\xi_j)=\zeta_j^{q_a}z(t)(1+o(1))= \zeta_j^{q_a}t^{q_b}(1+o(1)),\label{a=b}\end{equation} 
$$w(\xi_j)=c_a\zeta_j^{p_a}t^{p_b}(1+o(1))=
c\zeta_j^{p_a}w(t)(1+o(1)),$$
where $\zeta_j$ are the roots of the polynomial 
\begin{equation}R_{p_a,q_a,c}(\zeta)=c\zeta^{p_a}-r\zeta^{q_a}+r-1; \ r=\frac{p_a}{q_a}, \ c=\frac{c_a}{c_b}.
\label{rpq}\end{equation}
(In the case, when $b=a$, one has $c=1$, and the above polynomial has double root 1 corresponding to the 
tangency  point $t$.) 

Case 3): $r_a<r_b$. One has 
\begin{equation} z(\xi_j)=O((z(t))^{\frac{r_b}{r_a}})=o(z(t)), \label{a<b}\end{equation} 
$$w(\xi_j)=(1-r_b)w(t)(1+o(1))=(1-r_b)c_bt^{p_b}(1+o(1)).$$
\end{proposition}

\begin{definition} \label{deflm} \cite[definition 3.3]{gs} Let $L\subset\cp^2$ be a line, and let $C\in L$. 
A {\it $(L,C)$-local multigerm (divisor)} is respectively a finite union (linear combination $\sum_jk_jb_j$ with 
$k_j\in\rr\setminus\{0\}$) 
of distinct irreducible germs of analytic curves $b_j$ (called {\it components}) at base points  $C_j\in L$ 
such that each germ at  $C_j\neq C$ is different from the line $L$. (A germ at $C$ can be arbitrary, 
in particular, it may coincide with the germ $(L,C)$.)  
The {\it $(L,C)$-localization} of an algebraic curve (divisor) in $\cp^2$ is the corresponding $(L,C)$-local 
multigerm (divisor) formed by all its local branches of the above type.
\end{definition}

Everywhere below in the present subsection $b$ is a nonlinear irreducible germ of analytic curve at a point $C\in\cp^2$, 
$\Gamma$ is a $(T_Cb,C)$-local multigerm (or divisor), and $(z,w)$ is a local affine chart centered at $C$ that 
is adapted to $b$: $T_Cb$ is the $z$-axis. For every affine coordinate $h$, which will be either $z$, or $w$, we consider its 
restriction to the projective lines $T_ub$. 

\begin{definition} Let $h$ be an affine coordinate  on a neighborhood of the point $C$ in $\cp^2$ 
centered at $C$: $h(C)=0$. 
 {\it The points of intersection} $\Gamma\cap T_ub$ {\it with linear $h$-asymptotics} 
are those intersection points whose $h$-coordinates have asymptotics $\tau_j h(u)(1+o(1))$, $\tau_j\neq0$, as $u\to C$; the 
corresponding constant 
factors $\tau_j$ are called the {\it asymptotic $h$-factors.} In the case, when $\Gamma$ is a divisor, we take each factor $\tau_j$ with 
multiplicity, which is the total multiplicity $n_j$ of all the intersection points with the same asymptotic factor $\tau_j$. The formal 
linear combination $M_h=\sum_jn_j[\tau_j]$, which is a divisor in $\cc^*$, will be called the {\it asymptotic $h$-divisor.} \end{definition}

\begin{definition} We say that a continuous family of points $Q=Q(u)$ of intersection 
$T_ub\cap\Gamma$ has {\it sublinear (superlinear) $h$-asymptotics,} if $h(Q(u))=o(h(u))$ (respectively, 
if $h(u)=o(h(Q(u)))$), as $u\to C$. 
\end{definition} 

\begin{remark} In general, the function $h(Q(u))$ can be multivalued. It can be always 
written as a Puiseux series in $z(u)$  (after multiplication by a power $z^s(u)$, $s\in\mathbb Q_{>0}$, 
if $h(Q(u))\to\infty$, as $u\to C$). The above notions of family of points with sublinear, 
linear and superlinear $h$-asymptotics and the asymptotic factors are well-defined in this 
general case. For every given affine coordinate $h$   on a neighborhood of the point $C$ in $\cp^2$ with 
$h(C)=0$ each (multivalued) continuous family of 
intersection points $Q(u)$ has one of the three above types. 
\end{remark}

In what follows for a multigerm (divisor) $\Gamma$ by $\Gamma(C)$ we will denote its part consisting of the irreducible 
germs based at $C$. 
Recall that for every irreducible germ $a$ in $\Gamma(C)$  we define the number $\rho_a$ by formula 
(\ref{roa}): $\rho_a=1$, if $a$ is transversal to $b$; $\rho_a=r_a$, if $a$ is tangent to $b$. Set 
\begin{equation}\Gamma_{\rho<r_b}=\text{ the collection (divisor) of germs } a \text{ in } \Gamma(C) 
\text{ with } \rho_a<r_b,\label{ga-b}\end{equation}
\begin{equation} \Gamma_{\rho>r_b}=\text{ the collection (divisor) of germs } a \text{ in } \Gamma(C)  
\text{ with } \rho_a>r_b,\label{ga+b}\end{equation}
\begin{equation}\Gamma_{\rho=r_b}=\text{ the collection (divisor) of germs } a \text{ in } \Gamma(C)  
\text{ with } \rho_a=r_b,\label{ga=b}\end{equation}
\begin{equation}\Gamma_{out}=\Gamma\setminus\Gamma(C), \text{ which consists of germs   that are 
not based at } C.\label{ga>>}\end{equation}

Thus, $\Gamma_{\rho<r_b}$ consists of exactly those germs $a$ in $\Gamma$ that 
are based at $C$, and such that 

-  either $a$ is transversal  to $b$,

- or $a$ is tangent to $b$ and $r_a<r_b$.

All the germs in $\Gamma_{\rho>r_b}$ and $\Gamma_{\rho=r_b}$ are tangent to $b$.

\begin{proposition} \label{propasy} 1) The points of intersection $T_ub\cap\Gamma$ with sublinear $z$-asymptotics are 
exactly the points of intersection of the line $T_ub$ with $\Gamma_{\rho<r_b}$. 

2) If $\Gamma_{\rho>r_b}\neq\emptyset$, then  
$T_ub\cap\Gamma_{\rho>r_b}$ is split into two  parts, 
\begin{equation}T_ub\cap\Gamma_{\rho>r_b}=\mcl_u^<\sqcup \mcl_u^>, \ \mcl_u^<\neq\emptyset:\label{tusplit}\end{equation}
- the points in $\mcl_u^<$ have linear $z$-asymptotics with $z$-factors equal to $\frac{r_b-1}{r_b}$; 

- $\mcl_u^>\neq\emptyset$ if and only if $\Gamma_{\rho>r_b}$ contains at least one non-linear germ; 
the points in $\mcl_u^>$ have superlinear $z$-asymptotics.

3) The  set of points in $T_ub\cap\Gamma$ with superlinear $z$-asymptotics is  $\mcl_u^>\sqcup (T_ub\cap\Gamma_{out})$.

4) The set of points of  intersection  $T_ub\cap\Gamma$ with linear $z$-asymptotics coincides with 
$(T_ub\cap\Gamma_{\rho=r_b})\sqcup L_u^<$.

5) Let 
$$r=r_b=\frac pq$$
be the irreducible fraction presentation of the Puiseux exponent $r_b$. Let 
 $a_1,\dots,a_N$ denote the germs forming $\Gamma_{\rho=r_b}$: they are tangent to $b$ and $r_{a_i}=r$. 
 Let $p_{a_i}$, $q_{a_i}$, $c_{a_i}$ be respectively 
 the asymptotic exponents and coefficients in their parametrizations (\ref{curve}):
 \begin{equation} p_{a_i}=s_ip, \ q_{a_i}=s_iq, \ s_i\in\nn, \ s_i=G.C.D.(p_{a_i}, q_{a_i}); \ c_{a_i}\in\cc^*.
 \label{pqs}\end{equation}
 Let $\zeta_{ij}$ ($i=1\dots,N$, $j=1,\dots p$) be the roots of the polynomials 
 \begin{equation} R_{p, q, c(i)}(\zeta)=c(i)\zeta^p-r\zeta^q+r-1, \ c(i)=\frac{c_{a_i}}{c_b}\in\cc^*. 
 \label{rpc}\end{equation}
 The asymptotic $z$-factors of points of the intersection $T_ub\cap\Gamma_{\rho=r_b}$ are $\zeta_{ij}^q$. 
 
 6) One has 
 \begin{equation} \zeta_{ij}^q\neq\frac{r-1}r \text{ for all } i \text{ and } j.\label{rneq}\end{equation}
 \end{proposition}
 
 {\bf Addendum to Proposition \ref{propasy}.} {\it In the conditions of Proposition 
 \ref{propasy} in the case, when $\Gamma$ is  a divisor,  let $m_i\in\nn$ denote the 
 multiplicities of the germs $a_i$ in $\Gamma_{\rho=\rho_b}$. The asymptotic $z$-divisor of the divisor 
 $\Gamma$ equals to 
  \begin{equation} M_z=\sum_{i=1}^N\sum_{j=1}^p\ell_i[\zeta_{ij}^q]+\kappa_z[\frac{r-1}r], \ \ell_i=m_is_i\in\nn, \ \kappa_z\in\zz_{\geq0},\label{divmz}
  \end{equation}
 \begin{equation}\kappa_z=|\mcl_u^<|>0 \text{ if and only if } \Gamma_{\rho>r_b}\neq\emptyset.\label{kappanmz}\end{equation}
  }
   
 \begin{proof} All the statements of Proposition \ref{propasy}, except for  
 inequality (\ref{rneq}), follow from 
 Propositions \ref{pasym} and \ref{asym}, see more details below. 
 Inequality (\ref{rneq}) is implied by the following 
 general proposition.
 
 \begin{proposition} \label{prneq} For every $p,q\in\nn$, $1\leq q<p$, $c\in\cc^*$, set $r=\frac pq$,  and every root $\zeta$ of the polynomial  $R_{p,q,c}(z)=cz^p-rz^q+r-1$ one has 
 \begin{equation} \zeta^q\neq\frac{r-1}r, \ c\zeta^p\neq1-r.\label{ineqq}\end{equation}
 \end{proposition}
 
\begin{proof} The proof of the first inequality repeats the proof of an equivalent statement from 
 \cite[proof of proposition 3.13]{gs}. Suppose the contrary: $\zeta^q=\frac{r-1}r$ for some root $\zeta$. 
 Then 
 $$R_{p,q,c}(\zeta)=c\zeta^p-r\zeta^q+r-1=c\zeta^p=c(\frac{r-1}r)^r\neq0.$$
 The contradiction thus obtained proves the first inequality in (\ref{ineqq}). Let us prove the second one. Suppose the contrary: 
 $c\zeta^p=1-r$ for some root $\zeta$. Then one has 
 $$R_{p,q,c}(\zeta)=c\zeta^p-r\zeta^q+r-1=-r\zeta^q\neq0.$$
 The contradiction thus obtained proves the second inequality  in (\ref{ineqq}) and Proposition \ref{prneq}. 
\end{proof}

 Set $W_i=R_{p,q,c(i)}$, $\wt W_i=R_{p_{a_i},q_{a_i},c(i)}$. 
 Statement 5) of Proposition \ref{propasy}  
 follows from Proposition \ref{asym}, Case 2) and the relation $\wt W_i(h)=W_i(h^{s_i})$, 
which implies that 
 to every root $\zeta$ of the polynomial $W_i$ correspond   $s_i$ 
 roots $\zeta^{\frac1{s_i}}$  of the polynomial $\wt W_i$ whose  
 $q_{a_i}$-th powers are equal to $\zeta^q$. Statements (\ref{divmz}) and (\ref{kappanmz}) follow from 
Statements 4), 5) of Proposition \ref{propasy}, the above discussion 
and  inequality (\ref{rneq}). 
\end{proof}

Recall that $a_1,\dots,a_N$ denote the germs forming $\Gamma_{\rho=r_b}$. 

\begin{proposition} \label{propasy2} 1) The set  
of points of intersection $T_ub\cap\Gamma$ with sublinear 
$w$-asymptotics is exactly the set $\mcl_u^<$ from (\ref{tusplit}).

2) The set of points of intersection $T_ub\cap\Gamma$ with superlinear $w$-asymptotics is the union 
$\mcl_u^>\sqcup (T_ub\cap\Gamma_{out})$.

3) The set of points of  intersection  $T_ub\cap\Gamma$ with linear $w$-asymptotics is  
$T_ub\cap(\Gamma_{\rho<r_b}\sqcup\Gamma_{\rho=r_b})$.  The asymptotic $w$-factors of the points in 
$T_ub\cap \Gamma_{\rho<r_b}$ 
are all equal to $1-r$, $r=r_b$. The asymptotic $w$-factors of the points in $T_ub\cap a_i$ are equal to $c(i)\zeta_{ij}^p$, 
$i=1,\dots,N$, $j=1,\dots,p$, where $\zeta_{ij}$ are the roots of the polynomials 
$R_{p,q,c(i)}$, see (\ref{rpc}). One has 
\begin{equation}c(i)\zeta_{ij}^p\neq 1-r \text{ for all } i \text{ and } j.\label{wneq}\end{equation}

4)  In the case, when $\Gamma$ is a divisor, let  $m_i$, $s_i$ be the same, as in (\ref{divmz}). 
The asymptotic $w$-divisor of the multigerm $\Gamma$ equals to 
  \begin{equation} M_w=\sum_{i=1}^N\sum_{j=1}^p\ell_i[c(i)\zeta_{ij}^p]+\kappa_w[(1-r)], \ \ell_i=m_is_i\in\nn, \ \kappa_w\in\zz_{\geq0},\label{divmw}
  \end{equation}
 \begin{equation}\kappa_w=|T_ub\cap\Gamma_{\rho<r_b}|>0 \text{ if and only if } \Gamma_{\rho<r_b}\neq\emptyset.\label{kappan}\end{equation}
\end{proposition} 

All the statements of Proposition \ref{propasy2} follow from Propositions \ref{pasym} and \ref{asym}, except for 
inequality (\ref{wneq}) (which follows from (\ref{ineqq})) and the part of statement  2) saying that the points in $T_ub\cap\Gamma_{out}$ 
have superlinear $w$-asymptotics, which is given by the following proposition.  

\begin{proposition} \label{supw} For every irreducible germ $a$ of analytic curve at 
any point $B\in T_Cb$, $B\neq C$, the points of intersection $T_ub\cap a$ have superlinear $w$-asymptotics, 
as $u\in b$ tends to $C$.
\end{proposition}

\begin{proof} For $u\in b$ being close enough to $C$, let $Q_1=Q_1(u)$  denote  the point of the intersection of 
the line $T_ub$ with  the $z$-axis. Fix an arbitrary family of points $Q_2(u)$ of the 
intersection $T_ub\cap a$. 
Their limits $Q_1(C)=C$ and $Q_2(C)=B$ lie in the $z$-axis 
and are distinct, by assumption; $z(C)=0\neq z(B)$. Let us show that $w(u)=o(w(Q_2(u)))$, as  $u\to C$. 

Let $T=T(u)$, $O=O(u)$  denote the respectively the projections of the points $u$ and $Q_2$  to the $z$-axis: 
$z(T)=z(u)$, $z(O)=z(Q_2)$. 
Consider the triangles $TQ_1u$ and $OQ_1Q_2$. They are similar in the following complex sense. 
Their edges $Tu$ and $OQ_2$ lie in complex lines 
parallel to the $w$-axis. Their edges $TQ_1$, $OQ_1$ lie in the  complex $z$-axis. 
Their edges $uQ_1$ and $Q_2Q_1$ lie in the same complex line $Q_1Q_2$. The parallelness of complexified 
edges of the above triangles implies that 
\begin{equation}\frac{w(u)-w(T)}{w(Q_2)-w(O)}=\frac{z(T)-z(Q_1)}{z(O)-z(Q_1)}.\label{above}\end{equation}
Substituting the equalities and asymptotics $w(T)=w(O)=0$, $z(Q_1(u))\to0$, $z(T)=z(u)\to0$, 
and $z(O(u))-z(Q_1(u))\to z(O(C))=z(B)\neq0$ to  formula (\ref{above}) yields to 
$\frac{w(u)}{w(Q_2)}\to0$. This proves 
Propositions \ref{supw} and  \ref{propasy2}.
\end{proof}

\subsection{Families of degenerating conformal involutions}
 
 In Subsection 4.3 we show that for every local branch $b$ as  in Theorem \ref{iii}  the corresponding 
 family of $\ii$-angular symmetries $T_ub\to T_ub$ with center $u$ written in appropriate coordinate 
 becomes a degenerating family of  conformal involutions $\oc\to\oc$ of one of the following types.
 
 \begin{definition} \label{defiab} Consider a family of  non-trivial 
 conformal involutions $\sigma_u:\oc\to\oc$ of the Riemann 
 sphere with coordinate $z$ that are parametrized by a small complex parameter $u$ with a given family 
 of fixed points $\zeta(u)$: 
 $$\sigma_u(\zeta(u))=\zeta(u); \ \ \zeta(u)\to0, \text{ as } u\to0.$$
  The family $\sigma_u$ is said to be 
 
 - {\it of type A,} if there  exist families of points $\alpha(u), \omega(u)\in\oc$ such that 
$$\sigma_u(\alpha(u))=\omega(u), \ \alpha(u)=o(\zeta(u)), \ \zeta(u)=o(\omega(u)), \text{ as } u\to0;$$ 

- {\it of type B,} if there exist families of points $\alpha(u), \omega(u)\in\oc$ such that 
$$\sigma_u(\alpha(u))=\omega(u), \ \alpha(u),\omega(u)=o(\zeta(u)), \text{ as } u\to0.$$ 
\end{definition}

\begin{proposition} \label{invdeg}  Each family of involutions $\sigma_u:\oc\to\oc$ of type A with given fixed points 
$\zeta(u)$ satisfies the following statements:

(a) The involutions $\sigma_u$ converge to the 
constant mapping $\oc\mapsto0$ uniformly on compact subsets in $\oc\setminus\{0\}$. 

(b) Fix a $c\in\cc^*$ and a family of points $z_u\in\cc$ 
with the asymptotics $z_u=c\zeta(u)(1+o(1))$, as $u\to0$. Then 
\begin{equation}\sigma_u(z_u)=c^{-1}\zeta(u)(1+o(1)), \text{ as } u\to0.\label{invdeg1}\end{equation}
\end{proposition}

 \def\la{\lambda}

\begin{proof} The scalings $\phi_u:z\mapsto \wt z=\frac z{\zeta(u)}$ conjugate the involutions $\sigma_u$ to the conformal involutions 
$\Sigma_u=\phi_u\circ\sigma_u\circ\phi_u^{-1}:\oc\to\oc$ fixing 1 and permuting the points $\frac{\alpha(u)}{\zeta(u)}$ and $\frac{\omega(u)}{\zeta(u)}$; $\frac{\alpha(u)}{\zeta(u)}\to0$, and $\frac{\omega(u)}{\zeta(u)}\to\infty$, 
as $u\to 0$.  
Hence, $\Sigma_u(z)\to\frac1z$ in $Aut(\oc)$ and thus, uniformly on $\oc$. 
For every $\delta>0$ the mapping $\sigma_u=\phi_u^{-1}\circ\Sigma_u\circ\phi_u$ converges 
to the constant mapping $\oc\mapsto0$ uniformly on $\oc\setminus D_{\delta}$. Indeed,   
$\phi_u(z)=\frac z{\zeta(u)}\to\infty$  uniformly on $\oc\setminus D_{\delta}$, since $\zeta(u)\to0$. Hence $f_u=\Sigma_u\circ\phi_u\to0$, 
$\sigma_u=\phi_u^{-1}\circ f_u=\zeta(u)f_u\to0$. This proves statement (a). For 
$z_u=c\zeta(u)(1+o(1))$ with $c\neq0$ one has 
$$\sigma_u(z_u)=\zeta(u)\Sigma_u((\zeta(u))^{-1}z_u)=\zeta(u)\Sigma_u(c+o(1))=
\zeta(u)(c^{-1}+o(1)).$$
This proves statement (b) and finishes the proof of the proposition.
\end{proof}

\begin{proposition} \label{invdeg2}  Each family of involutions $\sigma_u:\oc\to\oc$ of type B with given fixed points 
$\zeta(u)$ satisfies the following statements:

(a) The coordinate change $\wt z=\frac{\zeta(u)}z$ conjugates the involutions $\sigma_u$ 
to conformal involutions $\Sigma_u:\oc\to\oc$ that converge in $Aut(\oc)$ to the central symmetry with respect to one: 
$\wt z\mapsto 2-\wt z$. 

(b) For every $c\in\cc\setminus\{0,2\}$ and every family of points $z_u=c^{-1}\zeta(u)(1+o(1))$ one has $\sigma_u(z_u)=d^{-1}\zeta(u)(1+o(1))$, where $d=2-c$.
\end{proposition}

\begin{proof} The above change of coordinate $z\mapsto\wt z$ sends the fixed point $\zeta(u)$ of the involution $\sigma_u$ 
to 1, and $\wt z(\alpha(u)),\wt z(\omega(u))\to\infty$, as $u\to0$, since 
$\alpha(u),\omega(u)=o(\zeta(u))$. Therefore, the involution $\sigma_u$ written 
in the coordinate $\wt z$ fixes 1 and permutes two  points converging to infinity.  Its derivative at the fixed point 1 equals to -1, 
since the involution is nontrivial. Therefore, it converges to the unique non-trivial involution fixing 1 and $\infty$: the central symmetry 
with respect to 1. Statement (a) is proved. Statement (a) immediately implies statement (b). The proposition is proved. 
\end{proof}

\def\mcd{\mathcal D}

\subsection{Relative projective symmetry properties and their types}

 \begin{definition} \label{bab} Let $b$ be a nonlinear irreducible germ of analytic curve at a point $C\in\cp^2$. 
Let 
$\Delta=\sum_{j=1}^lk_jb_j$ be a $(T_Cb,C)$-local divisor containing $b$: say, $b_1=b$.
We say that the germ $b$ 
 has {\it relative projective symmetry property} 
with respect to the divisor $\Delta$,   if for every 
$u\in b\setminus\{ C\}$ there exists a projective involution 
$\sigma_u:T_ub\to T_ub$ with fixed point $u$ 
 such that the intersection  $\Delta\cap T_ub$  treated as a divisor on 
 $T_ub$ is $\sigma_u$-invariant. (We identify a point $u\in b$ with the corresponding value of the small complex 
 parameter $t$ of the curve $b$, $t(C)=0$; thus, $t(u)\to0$, as $u\to C$.) 
 For any given affine coordinate $h$ on a neighborhood of the point $C$ in $\cp^2$ with $h(C)=0$ 
 we say that $b$ has  relative projective symmetry property {\it of type A-$h$ (B-$h$)}, 
 if the family of involutions $\sigma_u$ written in the coordinate $h$ on the lines $T_ub$ is of type A (respectively, B), 
 see Definition \ref{defiab}, with $\zeta(u)=h(u)$ (the specified fixed point family). 
 \end{definition}
 
 \begin{proposition} \label{typin} Let $\ii\subset\cp^2$ be a conic: either a  regular conic, or a pair of distinct lines. 
  Let an irreducible algebraic curve $\gamma\subset\cp^2$ generate a rationally integrable 
 $\ii$-angular billiard with integral $G$, let $C\in\gamma$. Let $\Delta$ denote  the zero divisor  of the function $G$. Every local  branch $b$ of the curve 
 $\gamma$ at $C$ has  relative projective symmetry property with respect to the $(T_Cb,C)$-localization (see Definition \ref{deflm}) of 
 each one of the divisors $\Delta$ and $\Delta+\ii$: 
 the corresponding projective involution from Definition \ref{bab} is 
 the $\ii$-angular symmetry centered at $u$. In the case, when $C\in\gamma\cap\ii$, the following statements hold 
 in the corresponding cases listed below; here $(z,w)$ is a system of affine coordinates centered at $C$ and 
 adapted to $b$. 
 
 Case 1):  $C$ is a regular point of the conic $\ii$, and $b$ is transversal to $\ii$. 
 Then $b$ has relative projective symmetry property of type A-$z$.
 
 Case 2): $\ii$ is a pair of lines through the point $C$ that are both transversal to $b$. 
 Then $b$ has relative projective symmetry property of type B-$z$.
 
 Case 3): $C$ is a regular point of the conic $\ii$, and $b$ is tangent to $\ii$. 
 
Subcase 3a): $\ii$ is a pair of lines. Then $b$ has relative projective symmetry property of type A-$w$.
 
 Subcase 3b): $\ii$ is a regular conic and $r_b<2$. Then $b$ has relative projective symmetry property of type A-$w$.
 
 Subcase 3c): $\ii$ is a regular conic and $r_b>2$. Then $b$ has relative projective symmetry property of type B-$z$.
 \end{proposition}

\begin{proof} The first statement of the proposition follows immediately from definition. Let us prove its other 
statements case by case.  

Case 1). Then the line $T_Cb$ intersects $\ii$ at two points: the point $C$ and a  point $B\neq C$. Let 
$\ii_C$ and $\ii_B$ denote  the germs of the conic $\ii$ at $C$ and $B$ respectively. As $u\in b$ 
tends to $C$, the $\ii$-angular symmetry of the line $T_ub$ with center $u$ permutes its points $C_u$, $B_u$ of 
intersection with  $\ii_C$ and $\ii_B$. The coordinate $z(B_u)$ tends to a non-zero (may be infinite) limit, 
and $z(C_u)=o(z(u))$, as $u\to C$, by transversality of the germs $\ii_C$ and $b$ and 
Proposition \ref{pasym}. Therefore, 
the $\ii$-angular symmetries under question written in the coordinate $z$ form a family of conformal involutions 
of type A. 

Case 2). As $u\to b$, the line $T_ub$ intersects $\ii$ at two points permuted by the 
$\ii$-angular symmetry. These intersection points tend to $C$, and their $z$-coordinates are $o(z(u))$, by 
transversality, as in the above case. Hence, the $\ii$-angular symmetries of the lines $T_ub$ 
written in the coordinate $z$ 
form a family of involutions of type B. 

Case 3). 

Subcase 3a). Then the conic $\ii$ consists of two distinct lines intersecting at some 
point $B\neq C$: the line $\ii_C=T_Cb$ 
and a line $\ii_B$. The $(T_Cb,C)$-localization of the conic $\ii$ consists of two germs: the germ of the line $\ii_C$ at $C$; the germ of the line $\ii_B$ at $B$. 
As $u\in b$ tends to $C$, the line $T_ub$ intersects $\ii_C$ and $\ii_B$ at points $C_u$ and $B_u$ respectively, 
which are permuted by the $\ii$-angular symmetry with center $u$; $C_u\to C$, 
$B_u\to B$, as $u\to C$. One has  
$w(C_u)=0$, since $\ii_C=T_Cb$ is the $z$-axis, 
 and $w(u)=o(w(B_u))$,  by  Proposition 
\ref{supw}. Therefore, the $\ii$-angular symmetries of the lines $T_ub$ written in the coordinate $w$  
form a family of involutions of type A.

Subcase 3b). Then the  $(T_Cb,C)$-localization of the conic $\ii$ consists of just one regular germ 
at $C$, whose Puiseux exponent 2 is greater than $r_b$. As $u\in b$ tends to $C$, the line $T_ub$ intersects 
$\ii$ at two points $C_u$ and $B_u$ tending to $C$ so that $w(C_u)=o(w(u))$ and $w(u)=o(w(B_u))$, 
by Proposition \ref{asym}, Case 1). The points $C_u$ and $B_u$ are permuted by the $\ii$-angular symmetry 
with center $u$. Therefore, the $\ii$-angular symmetries of the lines $T_ub$ written in the coordinate $w$ 
form a family of conformal involutions of type $A$. 

Subcase 3c). Then $r_b>2=r_{\ii}$. As $u\in b$ tends to $C$, both points of intersection $T_ub\cap\ii$ tend to $C$ 
so that their $z$-coordinates are $o(z(u))$, by Proposition \ref{asym}, Case 3). The latter points are permuted by the 
$\ii$-angular symmetry centered at $u$. Therefore, these $\ii$-angular symmetries of the lines $T_ub$ written in the 
coordinate $z$ form a family of conformal involutions of type $B$. This proves Proposition \ref{typin}.
\end{proof}

 \subsection{Symmetry of asymptotic divisors.  Proof of statements 
 (i) and (ii-a)}
 
 Here we prove the following theorem generalizing statements (i) and (ii-a). 
   
 \begin{theorem} \label{tquad} Let $b$ be a nonlinear irreducible germ of analytic curve in $\cp^2$ at a point $C$, and let $(z,w)$ be affine 
 coordinates centered at $C$ that are adapted  to $b$. Let $b$ have local relative 
 projective symmetry property of type either A-$w$, or B-$z$. Then $b$ is quadratic.
 \end{theorem}
 
 We will deduce Theorem \ref{tquad} from invariance of asymptotic  divisors under appropriate 
 conformal involutions, see the following propositions. 
 
 \begin{proposition} \label{inva} Let an irreducible germ $b\subset\cp^2$ of analytic curve at a point $C$ have local relative projective 
 symmetry property of type A-$h$ for some affine coordinate $h$, $h(C)=0$. Then its asymptotic $h$-divisor 
 is invariant under the involution $\oc\to\oc$ of taking inverse: $z\mapsto z^{-1}$.
 \end{proposition}
 
 Proposition \ref{inva} follows from  Proposition \ref{invdeg}, Statement (b). 
 
 \begin{definition} For a divisor $M=\sum_jk_j[z_j]$ on $\oc$ its {\it inverse divisor} is 
 $$M^{-1}=\sum_jk_j[z_j^{-1}].$$
 For every divisor $M$ on $\oc$ and every subset $K\subset\oc$ by $M\setminus K$ we denote the divisor obtained 
 from $M$ by deleting those its points that lie in $K$ (taken with their total multiplicities). 
 \end{definition}
 
 \begin{proposition} \label{invb} Let an irreducible germ $b\subset\cp^2$ of analytic curve at a point $C$ have local relative projective 
 symmetry property of type B-$h$ for some affine coordinate $h$, $h(C)=0$. Let $M_h^{-1}$ denote 
  the inverse to its asymptotic $h$-divisor $M_h$. The divisor $M_h^{-1}\setminus\{2\}$ 
   is invariant under the central symmetry $\cc\to\cc$ with respect to one: $z\mapsto 2-z$.
   \end{proposition}
   
   Proposition \ref{invb} follows from  Proposition \ref{invdeg2}, Statement (b). 
   
   \begin{proof} {\bf of Theorem \ref{tquad}.} 
   
   Case 1) of symmetry property of type A-$w$. The asymptotic $w$-divisor 
    $M_w$ being invariant under taking inverse (Proposition \ref{inva}), the product of its points equals to one. On the 
    other hand, the latter product equals to the product of natural powers of expressions 
    \begin{equation} U_i=\prod_{j=1}^p(c(i)\zeta_{ij}^p)=(c(i))^p(\prod_{j=1}^p\zeta_{ij})^p\label{ppi}\end{equation}
    and a non-negative integer power  of the number $1-r$, see 
    (\ref{divmw}). One has $\prod_{j=1}^p\zeta_{ij}=(c(i))^{-1}(r-1)$ up to sign, by Vieta's Formula. 
    Therefore, in  formula (\ref{ppi}) the number $c(i)$ cancels out and $U_i=\pm(1-r)^p$. Finally, the product of 
    points of the divisor $M_w$, which is equal to one,  equals to  a natural power of the number $1-r$, up to sign. 
    Hence, $r=2$ and the germ $b$ is quadratic. 
    
  Case 2) of symmetry property of type B-$z$. The divisor $M_z^{-1}\setminus\{2\}$  
  being invariant under the symmetry with respect to one (Proposition \ref{invb}), the sum of its points equals to 
  its degree. Let us write this equation explicitly and deduce that $r=r_b=2$. 
  
  The divisor $M_z^{-1}$ has the form 
$$M_z^{-1}=\sum_i\ell_i\sum_{j=1}^p[\theta_{ij}^q] + \kappa_z[\frac r{r-1}], \ \theta_{ij}=\zeta_{ij}^{-1},  \ \kappa_z\in\zz_{\geq0},$$
$\ell_i\in\nn$, see (\ref{divmz}). The numbers $\theta_{ij}$ are the roots of the polynomials 
$$H_{p,q,c(i)}(\theta)=\theta^p R_{p,q,c(i)}(\theta^{-1})=(r-1)\theta^p-r\theta^{p-q}+c(i).$$ 
The points of the divisor $M_z^{-1}$ are  distinct from zero. Those of them that are powers $\theta_{ij}^q$ are different  
 from the number $\frac r{r-1}$, by Proposition \ref{prneq}. 
A priori, $M_z^{-1}$ may contain some of the points 2 and $\frac{r-2}{r-1}=2-\frac r{r-1}$, which are 
symmetric to 0 and  $\frac r{r-1}$, respectively. Set  
$M=M_z^{-1}\setminus\{2, \frac r{r-1}, \frac{r-2}{r-1}\}$:
\begin{equation}M=\text{ the sum of those terms } \ell_i[\theta_{ij}^q], \text{ for which } 
\theta_{ij}^q\neq2,\frac{r-2}{r-1}.\label{msumt} \end{equation}
 The divisor $M$ is symmetric with respect to one,  as is $M_z^{-1}\setminus\{2\}$. 

\begin{lemma} \label{lem27} \cite[lemma 3.16]{gs}. 
Let $r=\frac pq>1$; here $p,q\in\nn$, $(p,q)=1$. 
Consider a finite collection of polynomials $H_{p,q,c(i)}(\theta)$, $c(i)\neq0$ and numbers $\ell_i\in\nn$, 
$i=1,\dots,N$. Let $\theta_{ij}$ denote the roots of the polynomials $H_{p,q,c(i)}$. Let the divisor $M$ given by (\ref{msumt}) 
be invariant under the symmetry of the line $\cc$ with respect to one. Then $r=2$.
\end{lemma}

\begin{remark} In fact, lemma 3.16 in \cite{gs}  was stated in a slightly different but equivalent form. It dealt with 
a collection of polynomials $H_{p_i,q_i,c(i)}$, $q_i,p_i\in\nn$, $\frac{p_i}{q_i}=r>1$, $c(i)\neq0$ and the divisor $M$ of those 
 $q_i$-th powers of their roots that are distinct from the numbers 2 and $\frac{r-2}{r-1}$.  
Set $s_i=G.C.D(p_i,q_i)$. The latter $q_i$-th powers of roots coincide with the $q$-th powers of roots of the 
corresponding polynomials $H_{p,q,c(i)}$, $p=\frac{p_i}{s_i}$, $q=\frac{q_i}{s_i}$, and the divisor $M$ contains each of them $s_i$ times. Hence,  $M$ 
is given by (\ref{msumt}) with $\ell_i=s_i$, and this yields to equivalence of the above lemma to 
\cite[lemma 3.16]{gs}.
\end{remark}

Lemma \ref{lem27} together with the symmetry of the divisor $M$ given by (\ref{msumt})  imply that $r=2$. 
Theorem \ref{tquad} is proved.
\end{proof}

\begin{proof} {\bf of statements (i) and (ii-a) of Theorem \ref{iii}.} Every branch $b$ satisfying condition (i) of Theorem \ref{iii} has local relative projective symmetry property 
of type B-$z$, by Proposition \ref{typin}, Case 2). Hence, it is quadratic, by 
Theorem \ref{tquad}. Statement (i) is proved.

Let us prove statement (ii-a). Let $b$ be a branch satisfying condition (ii-a) of Theorem \ref{iii}. Then its base 
point $C$ is a regular point of the conic $\ii$, and $b$ is tangent to $\ii$. We treat the two 
following cases separately.

Case 1): $\ii$ is a union of two lines. Then $b$ has local relative projective symmetry property of type A-$w$, by Proposition \ref{typin}, Subcase 3a). Hence, it is quadratic, by Theorem 
\ref{tquad}. 

Case 2): $\ii$ is a regular conic. Suppose the contrary: $r=r_b\neq2$. We treat 
the two following subcases separately.

Subcase 2a): $r<2$. Then $b$ has local relative projective symmetry property of type A-$w$, by Proposition \ref{typin}, Subcase 3b). Hence, it is quadratic, by Theorem 
\ref{tquad}, -- a contradiction.

Subcase 2b): $r>2$. Then $b$ has local relative projective symmetry property of type B-$z$, by Proposition \ref{typin}, Subcase 3c). Hence, it is quadratic, by Theorem 
\ref{tquad}, -- a contradiction. Statements (i) and (ii-a) are proved.
\end{proof}

\subsection{Subquadraticity}
Here we prove the following theorem implying that every local branch $b$ satisfying the conditions 
of Statement (ii-b) of Theorem \ref{iii}  is subquadratic. Recall that such a branch has local relative projective 
symmetry property of type A-$z$, see Proposition \ref{typin}, Case 1). 

In what follows $b\subset\cp^2$ is a nonlinear irreducible germ of 
analytic curve at a point $C$, and $(z,w)$ are affine coordinates centered at $C$ and adapted to $b$. 

\begin{theorem} \label{subq} Every germ $b$ having local relative projective symmetry 
property of type A-$z$ with respect to some $(T_Cb,C)$-local divisor $\Gamma$ 
 is subquadratic. 
 \end{theorem}
 
 \begin{proof} In what follows for a given divisor $M$ on $\cc$ by $S(M)$ we denote the sum of its points. 
 The asymptotic $z$-divisor $M_z$ is invariant under taking inverse (Proposition \ref{inva}). 
 Therefore, $S(M_z)=S(M^{-1}(z))$. Let us write down the latter equality explicitly. 
 Let $a_1,\dots,a_N$ be the germs in $\Gamma$ that are tangent to $b$ and have 
 the same Puiseux exponent $r=r_b$. Let $\zeta_{ij}$ be the same, as in 
 (\ref{divmz}), set $\theta_{ij}=\zeta_{ij}^{-1}$.  One has 
\begin{equation}S(M_z)=\sum_{ij}\ell_i\zeta_{ij}^q+\kappa_z\frac{r-1}r=S(M_z^{-1})=\sum_{ij}\ell_i\theta_{ij}^{q}+
\kappa_z\frac r{r-1},\label{eqsum}\end{equation}
by (\ref{divmz}). 
Recall that for every fixed $i$ the numbers $\theta_{ij}$ are the roots of the polynomial $(r-1)\theta^p-r\theta^{p-q}+c(i)$. 
Hence, the sum of their $q$-th powers equals to 
$\frac p{r-1}$, by \cite[formula (3.17)]{gs}, and 
\begin{equation}S(M_z^{-1})=\frac{\Pi}{r-1}+\kappa_z\frac r{r-1}, \ \Pi=p\sum_i\ell_i.\label{equi}\end{equation}
Suppose the contrary: $r>2$, i.e., $p>2q$. Then $\sum_j\zeta_{ij}^q=0$ for every $i=1,\dots,N$. Indeed, the latter sum is expressed 
as a polynomial in the symmetric polynomials in $\zeta_{ij}$ of degrees $1,\dots,q$. 
All of these symmetric polynomials vanish, as do the coefficients 
of the polynomial $R_{p,q,c(i)}(\zeta)=c(i)\zeta^p-r\zeta^q+r-1$ at monomials of degrees $p-1,\dots,p-q>q$. Hence, 
$S(M_z)=\kappa_z\frac{r-1}r$. Substituting the latter equality and (\ref{equi}) to (\ref{eqsum}) yields to 
$$S(M_z)=\kappa_z\frac{r-1}r=S(M_z^{-1})=\frac{\Pi}{r-1}+\kappa_z\frac r{r-1}>\kappa_z\frac r{r-1}.$$
The latter inequality is strict, since $\Pi>0$: the collection of germs $a_i$ contains $b$, and hence, is non-empty. 
But its right-hand side  is no less than the left-hand side, since $\frac r{r-1}>1>\frac{r-1}r$. The 
contradiction thus obtained proves the inequality $r\leq2$. 
\end{proof}

{\bf Open Problem.} {\it Is it true that every germ $b$ having local relative projective symmetry property of type A-$z$ is 
a) quadratic? 
b) regular and quadratic?}

\subsection{Puiseux exponents}
Here we prove the following theorem implying that for every local branch $b$ of the curve $\gamma$ 
satisfying the conditions 
of Statement (ii-b) one has $\Gamma_{\rho>r_b}=\emptyset$, that is,  $b$ has the maximal Puiseux exponent 
among all the local branches of the curve $\Gamma$ that are tangent to $b$. 
\def\mcp{\mathcal P}

 \begin{theorem} \label{puiseux} Let  $b\subset\cp^2$ be a nonlinear irreducible germ of analytic curve at a point $C$, 
 and let $(z,w)$ be affine coordinates centered at $C$ and adapted to $b$. Let $b$ have local relative projective 
  symmetry property of type A-$z$ with respect to a $(T_Cb,C)$-local divisor  
 $\Delta$. Then each irreducible germ  at $C$ tangent to $b$ in the divisor $\Delta$ 
has Puiseux exponent no greater than $r_b$. 
\end{theorem}

The existence of a germ $a$  in $\Delta$ tangent to $b$ with $r_a>r=r_b$ is equivalent to the statement that the asymptotic $z$-divisor $M_z$ 
contains the point $\theta=\frac{r-1}r$. Recall that its other points are the $q$-th powers of roots of a finite 
collection of polynomials $R_{p,q,c(i)}$. See the Addendum to Proposition \ref{propasy}.

We will deduce Theorem \ref{puiseux} from the following  proposition.
 
\begin{proposition} \label{port} Let  $p,q\in\nn$, $1\leq q<p$, $r=\frac pq$, 
$$W(z)=R_{p,q,c}(z)=cz^p-rz^q+r-1, \ \phi=\left(\frac{r-1}r\right)^{\frac1q}.$$ 
The polynomial $W(z)$ has a real root  $z>\phi$, if and only if 
$0<c\leq1$. In this case it has a pair of  roots $z_0=z_0(c)$ and $z_1=z_1(c)$ in the interval 
$(\phi,+\infty)$ that are separated by one, 
if $0<c<1$, and  both equal to 1, if $c=1$: 
\begin{equation}\phi<z_0(c)<1<z_1(c), \ \text{ whenever } 0<c<1.\label{ineqrrr}\end{equation}
The functions $z_0(c)$ and $z_1(c)$ of $c\in(0,1)$ are strictly increasing (decreasing) homeomorphisms of 
the interval $(0,1)$ onto $(\phi,1)$ (respectively, $(1,+\infty)$). 
\end{proposition}

\begin{proof} For $c\notin\rr_{+}$ one has $W|_{\{z>\phi\}}\neq0$, since $-rz^q+r-1<0$ for every $z>\phi$. Therefore, 
we consider that $c>0$. 
The derivative equals to $W'(z)=cpz^{p-1}-rqz^{q-1}=pz^{q-1}(cz^{p-q}-1)$. Therefore, $c^{-\frac1{p-q}}$ is the unique 
local extremum of the polynomial $W$ in the positive semiaxis, and it is 
obviously a local minimum. For $c=1$ one has $W(1)=0$, and $z=1$ is exactly the  minimum. 
Therefore, as $c$ increases, the graph of the polynomial $W$ becomes disjoint from the positive coordinate semiaxis, and 
it has no positive root, if $c>1$. As the positive $c$ decreases from 1 to 0, the graph  intersects 
the coordinate axis on both sides from 1 at two points $z_0(c)$ and $z_1(c)$ separated by the minimum and by 1, 
$\phi<z_0(c)<1<z_1(c)$; $z_0(c)$ 
moves to the left, and $z_1(c)$ moves to the right. This follows from Proposition \ref{prneq} (which 
 implies that $z_0(c)\neq\phi$, hence $z_0(c)$ remains greater than $\phi$) and the inequality 
 $W'(z_0(c))<0<W'(z_1(c))$ (which holds, since 
 the points $z_0(c)$ and $z_1(c)$ lie on different sides from the minimum). 
The root $z_1(c)$ cannot disappear to infinity before $c$ will arrive to 0, since $W(z)\to+\infty$, as $z\to+\infty$, for every fixed $c>0$. 
The above discussion implies that the functions 
 $z_0(c)$ and  $z_1(c)$ are strictly increasing (respectively, decreasing) continuous 
 mappings from $(0,1)$ to $(\phi,1)$ and $(1,+\infty)$ respectively. These mappings are homeomorphisms "onto", since 
  each point $x\in(\phi,+\infty)$ is a 
 root of a polynomial $R_{p,q,c}$ with $c=\frac{rx^q-r+1}{x^p}>0$, and one has $c\leq1$, as was shown above. 
This implies the statements of   Proposition \ref{port}. \end{proof}

\begin{proof} {\bf of Theorem \ref{puiseux}.} Suppose the contrary. Then the asymptotic $z$-divisor $M_z$ contains 
the point $\theta=\frac{r-1}r$, as was noted after Theorem \ref{puiseux}. There exists a strictly decreasing homeomorphism $J:[1,+\infty)\to(\theta,1]$ 
such that $J(z_1^q(c))=z_0^q(c)$ for every $c\in(0,1]$, by Proposition \ref{port}. 
Set 
$$\sigma(z):=z^{-1}, \ \beta:=J\circ\sigma.$$ 
The mapping $\beta$ is a strictly increasing mapping $[\theta,1]\to(\theta,1]$, and  $\beta(\theta)=
J(\theta^{-1})\in(\theta,1)$. Hence, the iterates $\beta^n(\theta)\in(\theta,1)$ form 
an infinite increasing sequence of points. All of them  lie in $M_z$,  
by $\sigma$-symmetry of the divisor $M_z$ (Proposition \ref{inva}), the inclusion $\theta\in M_z$  and the fact that 
the points in $M_z$ different from $\theta$ are exactly $q$-th powers of roots  
of a finite collection of polynomials $W_i=R_{p,q,c(i)}$ (the Addendum to Proposition \ref{propasy}). Indeed, if a point 
$\zeta\in[\theta,1)$ lies 
in $M_z$, then $\sigma(\zeta)\in(1,+\infty)\cap M_z$, by symmetry. Hence, $\sigma(\zeta)$ is  
a $q$-th power of root of some polynomial $W_i$. But we already know that the number $(\sigma(\zeta))^{\frac1q}>1$ 
is a root of a real polynomial $W_0=R_{p,q,c_0}$ with $0<c_0<1$ (Proposition \ref{port}). 
This implies that the ratio of the numbers $c_0$ and $c(i)$ is a $\frac pq$-th power of unity, and the polynomials 
$W_i$ and $W_0$ have the same collection of $q$-th powers of roots. But then $\beta(\zeta)=J(\sigma(\zeta))\in(\theta,1)$ 
is a $q$-th power of root of the same polynomial $W_0$, or equivalently, $W_i$, hence, $\beta(\zeta)\in M_z$. 
Finally, the finite divisor $M_z$ contains an infinite sequence of points $\beta^n(\theta)$. The contradiction thus obtained 
proves Theorem \ref{puiseux}.
\end{proof}

\subsection{Concentration of intersection index}

In the condition of statement (ii-b) of Theorem \ref{iii} let $\Delta$  be the zero divisor of a rational integral of the 
$\ii$-angular billiard generated by  $\gamma$; we normalize $\Delta$ by 
positive  rational factor so that $b$ is included in $\Delta$ 
with multiplicity one. 
Here we prove the following theorem implying that more than 
one half of  the intersection index  $(\Delta,T_Cb)$ 
is concentrated at the base point $C$. 
 
\begin{theorem} \label{intloc}  Let $b\subset\cp^2$ be a nonlinear irreducible germ of analytic curve at a point $C$.  Let $(z,w)$ be affine coordinates centered at $C$ and adapted to $b$. Let $b$
 have local relative projective symmetry property of type A-$z$ with respect to an {\bf effective} $(T_Cb,C)$-local divisor 
$\Delta=\sum_{j=1}^Nk_jb_j$, i.e., $k_j>0$. Let $\Delta$   include the germ $b$ with coefficient 1. 
Set $D=\deg(\Delta)$: this is the intersection index  $(\Delta,T_Cb)$. 
Then the local intersection index of the projective tangent line $T_Cb$ with  
$\Delta$  at $C$ is no less than $\frac{D}2+1$.  
The equality may take place only in the case, when 
the germ $b$ is quadratic and regular, and $\Delta$ contains no other germs tangent to $b$ at $C$ with the same Puiseux exponent, 
as $b$. 
\end{theorem}

\begin{proof} 
Everywhere below for any effective divisor $\mcd=\sum_jn_j[\tau_j]$ on $\cc$, $n_j>0$, we denote by $|\mcd|$ its degree: $|\mcd|=\sum_jn_j$. 
For every $u\in b$ close to $C$ let $\Chi=\Chi(u)$ denote the part of the divisor $T_ub\cap\Delta$ on $T_ub$ consisting of those 
its points that tend to $C$, as $u\to C$. Let $\Psi(u)$ denote the remaining part 
of the divisor $T_ub\cap\Delta$, consisting of those its points that do not tend to $C$: 
they tend to the other base points of the germs in $\Delta$. 
The local intersection index $(T_Cb, \Delta)_C$ at the point $C$ 
equals to the degree $|\Chi(u)|$ of the divisor $\Chi(u)$, whenever $u$ is close enough to $C$. 

Let $\Chi_1=\Chi_1(u)$ and $\Chi_0=\Chi_0(u)$ denote the parts of the divisor $\Chi(u)$ formed respectively by the points with 
linear $z$-asymptotics and the points that do not have linear $z$-asymptotics. 

Recall that the divisors $T_ub\cap\Delta$ 
are invariant under projective involutions $\sigma_u:T_ub\to T_ub$ fixing $u$ 
and forming a family of type A in the coordinate $z$. 

{\bf Claim 1.} {\it The involution $\sigma_u$  sends the points of the divisor $\Psi(u)$ 
 to some points in $\Chi_0(u)$, and $|\Chi_0(u)|\geq|\Psi(u)|$.}

\begin{proof} The involutions $\sigma_u$ written in the coordinate $z$
converge to the constant mapping $\oc\mapsto 0$ uniformly on compact subsets in $\oc\setminus\{0\}$, as $u\to C$, by 
Proposition \ref{invdeg} (a). Therefore, the image of a point converging to a limit distinct from $C$, as 
$u\to C$, is a point converging to $C$. This implies that each point of the divisor $\Psi(u)$ is sent to a point in $\Chi(u)$. Its image in $\Chi(u)$ 
cannot lie in $\Chi_1(u)$, since the divisor $\Chi_1(u)$ of points with linear $z$-asymptotics is 
$\sigma_u$-invariant, by  
Proposition \ref{invdeg} (b). 
Hence, $\sigma_u$ sends $\Psi(u)$ to a part of the divisor 
$\Chi_0(u)$. This  proves the claim. 
\end{proof}

Thus, one has  
$$\Delta\cap T_ub=\Chi_0(u)+\Chi_1(u)+\Psi(u), \ \ |\Chi_0(u)|\geq|\Psi(u)|,$$
$$|\Chi_0(u)|+\frac12|\Chi_1(u)|\geq\frac{|\Chi_0(u)|+|\Chi_1(u)|+|\Psi(u)|}2=\frac12|\Delta\cap T_ub|=\frac D2.$$
This implies that 
\begin{equation}(T_Cb,\Delta)_C=|\Chi(u)|=|\Chi_0(u)|+|\Chi_1(u)|\geq\frac D2+\frac12|\Chi_1(u)|.\label{dhalf}\end{equation}
One has $|\Chi_1(u)|\geq2$. Indeed, the  divisor $\Chi_1(u)$ of points with linear $z$-asymptotics 
 includes the intersection $b\cap T_ub$ (which has degree at least two) with coefficient one and  the intersections 
 (with positive coefficients)
of the line $T_ub$ with those germs in $\Delta$ that are tangent to $b$ and have the same Puiseux exponent $r=r_b$. The equality may take place only if $b$ is regular and quadratic and there are no additional latter germs. 
This together with (\ref{dhalf}) implies that $(T_Cb,\Delta)_C\geq\frac D2+1$ and proves Theorem \ref{intloc}. 
\end{proof}

\subsection{Exponent in the asymptotics of Bialy--Mironov Formula. 
Proof of statement (ii-b)} 

Let $b$ be a local branch of the curve $\gamma$ at a point $C\in\gamma\cap\ii$ that 
 is a regular 
point of the conic $\ii$, and let $b$ be transversal to $\ii$. Let $\sum_{j=1}^lk_jb_j$, $b_1=b$, $k_1=1$, be 
the germ at $C$ of the divisor $\frac1k\Delta$, see (\ref{deldeg}); here $k_j>0$ for all $j$. Let $\rho_{b_j}$ and $\eta$ be the corresponding constants from formulas (\ref{roa}) and (\ref{nuprec}) respectively. Let us show that the upper 
bound (\ref{ineta}) on the number $\eta$ proved in Subsection 3.4 cannot hold, unless $b$ is regular and quadratic. 
Indeed, let $(z,w)$ be affine coordinates adapted to $b$. The branch 
$b$ has local relative projective symmetry property of type A-$z$, by Proposition \ref{typin}, Case 1). 
Therefore,  one has: 

 $r=r_b\leq2$, by Theorem \ref{subq}; 
 
$\rho_{b_j}\leq r$ for all $j=1,\dots,l$, by Theorem \ref{puiseux}.

Substituting these inequalities to formula (\ref{nuprec}), one gets 
\begin{equation} \eta=3\sum_{j=1}^lk_jq_{b_j}\min\{\rho_{b_j},r\}-2(r+1)\geq3\sum_{j=1}^lk_jq_{b_j}\rho_{b_j}-6.
\label{etanew}\end{equation}
The sum in the  right-hand side in (\ref{etanew}) equals to the local intersection index of the 
divisor $\frac1k\Delta$ with $T_Cb$ at the point $C$, by definition. The latter local intersection index is no 
less than $\frac{\deg(\Delta)}{2k}+1$, by Theorem \ref{intloc}. Therefore, 
$$\eta\geq3(\frac{\deg(\Delta)}{2k}+1)-6=3\frac{\deg(\Delta)}{2k}-3.$$
The latter inequality is strict, unless the local branch $b$ is regular and quadratic, as in Theorem \ref{intloc}. 
The strict inequality would obviously contradict inequality (\ref{ineta}), and hence, $b$ is regular and quadratic. 
Statement (ii-b) is proved. The proof of Theorem \ref{iii} is complete.

\section{Generalized genus and Pl\"ucker formulas. Proof of Theorem \ref{thalg2}} 
The proof of Theorem  \ref{thalg2} is based on generalized Pl\"ucker and genus formulas for planar algebraic curves and their corollaries, see, e.g.,  \cite[subsection 4.1]{gs}. It  
is done by a modified version of Eugenii Shustin's arguments from \cite[subsection 4.2]{gs}. 
 The main observation is that the assumptions of Theorem \ref{iii}
 on the  Puiseux exponents of local branches of the curve and Pl\"ucker formulas yield that the singularity 
 invariants of the considered curve $\gamma$ must obey a relatively high lower bound. On the other hand, the contribution 
 of its  potential singular (inflection) points, which lie in the conic $\ii$, appears to be not sufficient to fit that lower bound, unless 
 the curve is a conic. 
 
 \subsection{Invariants of plane curve singularities}
 The material of the present subsection is contained in \cite[subsection 4.1]{gs}. It recalls classical results on invariants of 
 singularities presented in 
  \cite[Chapter III]{BK}, \cite[\S10]{Mi}, see also a modern exposition in \cite[Section I.3]{GLS}. 
 Let $\gamma\subset\cp^2$ be a non-linear irreducible 
 algebraic curve\footnote{Everything stated in the present subsection holds for every algebraic curve in $\cp^2$ 
 with no multiple components and no straight-line components, see \cite[theorem 1]{sh}.}.   
 Let $d$ denote its degree. 
The intersection index of the curve $\gamma$ with its Hessian $H_{\gamma}$ 
equals to $3d(d-2)$, by B\'ezout Theorem. On the other hand, it is equal to the sum of the contributions $h(\gamma,C)$, which are called 
the {\it Hessians of the germs} $(\gamma,C)$, through all 
the singular and inflection points $C$ of the curve $\gamma$:
\begin{equation}3d(d-2)=\sum_{C\in\gamma}h(\gamma,C).\label{hesin}\end{equation} 
 An explicit formula for the Hessians $h(\gamma,C)$ was found 
in \cite[formula (2) and theorem 1]{sh}. To recall it, let us introduce the following notations. For every local 
branch $b$ of the curve $\gamma$ at $C$ let $s(b)$ denote its multiplicity: its intersection index with a generic line through $C$. 
Let $s^*(b)$ denote the analogous multiplicity of the dual germ. Note that 
$$s(b)=q,  \ \ s^*(b)=p-q,$$
where $p$ and $q$ are the exponents in the parametrization $t\mapsto(t^q,c_bt^p(1+o(1)))$ of the local branch 
$b$ in adapted 
coordinates. Thus,

\begin{equation} s(b)=s^*(b) \text{ if and only if } b \text{ is quadratic,}\label{quit}\end{equation}
\begin{equation} s(b)\geq s^*(b) \text{ if and only if } b \text{ is subquadratic.}\label{quit2}\end{equation}

Let 
$b_{C1},\dots,b_{Cn(C)}$ denote the local branches of the curve $\gamma$ at $C$; here $n(C)$ denotes their number. 
The above-mentioned formula for $h(\gamma,C)$ 
from \cite{sh} has the form 
\begin{equation} h(\gamma,C)=3\kappa(\gamma,C)+\sum_{j=1}^{n(C)}(s^*(b_{Cj})-s(b_{Cj})),\label{hessa}\end{equation} 
where $\kappa(\gamma,C)$ is the $\kappa$-invariant, the class of the singular point. 
Namely, consider the germ of function $f$ defining the germ $(\gamma,C)$; $(\gamma,C)=\{ f=0\}$. Fix a line $L$ 
through $C$ that is transversal to all the local branches of the curve $\gamma$ at $C$. Fix a small ball  $U=U(C)$ 
centered at $C$ and consider a level curve $\gamma_{\var}=\{ f=\var\}\cap U$ with small $\var\neq0$, which is non-singular. 
The number 
$\kappa(C)=\kappa(\gamma,C)$ is the number of  points of the curve $\gamma_{\var}$ where its tangent line is parallel to $L$.  (One has $\kappa(C)=0$ 
for nonsingular points $C$.) 
It is well-known that 
\begin{equation}\kappa(\gamma,C)=2\delta(\gamma,C)+\sum_{j=1}^{n(C)}(s(b_{Cj})-1),\label{kap}\end{equation}
see, for example, \cite[propositions I.3.35 and I.3.38]{GLS}, where 
 $\delta(\gamma,C)=\delta(C)$ is the $\delta$-invariant. Namely, consider the curve $\gamma_{\var}$, which is a Riemann surface whose  boundary is a finite collection of 
closed curves: their number equals to $n(C)$. Let us take the 2-sphere 
with $n(C)$  deleted disks. Let us paste it to $\gamma_{\var}$: this yields to a compact surface. By definition, 
its genus is the $\delta$-invariant $\delta(C)$. 
One has $\delta(C)\geq0$, and $\delta(C)=0$ whenever  $C$ is a non-singular point.   Hironaka's genus formula \cite{hir} implies that 
\begin{equation}\sum_{C\in\Sing(\gamma)}\delta(\gamma,C)\leq\frac{(d-1)(d-2)}2.\label{d<d}\end{equation}
Formulas (\ref{hesin}), (\ref{hessa}) and  (\ref{kap}) together imply that 
$$3d(d-2)=6\sum_C\delta(\gamma,C)+3\sum_C\sum_{j=1}^{n(C)}(s(b_{Cj})-1)+\sum_C\sum_{j=1}^{n(C)}(s^*(b_{Cj})-s(b_{Cj})).$$
The first term in the latter right-hand side is no greater than $3(d-1)(d-2)$, by inequality (\ref{d<d}). This implies that 
$$ 3d(d-2)-3(d-1)(d-2) $$
\begin{equation}=3(d-2)\leq3\sum_C\sum_{j=1}^{n(C)}(s(b_{Cj})-1)+\sum_C\sum_{j=1}^{n(C)}(s^*(b_{Cj})-s(b_{Cj})).
\label{ind0}\end{equation}

\subsection{Proof of Theorem \ref{thalg2} for a union $\ii$  of two lines} 

Let $\ii$ be a union of two distinct lines $\La_1$ and $\La_2$ through the point $O$. 
We know that all the singular and inflection points of the curve $\gamma$ (if any) lie in  $\ii=\La_1\cup\La_2$.  Set 
$$\mcb_{tan}=\{\text{the local branches of  } \gamma \text{ at points } C\in\ii\setminus\{ O\} \text{ tangent to  } \ii\},$$
$$\mcb_{O,tr}=\{\text{the  branches of the curve } \gamma \text{ at  } O \text{ transversal to both  } \La_1, \ 
\La_2\},$$
$$\mcb_{O,tan,j}=\{\text{the branches of the curve } \gamma \text{ at  } O \text{ tangent to  } \La_j\},$$
$$\mcb_{O,tan}=\sqcup_{j=1,2}\mcb_{O,tan,j}, \ \mcb_O=\mcb_{O,tr}\sqcup\mcb_{O,tan}.$$

All the local  branches $b\notin \mcb_{O,tan}$  of the curve $\gamma$ at points in $\gamma\cap\ii$
 are subquadratic, by the conditions of Theorem \ref{thalg2}. 
 Therefore, their  contributions $s^*(b)-s(b)$ 
to the right-hand side in (\ref{ind0}) are non-positive, by (\ref{quit2}). Every local branch $b\notin(\mcb_{tan}\cup\mcb_O)$ 
is  regular, by assumption, hence its contribution $s(b)-1$ to (\ref{ind0}) vanishes. This together with (\ref{ind0}) implies that 
$$d-2\leq \sum_{b\in\mcb_{tan}\cup\mcb_{O,tr}\cup\mcb_{O,tan}}(s(b)-1)+\frac13\sum_{b\in\mcb_{O,tan}}(s^*(b)-s(b))$$
\begin{equation}= \sum_{b\in\mcb_{tan}\cup\mcb_{O,tr}\cup\mcb_{O,tan}}s(b)-|\mcb_{tan}|-|\mcb_{O,tr}|-|\mcb_{O,tan}|+\frac13\sum_{b\in\mcb_{O,tan}}(s^*(b)-s(b)),\label{ineqs}\end{equation}
where $|\mcb_{s}|$, $s\in\{ tan, (O,tr), (O,tan)\}$ are the cardinalities of the sets $\mcb_s$. 

Let us estimate the right-hand side in (\ref{ineqs}) from above. To do this, we use the next equality, 
which follows from B\'ezout Theorem. 

In what follows 
for every $j=1,2$ by $\mcb_{reg,j}$ we denote the collection of the local branches of the curve $\gamma$ at 
points in $\La_j\setminus\{ O\}$ that are transversal to $\La_j$. Recall that they are regular, by assumption. 
Set
$$\nu_j=|\mcb_{reg,j}|,$$
$$\mcb_{tan,j}=\{ b\in\mcb_{tan} \ | \ b \text{ is tangent to } \La_j\}, \ \mcb_{tan}=\mcb_{tan,1}\sqcup\mcb_{tan,2}.$$

{\bf Claim 1.} {\it   For every $j=1,2$ one has} 
$$\sum_{b\in\mcb_{tan,j}}s(b)+\frac12\sum_{b\in\mcb_{O,tan,3-j}}s(b)+\frac12\sum_{b\in\mcb_{O,tr}}s(b)$$
  \begin{equation}+\frac{\nu_j}2+\frac12\sum_{b\in\mcb_{O,tan,j}}(s^*(b)+s(b))=\frac d2.\label{ind}\end{equation} 

\begin{proof} The intersection index of the curve $\gamma$ with each line 
  $\Lambda_j$ equals to $d$ (B\'ezout Theorem). It is the sum of the  intersection indices of the line $\La_j$ with the branches from 
  the collections $\mcb_{tan,j}$, $\mcb_{O,tr}$, $\mcb_{O,tan}$, $\mcb_{reg,j}$.  Let us calculate the latter indices.  
  The contribution of each branch from $\mcb_{reg,j}$ equals to one, by regularity 
  and transversality. The  intersection index  of each branch $b\in\mcb_{O,tr}$ with $\La_j$ equals to $s(b)$. The intersection index 
  with $\La_j$  of 
each branch $b\in\mcb_{tan,j}$ equals to $p_b=2s(b)$, by quadraticity  (condition of Theorem \ref{thalg2}).  The intersection index with $\La_j$ 
 of each branch $b\in\mcb_{O,tan,j}$ equals to $p_b=s(b)+s^*(b)$. The remaining branches 
 $b\in\mcb_{O,tan,3-j}$ are transversal to $\La_j$, and their intersection indices with $\La_j$ are equal to $s(b)$.   Summing up the above intersection indices, writing that their sum should be equal to $d$ and dividing the 
equality thus obtained by two yields to (\ref{ind}). 
\end{proof} 

 Summing up equalities (\ref{ind}) for both  $j=1,2$ yields to 
  \begin{equation}\sum_{b\in\mcb_{tan}\cup\mcb_{O,tr}\cup\mcb_{O,tan}}s(b)= d-\frac12\sum_{b\in\mcb_{O,tan}}s^*(b)-
  \frac{\nu_1+\nu_2}2.\label{bsb}\end{equation}
  
  Substituting equality (\ref{bsb}) to (\ref{ineqs}) together with elementary inequalities yields to 
  $$d-2\leq d-\frac12\sum_{b\in\mcb_{O,tan}}s^*(b)-\frac{\nu_1+\nu_2}2-|\mcb_{tan}|-|\mcb_{O,tr}|-|\mcb_{O,tan}|$$
  $$+\frac13\sum_{b\in\mcb_{O,tan}}(s^*(b)-s(b))=d-|\mcb_{tan}|-|\mcb_{O,tr}|-|\mcb_{O,tan}|$$
  $$-\frac{\nu_1+\nu_2}2-\sum_{b\in\mcb_{O,tan}}(\frac16s^*(b)+\frac13s(b)),$$
  \begin{equation}|\mcb_{tan}|+|\mcb_{O,tr}|+|\mcb_{O,tan}|+\frac{\nu_1+\nu_2}2+\sum_{b\in\mcb_{O,tan}}(\frac16s^*(b)+\frac13s(b))\leq2.\label{bi}\end{equation}

{\bf Claim 2.} {\it The total cardinality of the set of singular and inflection points of the curve $\gamma$ is at most two. 
There are two possible cases:

- either there are no inflection points, and each local branch of the curve $\gamma$ at every singular point is subquadratic; 

- or there is just  one  special point (singular or inflection point), and $\gamma$ has one local branch at it.}
 
\begin{proof} Let $\Phi$ denote the collection of all the local branches of the curve $\gamma$ at points in $\ii$. Recall that 
$\ii$ contains all the singular and inflection points of the curve $\gamma$.

Case 1): $\mcb_{O,tan}=\emptyset$. Then all the local branches in $\Phi$ are subquadratic,  and there are no 
inflection points; $|\mcb_{tan}|+|\mcb_{O,tr}|\leq 2$, by (\ref{bi}). 

Subcase 1.1): $\mcb_{tan}=\mcb_{O,tr}=\emptyset$. Then  all the branches in $\Phi$ 
are regular and quadratic, and there are at most four of them: $\nu_1+\nu_2\leq4$, by (\ref{bi}). Thus, the 
 only possible candidates to be singular points of the curve $\gamma$ are 
intersections of  branches. Since the total number of branches under question is 
 at most four, the number of singular points is at most two. 

Subcase 1.2): $|\mcb_{tan}|+|\mcb_{O,tr}|=1$. The  branches from the complement $\Phi\setminus(\mcb_{tan}\cup\mcb_{O,tr})$ 
are transversal to the lines $\La_j$, quadratic and regular,  and there are at most two of them: $\nu_1+\nu_2\leq2$, by 
(\ref{bi}). Thus, $\Phi$ consists of at most three branches, and at most one of them is singular. Thus, the 
 only possible candidates to be singular points of the curve $\gamma$ are the base point of the unique branch from 
 $\mcb_{tan}\cup\mcb_{O,tr}$ and a point of intersection of  quadratic regular branches  (if it is different from the latter base point). 
Finally, we have at most two singular points. 

Subcase 1.3): $|\mcb_{tan}|+|\mcb_{O,tr}|=2$. Then $\Phi=\mcb_{tan}\cup\mcb_{O,tr}$,  by (\ref{bi}), the number of base points of 
the branches from the collection $\Phi$ is at most 2, and they are the only potential singular points.

Case 2):  $|\mcb_{O,tan}|\geq1$. Then $|\mcb_{O,tan}|=1$, and $\Phi=\mcb_{O,tan}$. This follows from inequality (\ref{bi}) and 
positivity of the sum in $b\in\mcb_{O,tan}$ in its left-hand side. Thus, the set 
$\Phi$ consists of just one branch, and we have at most one singular (or inflection) point. The claim is proved.
 \end{proof}

\begin{theorem} \label{af} \cite[theorem 1.6]{gs}. Let  $\gamma\subset\cp^2$ be an irreducible algebraic curve such that there exists 
a projective line $L$ satisfying the following statements: 

- all the singular and inflection points of the curve $\gamma$ (if any) lie in $L$; 

- each local branch of the curve $\gamma$ at every point of the intersection $\gamma\cap L$ that is transversal to $L$ is subquadratic. 

Then $\gamma$ is a conic.
\end{theorem} 

There exists a line $L$ satisfying the conditions of Theorem \ref{af} for  the curve $\gamma$ under consideration. 
Namely, in the first case of Claim 2 the line $L$ is the line passing though (at most two) singular points of the curve $\gamma$. 
In the second case we choose $L$ to be the tangent line to the unique local branch at the unique special point. 
This together with Theorem \ref{af} implies that $\gamma$ is a conic.  Theorem \ref{thalg2} is proved.

\subsection{Proof of Theorem \ref{thalg2}: case, when $\ii$ is a regular conic}

Let $\ii\subset\cp^2$ be a regular conic, and let $\gamma\subset\cp^2$ be an irreducible algebraic curve, 
$\gamma\neq\ii$, $d=\deg\gamma$. Let $\mcb_{tr}$, 
$\mcb_{tan}$ denote respectively the set of those local branches of the curve $\gamma$ at base points in 
$\gamma\cap\ii$ 
that are transversal (respectively, tangent) to $\ii$. Let $|\mcb_{tr}|$, $|\mcb_{tan}|$ denote their cardinalities.

The proof of Theorem \ref{thalg2} in the case under consideration is based on the following inequality. 

\begin{proposition} Let $\ii$, $\gamma$, $d$ be as above. Let each local branch in $\mcb_{tan}$ be 
quadratic, and each branch in $\mcb_{tr}$ be regular. Then 
\begin{equation}\frac12|\mcb_{tr}|+\sum_{b\in\mcb_{tan}}s(b)\leq d.\label{led}\end{equation}
\end{proposition}

\begin{proof} The intersection index of the curves $\gamma$ and $\ii$ equals to $2d$ (B\'ezout Theorem). 
On the other hand, it equals to the sum of intersection indices of the 
conic $\ii$ with the local branches from the collections $\mcb_{tr}$ and $\mcb_{tan}$. Each branch in $\mcb_{tr}$ 
has intersection index one with $\ii$, since it is regular and transversal to $\ii$, by 
assumptions. Each branch $b\in\mcb_{tan}$ has 
intersection index at least $2s(b)$ with $\ii$. Indeed, $b$ is quadratic,  as is the branch of the conic $\ii$ at the same base point. 
Therefore, applying coordinate change rectifying the germ of the conic $\ii$ transforms $b$ to a branch $\wt b$ with the same local 
degree $s(\wt b)=s(b)$ and Puiseux exponent $r\geq2$. 
The intersection index of the branch $b$ and the conic $\ii$ equals to the intersection index of the branch $\wt b$ with its tangent 
line at the base point, that is, $rs(\wt b)=rs(b)\geq2s(b)$. Finally, 
$2d\geq|\mcb_{tr}|+2\sum_{b\in\mcb_{tan}}s(b)$. This proves (\ref{led}). 
\end{proof}

Now let us prove Theorem \ref{thalg2}. Let $\gamma$ be a curve, as in Theorem \ref{thalg2}. 
Recall that all the singular and inflection points of the curve $\gamma$ (if any) lie in the conic $\ii$, 
and its local branches in $\mcb_{tan}$ ($\mcb_{tr}$) are quadratic (respectively, quadratic and regular). 
Let us calculate their contributions  to the right-hand side of inequality (\ref{ind0}) and substitute inequality (\ref{led}). The second sum 
in the right-hand side in (\ref{ind0}) vanishes, by quadraticity. The contribution of each branch $b\in\mcb_{tr}$ to the first sum 
also vanishes, since $s(b)=1$. The total contribution of the branches from the collection $\mcb_{tan}$ to the first sum 
equals to $\sum_{b\in\mcb_{tan}}s(b)-|\mcb_{tan}|$.  
This together with (\ref{ind0}) implies that 
$$d-2\leq \sum_{b\in\mcb_{tan}}s(b)-|\mcb_{tan}|.$$
The latter right-hand side is no greater than $d-\frac12|\mcb_{tr}|-|\mcb_{tan}|$, by (\ref{led}). Therefore,
\begin{equation} \frac12|\mcb_{tr}|+|\mcb_{tan}|\leq 2.\label{inbt}\end{equation}
 Let us show that this together with  Theorem \ref{af} implies that   $\gamma$ is a conic.

Inequality (\ref{inbt}) implies that  the following three cases are possible.

Case 1): $|\mcb_{tr}|\leq4$, $\mcb_{tan}=\emptyset$. Thus, all the local branches of the curve $\gamma$ at its intersection points with 
$\ii$ lie in $\mcb_{tr}$, and hence, they 
are quadratic and regular. A point of intersection $\gamma\cap\ii$ can be singular only in the case, 
when it is a point of intersection of some two of (at most 4)  branches in $\mcb_{tr}$.  
Hence, $\gamma$ has at most two singular points (thus, all of them 
lie in a line), and all the local branches of the curve $\gamma$ at them are quadratic. 
This together with Theorem \ref{af} implies that $\gamma$ is a conic. 

Case 2): $|\mcb_{tan}|=1$, $|\mcb_{tr}|\leq2$. Let $C$ denote the base point of the unique branch in 
$\mcb_{tan}$.  Each point of intersection $\gamma\cap\ii$ distinct from the point $C$ lies in the union of (at most two) branches in 
$\mcb_{tr}$. It is singular, if and only if it is the intersection point of two latter branches. Thus, $\gamma$ has at most two singular 
points, its local branches at them are quadratic, and hence, $\gamma$ is a conic, by Theorem \ref{af}, as in the above case. 

Case 3): $|\mcb_{tan}|=2$, $\mcb_{tr}=\emptyset$. Then $\gamma$ has at most two singular points, and all its branches at them, 
which lie in $\mcb_{tan}$,  are quadratic. Hence, $\gamma$ is a conic, as in Case 1). 
 Theorem \ref{thalg2} is proved.

\section{Proof of main theorems}

\subsection{Rationally integrable $\ii$-angular billiards. Proof of Theorem \ref{conj1}}

Let $\ii\subset\cp^2$ be a conic (regular or a pair of distinct lines), and let 
$\gamma\subset\cp^2$ be an 
irreducible algebraic curve different from a line and from $\ii$ 
and  generating a rationally integrable $\ii$-angular billiard.
\begin{theorem} \label{sis}  (\cite[theorem 1]{bm}, \cite[theorem 1.2]{bm2}). All the  singular and inflection points  (if any) of the curve $\gamma$ lie in  $\ii$. 
\end{theorem}
 \begin{remark} The  above-cited theorems from \cite{bm, bm2} are stated for a polynomially integrable billiard $\Omega$. Namely, 
 for every $C^2$-smooth arc $\alpha\subset\partial\Omega$ with non-zero geodesic curvature 
 the statement  of Theorem \ref{sis} is proved  there 
  for each non-linear irreducible component $\gamma$ of the Zariski closure of the $\Sigma$-dual curve $\alpha^*$. 
  But the proofs given in \cite{bm, bm2}  remain valid in the general context of Theorem \ref{sis}. 
  \end{remark}

Each local branch of the curve $\gamma$ at a base point in $\gamma\cap\ii$ that satisfies the conditions of 
some of the statements (i), (ii-a), or (ii-b) of Theorem \ref{iii} also 
satisfies the corresponding statement, by Theorem \ref{iii}. 
Therefore, $\gamma$ satisfies the conditions of Theorem \ref{thalg2}, by 
Theorem \ref{sis}.  Hence, it is a conic, by Theorem \ref{thalg2}. 
This proves Theorem \ref{conj1}.

\subsection{Confocal billiards. Proof of  Theorem \ref{thal}} 
Let $\Omega\subset\Sigma$ be a polynomially integrable 
billiard with countably piecewise $C^2$-smooth boundary that contains a $C^2$-smooth  arc 
$\alpha$ with non-zero geodesic curvature. Let $\Psi(M)$ be its non-trivial 
homogeneous polynomial integral  of even degree $2n$: $M=[r,v]$, and  
$\Psi([r,v])$ is not a function of the squared norm $||v||^2=<Av,v>$  
in the metric of the surface $\Sigma$. One has $\Psi(M)\not\equiv c<AM,M>^n$, since $<AM,M>=<Av,v>$, by 
Proposition \ref{vectis}. 
 Let $G$ be the corresponding rational 
function (\ref{ratg}): $G\not\equiv const$. The complex Zariski closure of the $\Sigma$-dual curve $\alpha^*$ 
is an algebraic curve that contains at least one nonlinear irreducible component. Each its 
non-linear irreducible component generates a rationally integrable $\ii$-angular billiard with 
integral $G$, by Corollary \ref{iang}. Hence, it is a conic, by Theorem \ref{conj1}. Therefore, 
$\alpha$ contains a non-geodesic conical arc. 
This together with Theorem \ref{bol2} implies that the billiard $\Omega$ is countably confocal and 
proves Theorem \ref{thal}.

\subsection{Case of smooth connected boundary. Proof of Theorem \ref{thal0}} 
Let $\Omega\subset\Sigma$ be a 
polynomially integrable billiard, and let its boundary be $C^2$-smooth, connected and do not lie in a geodesic. 
Then the billiard $\Omega$ is countably confocal, by Theorem \ref{thal}. This means that its boundary $\partial\Omega$ 
contains an open dense subset $R$ that is a disjoint union of open 
arcs of confocal conics and geodesic segments,  including at least one non-geodesic conical arc. Let 
us fix the latter arc and denote it by $c$, and let $\mcc\supset c$ denote the ambient conic. 
Let us show that $\partial\Omega$ coincides either with the whole conic $\mcc$, or with its connected component.  
We consider that $c$ is a maximal arc of the conic $\mcc$ that is 
contained in the $C^2$-smooth one-dimensional submanifold $\partial\Omega\subset\Sigma$. Suppose the 
contrary: $c$ has an endpoint $Q$. The point $Q$ cannot be an accumulation point of the union of geodesic 
segments in $\partial\Omega$, by $C^2$-smoothness and since $\partial\Omega$ 
has non-zero  geodesic curvature at $Q$, as does $\mcc$: it has quadratic tangency at $Q$ with the geodesic 
tangent to $T_Q\partial\Omega$. Therefore, the point $Q$ has a neighborhood $U$ in $\Sigma$ 
such that the intersections $I_U=\partial\Omega\cap U$, $c_U=\mcc\cap U$ are connected, 
$\partial U$ is transversal to 
$\partial\Omega$, and $R\cap U\subset I_U$ consists of arcs of conics confocal to $\mcc$. Their ambient 
conics intersect $U$ by leaves of  an analytic foliation having $c_U$ as a leaf, 
since  each confocal conic pencil is locally given by  a pair of orthogonal foliations and all the conics under question 
are $C^1$-close to $\mcc$. Thus, 
the $C^2$-smooth connected  submanifold $I_U\subset U$ 
contains an open and dense subset $R\cap U$ where it is tangent to the above foliation. 
Therefore, $I_U$ is a leaf of this foliation. The leaves $I_U=\partial\Omega\cap U$ and $c_U$ 
coincide, since both of them contain an arc adjacent to $Q$ of the conic $\mcc$, by construction. 
Finally, a neighborhood $I_U$ of the point $Q$ in $\partial\Omega$ is contained in the conic 
$\mcc$. This contradicts 
maximality of the conical arc $c\subset\partial\Omega$ and proves Theorem \ref{thal0}.
  
\subsection{Proof of complexification: Theorem \ref{talc}}

 The fact that each polynomially integrable complex billiard admits a homogeneous polynomial integral of the form 
  $\Psi(M)$ is proved by a straightforward complexification of Bolotin's proof of the same 
  statement in the real case in 
  \cite{bolotin, bolotin2}. 
  This implies that the  curves $\Gamma_t$ are algebraic, as in loc. cit., and the curves  $\Sigma$-dual to 
  the non-geodesic curves $\Gamma_t$
  generate rationally integrable $\ii$-angular billiards with a common 
  rational integral,  as in  the proofs of \cite[theorem 3]{bm}, \cite[theorem 1.3]{bm2} and  Theorem \ref{ext}.  
  Afterwards  confocality of the billiard is deduced from Theorem \ref{conj1} in the same way, as in Subsection 6.2, 
  by a straightforward complexification of Theorem \ref{bol2} and its proof. In the case, when the billiard contains 
  no admissible complex geodesic of type (\ref{newgeod}), it has a non-trivial integral of degree 2 in $P$, 
  as in \cite[proposition 1]{bolotin2}. Otherwise, if it contains a complex geodesic of type (\ref{newgeod}), 
   it has a non-trivial integral of degree 4  and no non-trivial integral 
  of lower degree; the proof of this statement given in \cite[p. 123]{bolotin2} in the real case remains valid in  
  the complex case without changes.

\section{Acknowledgements}

I am grateful to Misha Bialy and Andrey Mironov for introducing me to polynomially integrable billiards, providing
the fundamental first step (their works  \cite{bm, bm2}) of the proof of the main results of the present paper and helpful discussions. 
Some important parts of the work were done during my visits 
to   Sobolev Institute at Novosibirsk and to Tel Aviv University. I wish to thank Andrey Mironov and
Misha Bialy for their invitations and hospitality and both institutions for their hospitality and support. I wish to thank Andrey  Mironov 
for his hard work  and patience of going through my proofs  and helpful remarks. I with to thank Eugenii Shustin, to whom this work 
is very  much due, for helpful discussions. Some of the main arguments in the proof, namely, the curve invariant arguments in Section 5 are  
a modified version of Shustin's arguments from our paper \cite[section 4]{gs}. I wish to thank Anatoly  Fomenko and Elena  Kudryavtseva for helpful discussions and for convincing me to extend the results to piecewise smooth case. I wish to thank Sergei Bolotin, Vladimir Dragovi\'c, Etienne Ghys, Jean-Pierre Marco, Sergei Tabachnikov, Dmitry Treshchev and Alexander Veselov  for helpful discussions. I wish to thank the referee for his hard work of going through the proofs and helpful remarks.

\end{document}